\documentclass{article}
\usepackage{mathrsfs, amsmath, amsthm}
\usepackage{graphicx, color}
\usepackage{caption, subcaption}
\usepackage{array}
\usepackage{cases}
\makeatletter
\def\th@plain{%
  \upshape 
}
\makeatother

\makeatletter
\renewenvironment{proof}[1][\proofname]{\par
  \pushQED{\qed}%
  \normalfont \topsep6\p@\@plus6\p@\relax
  \trivlist
  \item[\hskip\labelsep
        \bfseries
    #1\@addpunct{.}]\ignorespaces
}{%
  \popQED\endtrivlist\@endpefalse
}
\makeatother

\newtheorem{theorem}{Theorem}
\numberwithin{theorem}{section}
\newtheorem{lemma}{Lemma}

\newtheorem{conjecture}{Conjecture}
\newtheorem*{conjecture*}{Conjecture}
\newtheorem{case}{Case}
\newtheorem{subcase}{Subcase}[case]
\newtheorem{subsubcase}{Subcase}[subcase]
\newtheorem{subsubsubcase}{Subcase}[subsubcase]
\newtheorem{claim}{Claim}
\newtheorem{fact}{Fact}

\theoremstyle{definition}

\usepackage[left=25mm,top=25mm,bottom=25mm,right=25mm]{geometry}
\usepackage[T1]{fontenc}
\usepackage[varg]{txfonts}

\usepackage{enumitem}
\usepackage[square, numbers, sort&compress]{natbib}

\ifx\pdfoutput\undefined
 \usepackage[dvipdfm,%
  bookmarks=true,%
  bookmarksnumbered=true, 
  bookmarksopen=true, 
  plainpages=false,%
  pdfpagelabels,%
  colorlinks=true, 
  linkcolor=blue, 
  citecolor=blue,%
  hyperindex=true,
  urlcolor=black]{hyperref}
\else
 \usepackage[pdftex,%
  bookmarks=true,%
  bookmarksnumbered=true, 
  bookmarksopen=true, 
  plainpages=false,%
  pdfpagelabels,%
  colorlinks=true, 
  linkcolor=blue, 
  citecolor=blue,%
  anchorcolor=green,
  urlcolor= blue,
  breaklinks=true,
  hyperindex=true]{hyperref}
\fi

\renewcommand{\figurename}{Fig.}

\newcounter{Hcase}
\newcounter{Hclaim}

\newcommand{\resetcounter}{\stepcounter{Hcase}\setcounter{case}{0}\stepcounter{Hclaim}\setcounter{claim}{0}}

\newcommand{\etal}{et~al.\ }

\def\int(#1){\mathrm{int}(#1)}
\def\ext(#1){\mathrm{ext}(#1)}
\def\Int(#1){\mathrm{Int}(#1)}
\def\Ext(#1){\mathrm{Ext}(#1)}
\def\ad(#1){\mathrm{ad}(#1)}
\def\mad(#1){\mathrm{mad}(#1)}
\def\la(#1){\mathrm{la}(#1)}

\newcommand{\mul}{\mathrm{mul}}
\begin{document}
\title{Further result on acyclic chromatic index of planar graphs}
\author{Tao Wang\,\textsuperscript{a, b, }\footnote{{\tt Corresponding
author: wangtao@henu.edu.cn} }\ \ \ Yaqiong Zhang\,\textsuperscript{b}\\
{\small \textsuperscript{a}Institute of Applied Mathematics}\\
{\small Henan University, Kaifeng, 475004, P. R. China}\\
{\small \textsuperscript{b}College of Mathematics and Information Science}\\
{\small Henan University, Kaifeng, 475004, P. R. China}}
\date{July 1, 2015}
\maketitle
\begin{abstract}
An acyclic edge coloring of a graph $G$ is a proper edge coloring such that every cycle is colored with at least three colors. The acyclic chromatic index $\chiup_{a}'(G)$ of a graph $G$ is the least number of colors in an acyclic edge coloring of $G$. It was conjectured that $\chiup'_{a}(G)\leq \Delta(G) + 2$ for any simple graph $G$ with maximum degree $\Delta(G)$. In this paper, we prove that every planar graph $G$ admits an acyclic edge coloring with $\Delta(G) + 6$ colors.

{\bf Keywords:} Acyclic edge coloring; Acyclic chromatic index; $\kappa$-deletion-minimal graph; $\kappa$-minimal graph; Acyclic edge coloring conjecture
\end{abstract}
\section{Introduction}

All graphs considered in this paper are finite, simple and undirected. An acyclic edge coloring of a graph $G$ is a proper edge coloring such that every cycle is colored with at least three colors. The acyclic chromatic index $\chiup_{a}'(G)$ of a graph $G$ is the least number of colors in an acyclic edge coloring of $G$. It is obvious that $\chiup'_{a}(G)\geq \chiup'(G)\geq \Delta(G)$. Fiam\v{c}{\'i}k \cite{MR526851} stated the following conjecture in 1978, which is well known as Acyclic Edge Coloring Conjecture, and Alon \etal \cite{MR1837021} restated it in 2001.

\begin{conjecture}\label{graph}
For any graph $G$, $\chiup'_{a}(G)\leq \Delta(G)+2$.
\end{conjecture}

Alon \etal \cite{MR1109695} proved that $\chiup'_{a}(G) \leq 64 \Delta(G)$ for any graph $G$ by using probabilistic method. Molloy and Reed \cite{MR1715600} improved it to $\chiup'_{a}(G) \leq 16 \Delta(G)$. Recently, Ndreca \etal \cite{MR2864444} improved the upper bound to $\lceil 9.62(\Delta(G)-1)\rceil$, and Esperet and Parreau \cite{MR3037985} further improved it to $4\Delta(G) - 4$ by using the so-called entropy compression method. The best known general bound is $\lceil 3.74(\Delta(G)-1)\rceil$ due to Giotis \etal \cite{2014arXiv1407.5374G}. Alon \etal \cite{MR1837021} proved that there is a constant $c$ such that $\chiup'_{a}(G)\leq \Delta(G) + 2$ for a graph $G$ whenever the girth is at least $c\Delta \log\Delta$.

Regarding general planar graph $G$, Fiedorowicz \etal \cite{MR2458434} proved that $\chiup'_{a}(G)\leq 2\Delta(G)+29$; Hou \etal \cite{MR2491777} proved that $\chiup'_{a}(G)\leq \max\{2\Delta(G)-2, \Delta(G)+22\}$. Recently, Basavaraju \etal \cite{MR2817509} showed that $\chiup_{a}'(G) \leq \Delta(G) + 12$, and Guan \etal \cite{MR3031510} improved it to $\chiup_{a}'(G) \leq \Delta(G) + 10$, and Wang \etal \cite{MR2994403} further improved it to $\chiup_{a}'(G) \leq \Delta(G) + 7$.

In this paper, we improve the upper bound to $\Delta(G) + 6$ by the following theorem.
\begin{theorem}\label{MResult}
If $G$ is a planar graph, then $\chiup'_{a}(G)\leq \Delta(G)+6$.
\end{theorem}

\section{Preliminary}
Let $\mathbb{S}$ be a multiset and $x$ be an element in $\mathbb{S}$. The {\em multiplicity} $\mul_{\mathbb{S}}(x)$ is the number of times $x$ appears in $\mathbb{S}$. Let $\mathbb{S}$ and $\mathbb{T}$ be two multisets. The union of $\mathbb{S}$ and $\mathbb{T}$, denoted by $\mathbb{S} \uplus \mathbb{T}$, is a multiset with $\mul_{\mathbb{S} \uplus \mathbb{T}}(x) = \mul_{\mathbb{S}}(x) + \mul_{\mathbb{T}}(x)$. Throughout this paper, every coloring uses colors from $[\kappa] = \{1, 2, \dots, \kappa\}$. 

We use $V(G)$, $E(G)$, $\delta(G)$ and $\Delta(G)$ to denote the vertex set, the edge set, the minimum degree and the maximum degree of a graph $G$, respectively. For a vertex $v\in V(G)$, $N_{G}(v)$ denotes the set of vertices that are adjacent to $v$ in $G$ and $\deg_{G}(v)$ (or simple $\deg(v)$) to denote the degree of $v$ in $G$. When $G$ is a plane graph, we use $F(G)$ to denote its face set and $\deg_{G}(f)$ (or simple $\deg(f)$) to denote the degree of a face $f$ in $G$. A $k$-, $k^{+}$-, $k^{-}$-vertex (resp. face) is a vertex (resp. face) with degree $k$, at least $k$ and at most $k$, respectively. A face $f = v_{1}v_{2} \dots v_{k}$ is a $(\deg(v_{1}), \deg(v_{2}), \dots, \deg(v_{k}))$-face. 

A graph $G$ with maximum degree at most $\kappa$ is {\em $\kappa$-deletion-minimal} if $\chiup_{a}'(G) > \kappa$ and $\chiup_{a}'(H) \leq \kappa$ for every proper subgraph $H$ of $G$. A graph property $\mathcal{P}$ is {\em deletion-closed} if $\mathcal{P}$ is closed under taking subgraphs. Analogously, we can define another type of minimal graphs by taking minors. A graph $G$ with maximum degree at most $\kappa$ is {\em $\kappa$-minimal} if $\chiup_{a}'(G) > \kappa$ and $\chiup_{a}'(H) \leq \kappa$ for every proper minor $H$ with $\Delta(H) \leq \Delta(G)$. Obviously, every proper subgraph of a $\kappa$-minimal graph admits an acyclic edge coloring with at most $\kappa$ colors, and then every $\kappa$-minimal graph is also a $\kappa$-deletion-minimal graph and all the properties of $\kappa$-deletion-minimal graphs are also true for $\kappa$-minimal graphs.

Let $G$ be a graph and $H$ be a subgraph of $G$. An acyclic edge coloring of $H$ is a {\em partial acyclic edge coloring} of $G$. Let $\mathcal{U}_{\phi}(v)$ denote the set of colors which are assigned to the edges incident with $v$ with respect to $\phi$. Let $C_{\phi}(v) = [\kappa] \setminus \mathcal{U}_{\phi}(v)$ and $C_{\phi}(uv) = [\kappa] \setminus (\mathcal{U}_{\phi}(u) \cup \mathcal{U}_{\phi}(v))$. Let $\Upsilon_{\phi}(uv) = \mathcal{U}_{\phi}(v) \setminus \{\phi(uv)\}$ and $W_{\phi}(uv) = \{\,u_{i} \mid uu_{i} \in E(G) \mbox{ and } \phi(uu_{i}) \in \Upsilon_{\phi}(uv)\,\}$. Notice that $W_{\phi}(uv)$ may be not same with $W_{\phi}(vu)$. For simplicity, we will omit the subscripts if no confusion can arise. 

An $(\alpha, \beta)$-maximal dichromatic path with respect to $\phi$ is a maximal path whose edges are colored by $\alpha$ and $\beta$ alternately. An $(\alpha, \beta, u, v)$-critical path with respect to $\phi$ is an $(\alpha, \beta)$-maximal dichromatic path which starts at $u$ with color $\alpha$ and ends at $v$ with color $\alpha$. An $(\alpha, \beta, u, v)$-alternating path with respect to $\phi$ is an $(\alpha, \beta)$-dichromatic path starting at $u$ with color $\alpha$ and ending at $v$ with color $\beta$.

Let $\phi$ be a partial acyclic edge coloring of $G$. A color $\alpha$ is {\em candidate} for an edge $e$ in $G$ with respect to a partial edge coloring of
$G$ if none of the adjacent edges of $e$ is colored with $\alpha$.  A candidate color $\alpha$ is {\em valid} for an edge $e$ if assigning the color $\alpha$ to $e$ does not result in any dichromatic cycle in $G$.

\begin{fact}[Basavaraju \etal \cite{MR2817509}]\label{Fact1}%
Given a partial acyclic edge coloring of $G$ and two colors $\alpha, \beta$, there exists at most one $(\alpha, \beta)$-maximal dichromatic path containing a particular vertex $v$. \qed
\end{fact}

\begin{fact}[Basavaraju \etal \cite{MR2817509}]\label{Fact2}%
Let $G$ be a $\kappa$-deletion-minimal graph and $uv$ be an edge of $G$. If $\phi$ is an acyclic edge coloring of $G - uv$, then no candidate color for $uv$ is valid. Furthermore, if $\mathcal{U}(u) \cap \mathcal{U}(v) = \emptyset$, then $\deg(u) + \deg(v) = \kappa + 2$; if $|\mathcal{U}(u) \cap \mathcal{U}(v)| = s$, then $\deg(u) + \deg(v) + \sum\limits_{w \in W(uv)} \deg(w) \geq \kappa+2s+2$. \qed
\end{fact}

We remind the readers that we will use these two facts frequently, so please keep these in mind and we will not refer it at every time. 

\section{Structural lemmas}
Wang and Zhang \cite{MR3166127} presented many structural results on $\kappa$-deletion-minimal graphs and $\kappa$-minimal graphs. In this section, we give more structural lemmas in order to prove our main result.
\begin{lemma}\label{delta2}
If $G$ is a $\kappa$-deletion-minimal graph, then $G$ is $2$-connected and $\delta(G) \geq 2$.
\end{lemma}
\subsection{Local structure on the $2$- or $3$-vertices}

\begin{lemma}[Wang and Zhang \cite{MR3166127}]\label{2InTriangle}%
Let $G$ be a $\kappa$-minimal graph with $\kappa \geq \Delta(G) + 1$. If $v_{0}$ is a $2$-vertex of $G$, then $v_{0}$ is contained in a triangle.
\end{lemma}

\begin{lemma}[Wang and Zhang \cite{MR3166127}]\label{2+edge}%
Let $G$ be a $\kappa$-deletion-minimal graph. If $v$ is adjacent to a $2$-vertex $v_{0}$ and $N_{G}(v_{0}) = \{w, v\}$, then $v$ is adjacent to at least $\kappa - \deg(w) + 1$ vertices with degree at least $\kappa - \deg(v) + 2$. Moreover,
\begin{enumerate}[label=(\Alph*)]%
\item\label{A} if $\kappa \geq \deg(v) + 1$ and $wv \in E(G)$, then $v$ is adjacent to at least $\kappa - \deg(w) + 2$ vertices with degree at least $\kappa - \deg(v) + 2$, and $\deg(v) \geq \kappa - \deg(w) + 3$;
\item\label{B} if $\kappa \geq \Delta(G) + 2$ and $v$ is adjacent to precisely $\kappa - \Delta(G) + 1$ vertices with degree at least $\kappa - \Delta(G) + 2$, then $v$ is adjacent to at most $\deg(v) + \Delta(G) - \kappa - 3$ vertices with degree two and $\deg(v) \geq \kappa - \Delta(G) + 4$.
\end{enumerate}
\end{lemma}

\begin{lemma}[Wang and Zhang \cite{MR3166127}]\label{2++edge}%
Let $G$ be a $\kappa$-deletion-minimal graph with $\kappa \geq \Delta(G) + 2$. If $v_{0}$ is a $2$-vertex, then every neighbor of $v_{0}$ has degree at least $\kappa - \Delta(G) + 4$.
\end{lemma}

\begin{lemma}[Hou \etal \cite{MR2849391}]\label{3+vertex}%
Let $G$ be a $\kappa$-deletion-minimal graph with $\kappa \geq \Delta(G) + 2$. If $v$ is a $3$-vertex, then every neighbor of $v$ is a $(\kappa - \Delta(G) + 2)^{+}$-vertex.
\end{lemma}

\begin{lemma}[Wang and Zhang \cite{MR3166127}]\label{3++edge}
Let $G$ be a $\kappa$-minimal graph with $\kappa \geq \Delta(G) + 2$. If $v$ is a $3$-vertex in $G$, then every neighbor of $v$ is a $(\kappa-\Delta(G) + 3)^{+}$-vertex.
\end{lemma}

\begin{lemma}[Wang and Zhang \cite{MR3166127}]\label{L9}%
Let $G$ be a $\kappa$-deletion-minimal graph with $\kappa \geq \Delta(G) + 2$, and let $w_{0}$ be a $3$-vertex with $N_{G}(w_{0}) = \{w, w_{1}, w_{2}\}$, and $\deg(w) = \kappa - \Delta(G) + 3$. If $ww_{1}, ww_{2} \in E(G)$, then $\deg(w_{1}) = \deg(w_{2}) = \Delta(G)$ and $w$ is adjacent to precisely one vertex (namely $w_{0}$) with degree less than $\Delta(G) - 1$.
\end{lemma}

\begin{lemma}\label{3-10vertex}
Let $G$ be a $\kappa$-deletion-minimal graph with maximum degree $\Delta$, and let $w_{0}$ be a $3$-vertex with $N_{G}(w_{0}) = \{w, w_{1}, w_{2}\}$. If $\deg_{G}(w) = \kappa - \Delta + 4 = \ell$ with $8 \leq \ell \leq 10$ and $N_{G}(w) = \{w_{0}, w_{1}, w_{2}, \dots, w_{\ell - 1}\}$, then there exists no $4$-set $X^{*} \subseteq \{w_{1}, w_{2}, \dots, w_{\ell - 1}\}$ satisfying the following four conditions: (1) every vertex in $X^{*}$ is a $5^{-}$-vertex; (2) the degree-sum of vertices in $X^{*}$ is at most $\kappa - \Delta + 9$; (3) the degree-sum of any two vertices in $X^{*}$ is at most $\Delta$; (4) $X^{*}$ has at least two $4^{-}$-vertices. 
\end{lemma}
\begin{proof}%
Suppose to the contrary that there exists a $4$-set $X^{*}$ satisfying all the four conditions. Let $X$ be the subscripts of vertices in $X^{*}$. Since $G$ is $\kappa$-deletion-minimal, the graph $G - ww_{0}$ has an acyclic edge coloring $\phi$ with $\phi(ww_{i}) = i$ for $i \in \{1, \dots, \ell - 1\}$. The fact that $\deg_{G}(w) + \deg_{G}(w_{0}) \leq \Delta + 3 < \kappa + 2$ and Fact~\ref{Fact2} imply that $\mathcal{U}(w)\cap \mathcal{U}(w_{0}) \neq \emptyset$.

\begin{case}%
$|\mathcal{U}(w)\cap \mathcal{U}(w_{0})| = 1$.
\end{case}
It follows that $|C(ww_{0})| = \Delta - 4$.
\begin{subcase}%
The common color is on $ww_{1}$ or $ww_{2}$.
\end{subcase}
Without loss of generality, we may assume that $w_{0}w_{1}$ is colored with $\ell$ and $w_{0}w_{2}$ is colored with $1$. Note that there exists a $(1, \alpha, w, w_{0})$-critical path for every $\alpha \in \{\ell + 1, \dots, \kappa \}$, so we have that $\{\ell + 1, \dots, \kappa \} \subseteq \mathcal{U}(w_{1}) \cap \mathcal{U}(w_{2})$. Notice that the set  $\{1, \dots, \ell\} \setminus (\mathcal{U}(w_{1}) \cup \mathcal{U}(w_{2}))$ is nonempty. Now, reassigning $\ell$ to $ww_{0}$ and  a color in $\{1, \dots, \ell\} \setminus (\mathcal{U}(w_{1}) \cup \mathcal{U}(w_{2}))$ to $w_{0}w_{1}$ results in an acyclic edge coloring of $G$, a contradiction.
\begin{subcase}%
The common color is not on $ww_{1}$ and $ww_{2}$.
\end{subcase}
Without loss of generality, we may assume that $w_{0}w_{1}$ is colored with $\ell$ and $w_{0}w_{2}$ is colored with $3$. There exists a $(3, \alpha, w, w_{0})$-critical path for $\alpha \in \{\ell + 1, \dots, \kappa\}$. It follows that $\{\ell + 1, \dots, \kappa \} \subseteq \Upsilon(ww_{3}) \cap \Upsilon(w_{0}w_{2})$ and $\deg_{G}(w_{3}) \geq \Delta - 3 \geq 5$.

If $1 \notin \mathcal{U}(w_{2})$, then reassigning $1$ to $w_{0}w_{2}$ will take us back to Case~1.1. Hence, we have that $1 \in \Upsilon(w_{0}w_{2})$  and $\deg_{G}(w_{2}) \geq \Delta - 1 \geq 7$. By \autoref{3+vertex}, we have that $\deg_{G}(w_{1}) \geq \kappa - \Delta + 2 \geq 6$. 

Note that $w_{1}, w_{2}$ and $w_{3}$ are $5^{+}$-vertices, there exists a $4^{-}$-vertex $w_{x}$ with $x \in X \setminus \mathcal{U}(w_{2})$. If $\ell \notin \mathcal{U}(w_{2})$, then reassigning the color $x$ to $w_{0}w_{2}$ results in a new acyclic edge coloring $\sigma$ of $G - ww_{0}$, and then $C_{\sigma}(ww_{0}) = \{\ell + 1, \dots, \kappa \} \subseteq \Upsilon(ww_{x})$ and $\deg_{G}(w_{x}) \geq \Delta - 3 \geq 5$, which contradicts that $w_{x}$ is a $4^{-}$-vertex. Hence, $\Upsilon(w_{0}w_{2}) =  \{1, 2\} \cup \{\ell, \dots, \kappa\}$ and $\deg_{G}(w_{2}) = \Delta$, which implies that $X \cap \Upsilon(w_{0}w_{2}) = \emptyset$. 

\begin{claim}\label{MCL}
There exists a $(3, \ell, w, w_{2})$-alternating path. 
\end{claim}
\begin{proof}
Suppose to the contrary that there exists no $(3, \ell, w, w_{2})$-alternating path. We can revise $\phi$ by assigning $\ell$ to $ww_{0}$ and erase the color from $w_{0}w_{1}$, and obtain an acyclic edge coloring of $G - w_{0}w_{1}$. If some color $\alpha \in \{\ell + 1, \dots, \kappa\}$ is absent in $\mathcal{U}_{\phi}(w_{1})$, then we can further assign $\alpha$ to $w_{0}w_{1}$, since there exists a $(3, \alpha, w, w_{0})$-critical path with respect to $\phi$. If some color $\alpha \in \{4, \dots, \ell - 1\}$ is absent in $\mathcal{U}_{\phi}(w_{1})$, then we can further assign $\alpha$ to $w_{0}w_{1}$. Hence, $\mathcal{U}_{\phi}(w_{1}) \supseteq \{1\} \cup \{4, \dots, \kappa\}$ and $\deg_{G}(w_{1}) \geq \kappa - 2 > \Delta(G)$, a contradiction. 
\end{proof}

Therefore, $\{\ell, \dots, \kappa\} \subseteq \Upsilon(ww_{3})$ and $\deg_{G}(w_{3}) \geq \Delta - 2 \geq 6$, which implies that $X \cap \mathcal{U}(w_{2}) = \emptyset$. 

There exists a $(\ell, m, w_{0}, w_{2})$-critical path for every $m \in X$; otherwise, reassigning $m$ to $w_{0}w_{2}$ results in another new acyclic edge coloring $\phi_{m}$ of $G - ww_{0}$, by the above arguments, $\{\ell, \dots, \kappa \} \subseteq \Upsilon(ww_{m})$ and $\deg_{G}(w_{m}) \geq \Delta - 2 \geq 6$, a contradiction. Thus, we have that $X \subseteq \Upsilon(w_{0}w_{1})$. By symmetry, we may assume that $\{4, 5, 6, 7\} = X \subseteq \Upsilon(w_{0}w_{1})$.

Suppose that $\{3, 8, 9, \dots, \ell - 1\} \nsubseteq \mathcal{U}(w_{1})$, say $\lambda$ is a such color. There exists a $(\lambda, \alpha, w, w_{2})$-alternating path for $\ell + 1 \leq \alpha \leq \kappa$; otherwise, reassigning $\lambda$ to $w_{0}w_{2}$ (if $\lambda = 3$ there is no change to $w_{0}w_{2}$) and $\alpha$ to $ww_{0}$ results in an acyclic edge coloring of $G$. Similar to \autoref{MCL}, there exists a $(\lambda, \ell, w, w_{2})$-alternating path. Reassigning $\lambda$ to $w_{0}w_{1}$ and $4$ to $w_{0}w_{2}$ results in a new acyclic edge coloring $\varphi$ of $G - ww_{0}$. Since there is no $(\lambda, \alpha, w, w_{0})$-critical path with respect to $\varphi$, thus there exists a $(4, \alpha, w_{0}, w)$-critical path with respect to $\varphi$ for $\alpha \in \{\ell, \dots, \kappa\}$, and then $\{\ell, \dots, \kappa\} \subseteq \Upsilon(ww_{4})$, which contradicts the fact that $w_{4}$ is a $5^{-}$-vertex. Hence, we have that $\{1\} \cup \{3, 4, \dots, \ell\} \subseteq \mathcal{U}(w_{1})$. 

Let $\varphi_{m}$ be obtained from $\phi$ by reassigning $m$ to $ww_{0}$ and erasing the color on $ww_{m}$, where $m \in \{4, 5, 6, 7\}$. Note that $\varphi_{m}$ is an acyclic edge coloring of $G - ww_{m}$ for $m \in \{4, 5, 6, 7\}$. By Fact~\ref{Fact2}, we have that $|\Upsilon(ww_{m}) \cap \{1, 2, \dots, \ell - 1\}| \geq 1$ for $m \in \{4, 5, 6, 7\}$.

Let $\alpha$ be an arbitrary color in $\{\ell, \dots, \kappa\} \setminus (\Upsilon(w_{0}w_{1}) \cup \Upsilon(ww_{4}) \cup \Upsilon(ww_{5}) \cup \Upsilon(ww_{6}) \cup \Upsilon(ww_{7}))$. Since there exists neither $(1, \alpha, w, w_{x})$-critical path nor $(3, \alpha, w, w_{x})$-critical path (with respect to $\varphi_{x}$) for every $x \in X$, thus there exists a $(\lambda_{x}, \alpha, w, w_{x})$-critical path (with respect to $\varphi_{x}$), where $\lambda_{x} \in \{2, 8, 9, \dots, \ell - 1\}$. Moreover, there exists $(\lambda, \alpha, w, w_{x_{1}})$- and $(\lambda, \alpha, w, w_{x_{2}})$-critical path for some $\lambda \in \{2, 8, 9, \dots, \ell - 1\}$ since $|X| > |\{2, 8, 9, \dots, \ell - 1\}|$, but this contradicts Fact~\ref{Fact1}.

So we may assume that $\alpha \in \Upsilon(ww_{4}) \cup \Upsilon(ww_{5}) \cup \Upsilon(ww_{6}) \cup \Upsilon(ww_{7})$ for every $\alpha \in \{\ell, \dots, \kappa\} \setminus \Upsilon(w_{0}w_{1})$.

\begin{align*}
\kappa - \Delta + 9 \geq&\deg_{G}(w_{4}) + \deg_{G}(w_{5}) + \deg_{G}(w_{6}) + \deg_{G}(w_{7})\\
\geq & |\{\ell, \dots, \kappa\} \setminus \Upsilon(w_{0}w_{1})| + 4 + \sum_{t = 4}^{7}|\Upsilon(ww_{t}) \cap \{1, \dots, \ell - 1\}|\\
\geq & (\kappa - \Delta) + 4 + (1 + 1 + 1 + 1)\\
=& \kappa - \Delta + 8.
\end{align*}

By symmetry, we may assume that $|\Upsilon(ww_{4}) \cap \{1, \dots, \ell - 1\}| = |\Upsilon(ww_{5}) \cap \{1, \dots, \ell - 1\}| = |\Upsilon(ww_{6}) \cap \{1, \dots, \ell - 1\}| = 1$. Let $\Upsilon(ww_{4}) \cap \{1, \dots, \ell - 1\} = \{\mu_{1}\}$, $\Upsilon(ww_{5}) \cap \{1, \dots, \ell - 1\} = \{\mu_{2}\}$ and $\Upsilon(ww_{6}) \cap \{1, \dots, \ell - 1\} = \{\mu_{3}\}$. If $\mu_{1} = \mu_{2} = \mu$, then there exists a $(\mu, \alpha, w, w_{4})$- and $(\mu, \alpha, w, w_{5})$-critical path, where $\alpha \in \{\ell, \dots, \kappa\} \setminus (\Upsilon(ww_{4}) \cup \Upsilon(ww_{5}))$, which contradicts Fact~\ref{Fact1}. Thus $\mu_{1}, \mu_{2}, \mu_{3}$ are distinct.

If $\mu_{1} \in \{4, 5, 6, 7\}$, then every color $\alpha \in \{\ell, \dots, \kappa\} \setminus (\Upsilon(ww_{4}) \cup \Upsilon(ww_{\mu_{1}}))$ is valid for $ww_{4}$ with respect to $\varphi_{4}$; note that $\{\ell, \dots, \kappa\} \setminus (\Upsilon(ww_{4}) \cup \Upsilon(ww_{\mu_{1}}))$ is a nonempty set. By symmetry, we may assume that $\{\mu_{1}, \mu_{2}, \mu_{3}\} \cap \{4, 5, 6, 7\} = \emptyset$.

Since $\mu_{1}, \mu_{2}, \mu_{3}$ are distinct, we may assume that $\mu_{1} \neq 2$. If $2 \notin \Upsilon(w_{0}w_{1})$, then reassigning $2$ to $w_{0}w_{1}$ and $4$ to $w_{0}w_{2}$ results in a new acyclic edge coloring $\varphi^{*}$ of $G - ww_{0}$. For every color $\beta \in \{\ell, \dots, \kappa\} \setminus \Upsilon(w_{0}w_{1})$, there exists no $(2, \beta, w, w_{0})$-critical path with respect to $\varphi^{*}$, thus there exists a $(4, \beta, w, w_{0})$-critical path with respect to $\varphi^{*}$, and then $\{\ell, \dots, \kappa\} \setminus \Upsilon(w_{0}w_{1}) \subseteq \Upsilon(ww_{4})$ and $\deg_{G}(w_{4}) \geq |\{\ell, \dots, \kappa\} \setminus \Upsilon(w_{0}w_{1})| + 2 \geq 6$, which contradicts the degree of $w_{4}$.

Hence, we have that $\{1, \dots, \ell - 1\} \subseteq \Upsilon(w_{0}w_{1})$ and $|\{\ell, \dots, \kappa\} \setminus \Upsilon(w_{0}w_{1})| \geq \kappa - \Delta + 1$. By similar arguments as above, we can prove that $\Upsilon(ww_{7}) \cap \{1, \dots, \ell - 1\} = \{\mu_{4}\}$ and $\mu_{1}, \mu_{2}, \mu_{3}, \mu_{4}$ are distinct. Moreover, we can also conclude that $\{\mu_{1}, \mu_{2}, \mu_{3}, \mu_{4}\} \cap \{4, 5, 6, 7\} = \emptyset$.

Suppose that $\mu_{1} = 3$. Since there exists no $(3, \alpha, w, w_{4})$-critical path with respect to $\varphi_{4}$, where $\alpha \in \{\ell + 1, \dots, \kappa\}$, thus $\{\ell + 1, \dots, \kappa\} \subseteq \Upsilon(ww_{4})$, a contradiction. So, by symmetry, we may assume that $\{\mu_{1}, \mu_{2}, \mu_{3}, \mu_{4}\} = \{1, 2, 8, 9\}$.

By symmetry, we assume that $\mu_{1} = 1$. Note that there exists no $(1, \alpha, w, w_{4})$-critical path (with respect to $\varphi_{4}$) for every $\alpha \in \{\ell, \dots, \kappa\} \setminus \Upsilon(w_{0}w_{1})$, thus $\{\ell, \dots, \kappa\} \setminus \Upsilon(w_{0}w_{1}) \subseteq \Upsilon(ww_{4})$; otherwise, reassigning $4$ to $ww_{0}$ and a color $\alpha$ to $ww_{4}$ results in an acyclic edge coloring. Now, we have that $\deg_{G}(w_{4}) \geq 2 + |\{\ell, \dots, \kappa\} \setminus \Upsilon(w_{0}w_{1})| \geq 6$, a contradiction.

\begin{case}%
$\mathcal{U}(w)\cap \mathcal{U}(w_{0}) = \{\lambda_{1}, \lambda_{2}\}$, $\phi(w_{0}w_{1}) = \lambda_{1}$ and $\phi(w_{0}w_{2}) = \lambda_{2}$.
\end{case}
If follows that $|C(ww_{0})| = \Delta - 3$. First of all, we show the following claim:
\begin{enumerate}[label = ($\ast$)]
\item $C(ww_{0}) = \{\ell, \dots, \kappa\} \subseteq \mathcal{U}(w_{1}) \cap \mathcal{U}(w_{2})$.
\end{enumerate}
By contradiction and symmetry, assume that there exists a color $\zeta$ in $\{\ell, \dots, \kappa\} \setminus \mathcal{U}(w_{1})$. Clearly, there exists a $(\lambda_{2}, \zeta, w_{0}, w)$-critical path, and then there exists no $(\lambda_{2}, \zeta, w_{0}, w_{1})$-critical path. Now, reassigning $\zeta$ to $w_{0}w_{1}$ will take us back to Case 1. Hence, we have $\{\ell, \dots, \kappa\} \subseteq \mathcal{U}(w_{1})$; similarly, we have $\{\ell, \dots, \kappa\} \subseteq \mathcal{U}(w_{2})$. This completes the proof of ($\ast$).

Note that $w_{1}$ and $w_{2}$ have degree at least $\Delta - 1 \geq 7$, this implies that $\{1, 2\} \cap X = \emptyset$ and $|X \cap \Upsilon(w_{0}w_{1})| \leq 1$ and $|X \cap \Upsilon(w_{0}w_{2})| \leq 1$. Let $\{p, q\} \subseteq X \setminus (\Upsilon(w_{0}w_{1}) \cup \Upsilon(w_{0}w_{2}))$. Reassigning $p$ to $w_{0}w_{1}$ and $q$ to $w_{0}w_{2}$ results in a new acyclic edge coloring $\psi$ of $G - ww_{0}$. Hence, we have that $C_{\psi}(ww_{0}) \subseteq \Upsilon(ww_{p}) \cup \Upsilon(ww_{q})$, and then $\deg_{G}(w_{p}) + \deg_{G}(w_{q}) \geq (\Delta - 3) + 2 + 2 \geq \Delta + 1$, which is a contradiction.
\resetcounter
\end{proof}

\subsection{Local structure on the $4$-vertices}
\begin{lemma}\label{4Sum}%
Let $G$ be a $\kappa$-deletion-minimal graph with maximum degree $\Delta$ and $\kappa \geq \Delta + 2$, and let $w_{0}$ be a $4$-vertex with $N_{G}(w_{0}) = \{w, v_{1}, v_{2}, v_{3}\}$.
\begin{enumerate}[label = (\alph*)]%
\item\label{4Suma} If $\deg_{G}(w) \leq \kappa - \Delta$, then
\begin{equation}\label{EQ1}%
\sum_{x \in N_{G}(w_{0})} \deg_{G}(x) \geq 2\kappa - \deg_{G}(w_{0}) + 8 = 2\kappa + 4.
\end{equation}
\item\label{4Sumb} If $\deg_{G}(w) \leq \kappa - \Delta + 1$ and $ww_{0}$ is contained in two triangles $ww_{1}w_{0}$ and $ww_{2}w_{0}$, then
\begin{equation}\label{EQ2}%
\sum_{x \in N_{G}(w_{0})} \deg_{G}(x) \geq 2\kappa - \deg_{G}(w_{0}) + 9 = 2\kappa + 5.
\end{equation}
Furthermore, if the equality holds in \eqref{EQ2}, then all the other neighbors of $w$ are $6^{+}$-vertices.
\end{enumerate}
\end{lemma}
\begin{proof}%
We may assume that
\begin{enumerate}[label= ($\ast$)]%
\item The graph $G -ww_{0}$ admits an acyclic edge coloring $\phi$ such that the number of common colors at $w$ and $w_{0}$ is minimum.
\end{enumerate}

Here, (a) and (b) will be proved together, so we may assume that $\deg_{G}(w) \leq \kappa - \Delta + 1$. Since $\deg_{G}(w) + \deg_{G}(w_{0}) \leq \kappa - \Delta + 5 < \kappa + 2$, we have that $|\Upsilon(ww_{0}) \cap \Upsilon(w_{0}w)| = m \geq 1$. It follows that $|C(ww_{0})| = \kappa - (\deg_{G}(w) + \deg_{G}(w_{0}) - m - 2) \geq \Delta - 2$. Without loss of generality, let $N_{G}(w) = \{w_{0}, w_{1}, w_{2}, \dots\}$ and $\phi(ww_{i}) = i$ for $1 \leq i \leq \deg_{G}(w)-1$. Let $\mathbb{S} = \Upsilon(w_{0}v_{1}) \uplus \Upsilon(w_{0}v_{2}) \uplus \Upsilon(w_{0}v_{3})$.

\begin{claim}\label{WV-Path}%
For every color $\theta$ in $C(ww_{0})$, there exists a $(\lambda, \theta, w_{0}, w)$-critical path for some $\lambda \in \Upsilon(ww_{0}) \cap \Upsilon(w_{0}w)$. Consequently, every color in $C(ww_{0})$ appears in $\mathbb{S}$.
\end{claim}

\begin{case}%
$\mathcal{U}(w) \cap \mathcal{U}(w_{0}) = \{\lambda\}$.
\end{case}
It follows that $|C(ww_{0})| = \kappa - (\deg_{G}(w) + \deg_{G}(w_{0}) - 3)$.
\begin{enumerate}[label = (\alph*)]%
\item Suppose that $\deg_{G}(w) + \deg_{G}(w_{0}) \leq \kappa - \Delta + 4$. It follows that $|C(ww_{0})| \geq \Delta - 1$. Without loss of generality, let $\phi(w_{0}v_{1}) = 1, \phi(w_{0}v_{2}) = \kappa - \Delta$ and $\phi(w_{0}v_{3}) = \kappa - \Delta + 1$. By \autoref{WV-Path}, there exists a $(1, \theta, w_{0}, w)$-critical path for every $\theta$ in $C(ww_{0})$. Hence, we have that $\deg_{G}(w) = \kappa - \Delta$ and $\deg_{G}(v_{1}) = \deg_{G}(w_{1}) = \Delta$ and $\Upsilon(w_{0}v_{1}) = \Upsilon(ww_{1}) = \{\kappa - \Delta + 2, \dots, \kappa\}$. Notice that $\deg_{G}(w) = \kappa - \Delta \geq 3$ results from \autoref{2++edge}. Reassigning $\kappa - \Delta, 1$ and $2$ to $ww_{1}, ww_{0}$ and $w_{0}v_{1}$ respectively, and we obtain an acyclic edge coloring of $G$, a contradiction.

\item Suppose that $\deg_{G}(w) + \deg_{G}(w_{0}) = \kappa - \Delta + 5$ and $ww_{0}$ is contained in two triangles $ww_{1}w_{0}$ and $ww_{2}w_{0}$ ($w_{1} = v_{1}$ and $w_{2} = v_{2}$).
\end{enumerate}

\begin{subcase}\label{2**}%
The common color $\lambda$ does not appear on $w_{0}v_{3}$, but it appears on $ww_{1}$ or $ww_{2}$.
\end{subcase}
By symmetry, assume that $\phi(w_{0}w_{1}) = 2$, $\phi(w_{0}v_{2}) = \kappa - \Delta + 1$, $\phi(w_{0}v_{3}) = \kappa - \Delta + 2$. By \autoref{WV-Path}, we have that $\{\kappa - \Delta +3, \dots, \kappa\}\subseteq \Upsilon(w_{0}w_{1})\cap \Upsilon(ww_{2})$ and $\deg_{G}(w_{1}) = \deg_{G}(w_{2}) = \Delta$. Now, reassigning $\kappa - \Delta + 1$ to $w_{0}w$ and reassigning $3$ to $w_{0}w_{2}$ results in an acyclic edge coloring of $G$, a contradiction.

\begin{subcase}%
The common color $\lambda$ does not appear on $w_{0}v_{3}$ and it does not appear on $ww_{1}$ or $ww_{2}$ either.
\end{subcase}
 By symmetry, assume that $\phi(w_{0}w_{1}) = 3$, $\phi(w_{0}w_{2}) = \kappa - \Delta + 1$, $\phi(w_{0}v_{3}) = \kappa - \Delta + 2$. By \autoref{WV-Path}, we have that $\{\kappa - \Delta + 3, \dots, \kappa\} \subseteq \Upsilon(w_{0}w_{1})\cap \Upsilon(ww_{3})$, $\deg_{G}(w_{1}) = \Delta$ and $\deg_{G}(w_{3}) \geq \Delta - 1$. Reassigning $2$ to $w_{0}w_{1}$ will take us back to Subcase~\ref{2**}.

\begin{subcase}\label{**2}%
The common color $\lambda$ appears on $w_{0}v_{3}$ and it also appears on $ww_{1}$ or $ww_{2}$.
\end{subcase}
By symmetry, assume that $\phi(w_{0}w_{1}) = \kappa - \Delta + 1$, $\phi(w_{0}w_{2}) = \kappa - \Delta + 2$, $\phi(w_{0}v_{3}) = 2$. By \autoref{WV-Path}, we have that $\{\kappa - \Delta + 3, \dots, \kappa\}\subseteq \Upsilon(ww_{2})\cap \Upsilon(w_{0}v_{3})$, $\deg_{G}(w_{2}) = \Delta$ and $\deg_{G}(v_{3}) \geq \Delta - 1$. Now, reassigning $\kappa - \Delta + 1$ to $ww_{2}$ will take us back to Subcase~\ref{2**}.

\begin{subcase}%
The common color $\lambda$ appears on $w_{0}v_{3}$, but it does not appear on $ww_{1}$ or $ww_{2}$.
\end{subcase}
By symmetry, assume that $\phi(w_{0}w_{1}) = \kappa - \Delta + 1$, $\phi(w_{0}w_{2}) = \kappa - \Delta + 2$, $\phi(w_{0}v_{3}) = 3$. By \autoref{WV-Path}, we have that $\{\kappa - \Delta + 3, \dots, \kappa\} \subseteq \Upsilon(ww_{3})\cap \Upsilon(w_{0}v_{3})$, $\deg_{G}(w_{3}) \geq \Delta - 1$ and $\deg_{G}(v_{3}) \geq \Delta - 1$. If $\{2, \kappa - \Delta + 1\} \cap \Upsilon(w_{0}v_{3}) = \emptyset$, then reassigning $2$ to $w_{0}v_{3}$ will take us back to Subcase~\ref{**2}. So we may assume that  $\{2, \kappa - \Delta + 1\} \cap \Upsilon(w_{0}v_{3}) \neq\emptyset$. But we can still reassign $1$ to $w_{0}v_{3}$ and go back to Subcase~\ref{**2}.

\begin{case}%
$\mathcal{U}(w) \cap \mathcal{U}(w_{0}) = \{\lambda_{1}, \dots, \lambda_{m}\}$ and $m \geq 2$.
\end{case}
Let $\mathcal{A}(v_{1}) = C(ww_{0}) \setminus \Upsilon(w_{0}v_{1}) = \{\alpha_{1}, \alpha_{2}, \dots\}$, $\mathcal{A}(v_{2}) = C(ww_{0}) \setminus \Upsilon(w_{0}v_{2}) =  \{\beta_{1}, \beta_{2}, \dots\}$ and $\mathcal{A}(v_{3}) = C(ww_{0}) \setminus \Upsilon(w_{0}v_{3})$.

\begin{claim}\label{NoemptyA}%
$\mathcal{A}(v_{1}), \mathcal{A}(v_{2}), \mathcal{A}(v_{3}) \neq \emptyset$.
\end{claim}
\begin{proof}
Suppose to the contrary that $\mathcal{A}(v_{*}) = \emptyset$. It follows that $\Delta -1 \geq  |\Upsilon(w_{0}v_{*})| \geq |C(ww_{0})| = \kappa - (\deg_{G}(w) + \deg_{G}(w_{0}) - m - 2) \geq \kappa - (\kappa - \Delta + 5 - 2 - 2) = \Delta - 1$, thus $\deg_{G}(w) + \deg_{G}(w_{0}) = \kappa - \Delta + 5$, $m = 2$ and $\Upsilon(w_{0}v_{*}) = C(ww_{0})$ with $|\Upsilon(w_{0}v_{*})| = \Delta - 1$. This implies that the graph $G$ satisfies the condition (b) with $v_{*} = v_{3}$ (assume that $w_{1} = v_{1}$ and $w_{2} = v_{2}$). We may assume that $\mathcal{U}(w_{0}) = \{\lambda_{1}, \lambda_{2}, \kappa - \Delta + 1\}$.

If the color on $w_{0}w_{1}$ is $\lambda_{1}$ and the color on $w_{0}w_{2}$ is $\lambda_{2}$, then reassigning $\alpha_{1}, \beta_{1}$ and $\lambda_{2}$ to $ww_{0}, w_{0}w_{2}$ and $w_{0}v_{3}$, respectively, yields an acyclic edge coloring of $G$.

But if the color on $w_{0}w_{1}$ is $\kappa - \Delta + 1$ and the color on $w_{0}w_{2}$ is $\lambda_{2}$, then reassigning $2$ to $w_{0}v_{3}$ and $\beta_{1}$ to $ww_{0}$ results in an acyclic edge coloring of $G$. 
\end{proof}

\begin{claim}
Every color in $C(ww_{0})$ appears at least twice in $\mathbb{S}$.
\end{claim}
\begin{proof}%
Suppose that there exists a color $\alpha$ in $C(ww_{0})$ appearing only once in $\mathbb{S}$, say $\alpha \in \Upsilon(w_{0}v_{1})$. Without loss of generality, assume that $\phi(w_{0}v_{1}) = \lambda_{1}$ and $\phi(w_{0}v_{2}) = \lambda_{2}$. By \autoref{WV-Path}, there exists a $(\lambda_{1}, \alpha, w_{0}, w)$-critical path. Reassigning $\alpha$ to $w_{0}v_{2}$ results in a new acyclic edge coloring $\phi^{*}$ of $G - ww_{0}$ with $|\mathcal{U}_{\phi^{*}}(w) \cap \mathcal{U}_{\phi^{*}}(w_{0})| < |\mathcal{U}(w) \cap \mathcal{U}(w_{0})|$, which contradicts the assumption ($\ast$).
\end{proof}

Let $X = \{\,\alpha \mid \alpha \in C(ww_{0}) \mbox{ and } \mul_{\mathbb{S}}(\alpha) = 3\,\}$.
\begin{align*}%
&\sum_{x \in N_{G}(w_{0})} \deg_{G}(x)\\
&= \deg_{G}(w_{0}) + \deg_{G}(w) -1 + \sum_{\alpha \in [\kappa]} \mul_{\mathbb{S}}(\alpha)\\
&= \deg_{G}(w_{0}) + \deg_{G}(w) -1 + \sum_{\alpha \in C(ww_{0})} \mul_{\mathbb{S}}(\alpha) + \sum_{\alpha \in \mathcal{U}(w) \cup \mathcal{U}(w_{0})} \mul_{\mathbb{S}}(\alpha) \\
&= \deg_{G}(w_{0}) + \deg_{G}(w) - 1 + 2|C(ww_{0})| + |X| + \sum_{\alpha \in \mathcal{U}(w) \cup \mathcal{U}(w_{0})} \mul_{\mathbb{S}}(\alpha)\\
&= \deg_{G}(w_{0}) + \deg_{G}(w) - 1 + 2(\kappa - (\deg_{G}(w) + \deg_{G}(w_{0}) - 2 - m)) + |X| + \sum_{\alpha \in \mathcal{U}(w) \cup \mathcal{U}(w_{0})} \mul_{\mathbb{S}}(\alpha)\\
&= 2\kappa - \deg_{G}(w_{0}) - \deg_{G}(w) + 2m + 3 + |X| + \sum_{\alpha \in \mathcal{U}(w) \cup \mathcal{U}(w_{0})} \mul_{\mathbb{S}}(\alpha)
\end{align*}
It is sufficient to prove that

\begin{numcases}{\sum_{\alpha \in \mathcal{U}(w) \cup \mathcal{U}(w_{0})} \mul_{\mathbb{S}}(\alpha) + |X| \geq}
\deg_{G}(w) - 2m + 5, & if $\deg_{G}(w) \leq \kappa - \Delta$;\\
\deg_{G}(w) - 2m + 6, & if $\deg_{G}(w) \leq \kappa - \Delta + 1$ and $ww_{0}$ is contained in two triangles.\label{EQ**}
\end{numcases}

\begin{subcase}%
$\mathcal{U}(w) \cap \mathcal{U}(w_{0}) = \{\lambda_{1}, \lambda_{2}\}$.
\end{subcase}

\begin{claim}\label{VWS}
Every color in $\mathcal{U}(w)$ is in $\mathbb{S}$.
\end{claim}
\begin{proof}%

Assume that $w_{0}v_{1}$ is colored with $\lambda_{1}$ and $w_{0}v_{2}$ is colored with $\lambda_{2}$. Notice that $C(ww_{0}) \subseteq \Upsilon(w_{0}v_{1}) \cup \Upsilon(w_{0}v_{2})$ and $\mathcal{A}(v_{1}) \cap \mathcal{A}(v_{2}) = \emptyset$. By \autoref{NoemptyA}, we have that $\mathcal{A}(v_{1}), \mathcal{A}(v_{2}), \mathcal{A}(v_{3}) \neq \emptyset$. If $\lambda_{1} \notin \mathbb{S}$, then reassigning $\beta_{1}, \alpha_{1}$ and $\lambda_{1}$ to $w_{0}w, w_{0}v_{1}$ and $w_{0}v_{3}$ respectively, results in an acyclic edge coloring of $G$, a contradiction. Thus, we have that $\lambda_{1} \in \mathbb{S}$. Similarly, we can prove that $\lambda_{2} \in \mathbb{S}$. Let $\tau$ be an arbitrary color in $\mathcal{U}(w) \setminus (\mathbb{S} \cup \{\lambda_{1}, \lambda_{2}\})$. Let $\sigma$ be obtained from $\phi$ by reassigning $\tau$ to $w_{0}v_{1}$. It is obvious that $\sigma$ is an acyclic edge coloring of $G - ww_{0}$. So we can obtain a similar contradiction by replacing $\phi$ with $\sigma$. 
\end{proof}

\begin{claim}\label{Claim5}%
The color in $\mathcal{U}(w_{0}) \setminus \{\lambda_{1}, \lambda_{2}\}$ appears at least twice in $\mathbb{S}$.
\end{claim}
\begin{proof}%
Suppose that $\lambda_{1}, \lambda_{2}$ and $\lambda^{*}$ are on the edges $w_{0}v_{1}, w_{0}v_{2}$ and $w_{0}v_{3}$, respectively. There exists a $(\lambda^{*}, \alpha_{1}, w_{0}, v_{1})$-critical path; otherwise, reassigning $\alpha_{1}$ to $w_{0}v_{1}$ will take us back to Case \ref{1Common}. Hence, we have $\lambda^{*} \in \Upsilon(w_{0}v_{1})$. Similarly, there exists a $(\lambda^{*}, \beta_{1}, w_{0}, v_{2})$-critical path and $\lambda^{*} \in \Upsilon(w_{0}v_{2})$. Therefore, the color $\lambda^{*}$ appears exactly twice in $\mathbb{S}$.
\end{proof}

Now, we have
\[
\sum_{\alpha \in \mathcal{U}(w) \cup \mathcal{U}(w_{0})} \mul_{\mathbb{S}}(\alpha) + |X| \geq |\mathcal{U}(w)| + 2 + |X| = \deg_{G}(w) + 1 + |X|.
\]

So conclusion (a) holds. Now, suppose that $\deg_{G}(w) + \deg_{G}(w_{0}) \leq \kappa - \Delta + 5$ and $ww_{0}$ is contained in two triangles $ww_{0}w_{1}$ and $ww_{0}w_{2}$ ($w_{1} = v_{1}$ and $w_{2} = v_{2}$).

\begin{subsubcase}\label{721}%
The two common colors $\lambda_{1}$ and $\lambda_{2}$ are on $w_{1}w$ and $w_{1}w_{0}$.
\end{subsubcase}
There exists a $(\lambda_{1}, \alpha, w_{0}, w)$- or $(\lambda_{2}, \alpha, w_{0}, w)$-critical path for $\alpha \in C(ww_{0})$. Hence, we have that $C(ww_{0}) \subseteq \mathcal{U}(w_{1})$, and thus $\deg_{G}(w_{1}) \geq |C(ww_{0})| + |\{\lambda_{1}, \lambda_{2}\}| \geq \Delta + 1$, a contradiction.

\begin{subsubcase}\label{712}%
The two common colors $\lambda_{1}$ and $\lambda_{2}$ are on $w_{2}w$ and $w_{2}w_{0}$.
\end{subsubcase}

This is similar with Subcase~\ref{721}.

\begin{subsubcase}\label{713}%
$\{\lambda_{1}, \lambda_{2}\} \cap \{1, 2\} = \{\lambda_{1}\}$ and $\lambda_{1}$ appears on $w_{0}w_{1}$ or $w_{0}w_{2}$.
\end{subsubcase}
Without loss of generality, assume that $\phi(w_{0}w_{1})=\kappa - \Delta + 1$, $\phi(w_{0}w_{2})=1$, $\phi(w_{0}v_{3})=3$.
If $2 \notin \Upsilon(w_{0}w_{1}) \cup \Upsilon(w_{0}v_{3})$, then reassigning $2$ to $w_{0}v_{3}$ will take us back to Subcase~\ref{712}. Hence, $2 \in \Upsilon(w_{0}w_{1}) \cup \Upsilon(w_{0}v_{3})$ and $2$ appears at least twice in $\mathbb{S}$. Therefore, we have

\[
\sum_{\alpha \in \mathcal{U}(w) \cup \mathcal{U}(w_{0})} \mul_{\mathbb{S}}(\alpha) + |X| \geq |\mathcal{U}(w)| + 2 + |X| + |\{2\}| \geq \deg_{G}(w) + 2.
\]

Suppose that
\[
\sum_{x \in N_{G}(w_{0})} \deg_{G}(x) = 2\kappa - \deg_{G}(w_{0}) + 9.
\]
It follows that
\[
\sum_{\alpha \in \mathcal{U}(w) \cup \mathcal{U}(w_{0})} \mul_{\mathbb{S}}(\alpha) + |X| = |\mathcal{U}(w)| + 2 + |X| + |\{2\}| = \deg_{G}(w) + 2,
\]
and every color in $\mathcal{U}(w) \setminus \{2\}$ appears only once in $\mathbb{S}$.

There exists a $(3, \kappa - \Delta + 1, w_{0}, w)$-critical path, otherwise, reassigning $\kappa - \Delta + 1$ to $w_{0}w$ and $\alpha_{1}$ to $w_{0}w_{1}$ results in an acyclic edge coloring of $G$, a contradiction. By \autoref{Claim5}, we have that $\kappa - \Delta + 1 \in \Upsilon(w_{0}w_{2}) \cap \Upsilon(w_{0}v_{3})$. And by \autoref{VWS}, we have that $3 \in \Upsilon(w_{0}w_{1}) \cup \Upsilon(w_{0}w_{2})$. Since $|C(ww_{0})| \geq \Delta - 1$ and $\{1, 2, 3, \kappa - \Delta + 1\} \subseteq \Upsilon(w_{0}w_{1}) \cup \Upsilon(w_{0}w_{2})$, this implies that $|\mathcal{A}(w_{1})| + |\mathcal{A}(w_{2})| \geq 4$. There exists no $(1, \alpha, w, w_{0})$-critical path for every $\alpha \in \mathcal{A}(w_{1}) \cup \mathcal{A}(w_{2})$, thus there exists a $(3, \alpha, w, w_{0})$-critical path, and then $\mathcal{A}(w_{1}) \cup \mathcal{A}(w_{2}) \subseteq \Upsilon(ww_{3})$. Hence, $\deg_{G}(w_{3}) \geq |\mathcal{A}(w_{1})| + |\mathcal{A}(w_{2})| + |\{3, \kappa - \Delta + 1\}| \geq 6$.

Suppose that $4 \notin \Upsilon(w_{0}v_{3})$ and there exists no $(\kappa - \Delta + 1, 4, w_{0}, v_{3})$-critical path. Reassigning $4$ to $w_{0}v_{3}$ results in a new acyclic edge coloring $\varrho_{1}$ of $G - ww_{0}$. Similarly, we can prove $\deg_{G}(w_{4}) \geq 6$ by replacing $\phi$ with $\varrho_{1}$.

Suppose that $4 \in \Upsilon(w_{0}v_{3})$. This implies that $\{1, 2, 4, \kappa - \Delta + 1\} \subseteq \Upsilon(w_{0}w_{1}) \cup \Upsilon(w_{0}v_{3})$ and $|\mathcal{A}(w_{1})| + |\mathcal{A}(v_{3})| \geq 4$. Reassigning $4$ to $w_{0}w_{2}$ and reassigning $1$ to $w_{0}v_{3}$ results in another acyclic edge coloring $\pi$ of $G - ww_{0}$. Hence, there exists a $(4, \alpha, w_{0}, w)$-critical path with respect to $\pi$ for $\alpha \in \mathcal{A}(w_{1}) \cup \mathcal{A}(v_{3})$,  and then $\mathcal{A}(w_{1}) \cup \mathcal{A}(v_{3}) \subseteq \Upsilon(ww_{4})$. Similarly as above, there exists a $(4, \kappa - \Delta + 1, w_{0}, w)$-critical path with respect to $\pi$. Hence, $\deg_{G}(w_{4}) \geq |\mathcal{A}(w_{1})| + |\mathcal{A}(v_{3})| + |\{4, \kappa - \Delta + 1\}| \geq 6$.

Suppose that there exists a $(\kappa - \Delta + 1, 4, w_{0}, v_{3})$-critical path and $4 \in \Upsilon(w_{0}w_{1})$. This implies that $\{1, 2, 4, \kappa - \Delta + 1\} \subseteq \Upsilon(w_{0}w_{1}) \cup \Upsilon(w_{0}v_{3})$ and $|\mathcal{A}(w_{1})| + |\mathcal{A}(v_{3})| \geq 4$. Reassigning $4$ to $w_{0}w_{2}$ and reassigning $1$ to $w_{0}v_{3}$ results in another acyclic edge coloring $\varrho_{2}$ of $G - ww_{0}$. Similarly as above, we can prove that $\deg_{G}(w_{4}) \geq 6$.

In one word, the degree of $w_{4}$ is at least six. By symmetry, we have that $\deg_{G}(w_{i}) \geq 6$ for $4 \leq i \leq \deg_{G}(w) - 1$.

\begin{subsubcase}\label{731}%
$\{\lambda_{1}, \lambda_{2}\} \cap \{1, 2\} = \{\lambda_{1}\}$ and $\lambda_{1}$ appears on $w_{0}v_{3}$.
\end{subsubcase}
Without loss of generality, assume that $\phi(w_{0}w_{1}) = \kappa - \Delta + 1$, $\phi(w_{0}w_{2}) = 3$, $\phi(w_{0}v_{3}) = 1$. If $2 \notin \Upsilon(w_{0}w_{1}) \cup \Upsilon(w_{0}v_{3})$, then reassigning $2$ to $w_{0}w_{1}$ and reassigning $\beta_{1}$ to $w_{0}w_{2}$ will take us back to Subcase~\ref{721}. Hence, $2 \in \Upsilon(w_{0}w_{1}) \cup \Upsilon(w_{0}v_{3})$ and $2$ appears at least twice in $\mathbb{S}$. Therefore, we have

\[
\sum_{\alpha \in \mathcal{U}(w) \cup \mathcal{U}(w_{0})} \mul_{\mathbb{S}}(\alpha) + |X| \geq |\mathcal{U}(w)| + 2 + |X| + |\{2\}| \geq \deg_{G}(w) + 2.
\]

Suppose that
\[
\sum_{x \in N_{G}(w_{0})} \deg_{G}(x) = 2\kappa - \deg_{G}(w_{0}) + 9.
\]
It follows that
\[
\sum_{\alpha \in \mathcal{U}(w) \cup \mathcal{U}(w_{0})} \mul_{\mathbb{S}}(\alpha) + |X| = |\mathcal{U}(w)| + 2 + |X| + |\{2\}| = \deg_{G}(w) + 2,
\]
and every color in $\mathcal{U}(w) \setminus \{2\}$ appears only once in $\mathbb{S}$.

There exists a $(3, \kappa - \Delta + 1, w_{0}, w)$-critical path, otherwise, reassigning $\kappa - \Delta + 1$ to $w_{0}w$ and $\alpha_{1}$ to $w_{0}w_{1}$ results in an acyclic edge coloring of $G$, a contradiction. Since $|C(ww_{0})| \geq \Delta - 1$ and $\{1, 2, 3, \kappa - \Delta + 1\} \subseteq \Upsilon(w_{0}w_{1}) \cup \Upsilon(w_{0}v_{3})$, this implies that $|\mathcal{A}(w_{1})| + |\mathcal{A}(v_{3})| \geq 4$. There exists no $(1, \alpha, w, w_{0})$-critical path for every $\alpha \in \mathcal{A}(w_{1}) \cup \mathcal{A}(v_{3})$, thus there exists a $(3, \alpha, w, w_{0})$-critical path, and then $\mathcal{A}(w_{1}) \cup \mathcal{A}(v_{3}) \subseteq \Upsilon(ww_{3})$. Hence, $\deg_{G}(w_{3}) \geq |\mathcal{A}(w_{1})| + |\mathcal{A}(v_{3})| + |\{3, \kappa - \Delta + 1\}| \geq 6$.

Suppose that $4 \notin \Upsilon(w_{0}w_{2})$ and there exists no $(\kappa - \Delta + 1, 4, w_{0}, w_{2})$-critical path. Reassigning $4$ to $w_{0}w_{2}$ results in a new acyclic edge coloring $\varrho_{3}$ of $G - ww_{0}$. Similarly, we can prove $\deg_{G}(w_{4}) \geq 6$ by replacing $\phi$ with $\varrho_{3}$.

If $4 \in \Upsilon(w_{0}w_{2})$, then reassigning $1$ to $w_{0}w_{2}$ and reassigning $4$ to $w_{0}v_{3}$ will take us back to Subcase~\ref{713}. If there exists a $(\kappa - \Delta + 1, 4, w_{0}, w_{2})$-critical path and $4 \in \Upsilon(w_{0}w_{1})$, then reassigning $1$ to $w_{0}w_{2}$ and $4$ to $w_{0}v_{3}$ will take us back to Subcase~\ref{713} again.

Hence, we have that $\deg_{G}(w_{4}) \geq 6$. By symmetry, we also have that $\deg_{G}(w_{i}) \geq 6$ for $4 \leq i \leq \deg_{G}(w) - 1$.
\begin{subsubcase}%
$\{\lambda_{1}, \lambda_{2}\} \cap \{1, 2\} = \emptyset$ and the color on $w_{0}v_{3}$ is a common color.
\end{subsubcase}
Without loss of generality, assume that $\phi(w_{0}w_{1})=\kappa - \Delta + 1$, $\phi(w_{0}w_{2}) = 3$, $\phi(w_{0}v_{3}) = 4$. If $1 \notin \Upsilon(w_{0}w_{2}) \cup \Upsilon(w_{0}v_{3})$, then reassigning $1$ to $w_{0}w_{2}$ will take us back to Subcase~\ref{713}. Hence, $1 \in \Upsilon(w_{0}w_{2}) \cup \Upsilon(w_{0}v_{3})$ and $1$ appears at least twice in $\mathbb{S}$. If $2 \notin \Upsilon(w_{0}w_{1}) \cup \Upsilon(w_{0}v_{3})$, then reassigning $2$ to $w_{0}w_{1}$ and $\beta_{1}$ to $w_{0}w_{2}$ will take us back to Subcase~\ref{713}. Therefore, we have

\[
\sum_{\alpha \in \mathcal{U}(w) \cup \mathcal{U}(w_{0})} \mul_{\mathbb{S}}(\alpha) + |X| \geq |\mathcal{U}(w)| + 2 + |X| + |\{1, 2\}| \geq \deg_{G}(w) + 3.
\]
\begin{subsubcase}%
$\{\lambda_{1}, \lambda_{2}\} \cap \{1, 2\} = \emptyset$ and the color on $w_{0}v_{3}$ is not a common color.
\end{subsubcase}
Without loss of generality, assume that $\phi(w_{0}w_{1}) = 3$, $\phi(w_{0}w_{2}) = 4$, $\phi(w_{0}v_{3}) = \kappa - \Delta + 1$.

Suppose that $1 \notin \Upsilon(w_{0}w_{2}) \cup \Upsilon(w_{0}v_{3})$. Thus, there exists a $(3, 1, w_{0}, w_{2})$-critical path; otherwise, reassigning $1$ to $w_{0}w_{2}$ and $\alpha_{1}$ to $ww_{0}$ results in an acyclic edge coloring of $G$. But reassigning $\alpha_{1}, \beta_{1}$ and $1$ to $ww_{0}, w_{0}w_{2}$ and $w_{0}v_{3}$ respectively, yields an acyclic edge coloring of $G$. Hence, $1 \in \Upsilon(w_{0}w_{2}) \cup \Upsilon(w_{0}v_{3})$ and $1$ appears at least twice in $\mathbb{S}$. Similarly, we have that $2 \in \Upsilon(w_{0}w_{1}) \cup \Upsilon(w_{0}v_{3})$. Therefore, we have

\[
\sum_{\alpha \in \mathcal{U}(w) \cup \mathcal{U}(w_{0})} \mul_{\mathbb{S}}(\alpha) + |X| \geq |\mathcal{U}(w)| + 2 + |X| + |\{1, 2\}| \geq \deg_{G}(w) + 3.
\]

\begin{subcase}%
$|\mathcal{U}(w) \cap \mathcal{U}(w_{0})| = 3$.
\end{subcase}

\begin{claim}%
Every color in $\mathcal{U}(w)$ is in $\mathbb{S}$.
\end{claim}
\begin{proof}%
Assume that $w_{0}v_{1}, w_{0}v_{2}$ and $w_{0}v_{3}$ are colored with $\lambda_{1}, \lambda_{2}$ and $\lambda_{3}$, respectively. Suppose that $\lambda_{1} \notin \mathbb{S}$. If there is no $(\lambda_{2}, \alpha_{1}, w_{0}, v_{1})$-critical path, then reassigning $\alpha_{1}$ and $\lambda_{1}$ to $w_{0}v_{1}$ and $w_{0}v_{3}$ respectively, results in a new acyclic edge coloring of $G - ww_{0}$, which contradicts ($\ast$). Hence, there exists a $(\lambda_{2}, \alpha_{1}, w_{0}, v_{1})$-critical path, and hence there exists a $(\lambda_{3}, \alpha_{1}, w_{0}, w)$-critical path. But reassigning $\alpha_{1}$ and $\lambda_{1}$ to $w_{0}v_{1}$ and $w_{0}v_{2}$, yields another acyclic edge coloring of $G - ww_{0}$, which contradicts ($\ast$).

Hence, we have that $\lambda_{1} \in \mathbb{S}$. By symmetry, we have that $\{\lambda_{1}, \lambda_{2}, \lambda_{3}\} \subseteq \mathbb{S}$. Let $\tau$ be an arbitrary color in $\mathcal{U}(w) \setminus (\mathbb{S} \cup \{\lambda_{1}, \lambda_{2}, \lambda_{3}\})$. Let $\sigma$ be obtained from $\phi$ by reassigning $\tau$ to $w_{0}v_{1}$. It is obvious that $\sigma$ is an acyclic edge coloring of $G - ww_{0}$. So we can obtain a similar contradiction by replacing $\phi$ with $\sigma$. So we conclude that $\mathcal{U}(w) \subseteq \mathbb{S}$.
\end{proof}

\[
\sum_{\theta \in \mathcal{U}(w) \cup \mathcal{U}(w_{0})} \mul_{\mathbb{S}}(\theta) + |X| \geq |\mathcal{U}(w)| = \deg_{G}(w) - 1,
\]

In the following discussion, suppose that $\deg_{G}(w) + \deg_{G}(w_{0}) \leq \kappa - \Delta + 5$ and $ww_{0}$ is contained in two triangles $ww_{0}w_{1}$ and $ww_{0}w_{2}$ ($w_{1} = v_{1}$ and $w_{2} = v_{2}$).

\begin{subsubcase}%
$\mathcal{U}(w_{0}) \cap \{1, 2\} = \{1, 2\}$.
\end{subsubcase}
By symmetry, assume that $\phi(w_{0}w_{1}) =  3, \phi(w_{0}w_{2}) = 1, \phi(w_{0}v_{3}) = 2$. Since $\alpha_{1} \notin \mathcal{U}(w_{1})$, it follows that there exists a $(2, \alpha_{1}, w_{0}, w)$-critical path. Reassigning $\alpha_{1}$ to $w_{0}w_{1}$ will take us back to Subcase~\ref{712}.

\begin{subsubcase}\label{315}%
$\mathcal{U}(w_{0}) \cap \{1, 2\} = \{\lambda^{*}\}$ and $\lambda^{*}$ is not on $w_{0}v_{3}$.
\end{subsubcase}
By symmetry, assume that $\phi(w_{0}w_{1}) =  3, \phi(w_{0}w_{2}) = 1, \phi(w_{0}v_{3}) = 5$. Since $\alpha_{1} \notin \mathcal{U}(w_{1})$, it follows that there exists a $(5, \alpha_{1}, w_{0}, w)$-critical path. Reassigning $\alpha_{1}$ to $w_{0}w_{1}$ will take us back to Subcase~\ref{713}.

\begin{subsubcase}\label{341}%
$\mathcal{U}(w_{0}) \cap \{1, 2\} = \{\lambda^{*}\}$ and $\lambda^{*}$ is on $w_{0}v_{3}$.
\end{subsubcase}
By symmetry, assume that $\phi(w_{0}w_{1}) =  3, \phi(w_{0}w_{2}) = 4, \phi(w_{0}v_{3}) = 1$. Since $\alpha_{1} \notin \mathcal{U}(w_{1})$, it follows that there exists a $(4, \alpha_{1}, w_{0}, w)$-critical path. Reassigning $\alpha_{1}$ to $w_{0}w_{1}$ will take us back to Subcase~\ref{731}.

\begin{subsubcase}%
$\mathcal{U}(w_{0}) \cap \{1, 2\} = \emptyset$.
\end{subsubcase}
By symmetry, assume that $\phi(w_{0}w_{1}) =  3, \phi(w_{0}w_{2}) = 4, \phi(w_{0}v_{3}) = 5$. Suppose that $1$ only appears once in $\mathbb{S}$. Reassigning $1$ to $w_{0}w_{2}$ will create a (3, 1)-dichromatic cycle containing $w_{0}w_{2}$, for otherwise, we go back to Subcase~\ref{315}. But Reassigning $1$ to $w_{0}v_{3}$ will take us back to Subcase~\ref{341}. Hence, the color $1$ appears at least twice in $\mathbb{S}$. Similarly, the color $2$ appears at least twice in $\mathbb{S}$. Hence, we have

\[
\sum_{\theta \in \mathcal{U}(w) \cup \mathcal{U}(w_{0})} \mul_{\mathbb{S}}(\theta) + |X| \geq \deg_{G}(w) - 1 + |\{1, 2\}| = \deg_{G}(w) + 1. \qedhere
\]
\resetcounter
\end{proof}

\subsection{Local structure on $5$-vertices}

\begin{lemma}\label{5Sum}
Let $G$ be a $\kappa$-deletion-minimal graph with $\kappa \geq \Delta + 5$ and let $u$ be a $5$-vertex.
\begin{enumerate}[label = (\alph*)]
\item\label{5Suma} If $u$ is contained in a triangle $wuw_{1}w$ with $\deg_{G}(w) \leq \kappa - \Delta$ and $\deg_{G}(w_{1}) \leq 6$, then
\begin{equation}\label{EQ3}
\sum_{x \in N_{G}(u)} \deg_{G}(x) \geq 2\kappa - \deg_{G}(u) + 12 = 2\kappa + 7.
\end{equation}
\item\label{5Sumb} If $u$ is contained in a triangle $wuw_{1}w$ with $\deg_{G}(w) \leq \kappa - \Delta - 1$ and $\deg_{G}(w_{1}) \leq 7$, then
\begin{equation}\label{EQ4}
\sum_{x \in N_{G}(u)} \deg_{G}(x) \geq 2\kappa - \deg_{G}(u) + 12 = 2\kappa + 7.
\end{equation}
\end{enumerate}
\end{lemma}
\begin{proof}%

We may assume that
\begin{enumerate}[label= (*)]%
\item The graph $G - wu$ admits an acyclic edge coloring $\phi$ such that the number of common colors at $w$ and $u$ is minimum.
\end{enumerate}

Here, (a) and (b) will be proved together, so we may assume that $\deg_{G}(w) \leq \kappa - \Delta$. Since $\deg_{G}(w) + \deg_{G}(u) \leq \kappa - \Delta + 5 < \kappa + 2$, we have that $|\Upsilon(wu) \cap \Upsilon(uw)| = m \geq 1$. It follows that $|C(wu)| = \kappa - (\deg_{G}(w) + \deg_{G}(u) - m - 2) \geq \Delta - 2$. Without loss of generality, let $N_{G}(w) = \{u, w_{1}, w_{2}, \dots\}$ and $\phi(ww_{i}) = i$ for $1 \leq i \leq \deg_{G}(w)-1$. Let $N_{G}(u) = \{w, u_{1}, u_{2}, u_{3}, u_{4}\}$ and $\mathbb{S} = \Upsilon(uu_{1}) \uplus \Upsilon(uu_{2}) \uplus \Upsilon(uu_{3}) \uplus \Upsilon(uu_{4})$.

Let $\mathcal{A}(u_{1}) = C(wu) \setminus \mathcal{U}(u_{1}) = \{\alpha_{1}, \alpha_{2}, \dots\}$, $\mathcal{A}(u_{2}) = C(wu) \setminus \mathcal{U}(u_{2}) =  \{\beta_{1}, \beta_{2}, \dots\}$, $\mathcal{A}(u_{3}) = C(wu) \setminus \mathcal{U}(u_{3}) = \{\xi_{1}, \xi_{2}, \dots\}$, $\mathcal{A}(u_{4}) = C(wu) \setminus \mathcal{U}(u_{4}) = \{\zeta_{1}, \zeta_{2}, \dots\}$ and $\mathcal{A}(w_{2}) = C(wu) \setminus \mathcal{U}(w_{2}) = \{\zeta_{1}^{*}, \zeta_{2}^{*}, \dots\}$.

\begin{claim}\label{WV-Path5}%
For every color $\theta$ in $C(wu)$, there exists an $(\lambda, \theta, u, w)$-critical path for some $\lambda \in \mathcal{U}(w) \cap \mathcal{U}(u)$. Consequently, $\mul_{\mathbb{S}}(\theta) \geq 1$. 
\end{claim}

\begin{case}\label{1Common}%
$\mathcal{U}(w) \cap \mathcal{U}(u) = \{\lambda\}$. By symmetry, we may assume that $w_{1} = u_{1}$.
\end{case}
It follows that $|C(wu)| = \kappa - (\deg_{G}(w) + \deg_{G}(u) - 3) \geq \Delta - 2$.

\begin{subcase}\label{1Common1}%
The edge $ww_{1}$ is colored with $\lambda$. By symmetry, assume that $\phi(uw_{1}) = \kappa - \Delta + 2$, $\phi(uu_{2}) = 1$, $\phi(uu_{3}) = \kappa - \Delta$, $\phi(uu_{4}) = \kappa - \Delta + 1$.
\end{subcase}
By \autoref{WV-Path5}, we have that $\{\kappa - \Delta +3, \dots, \kappa\}\subseteq \Upsilon(ww_{1})\cap \Upsilon(uu_{2})$. Moreover, $\mathcal{U}(w_{1}) = \{1, \kappa - \Delta + 2\} \cup \{\kappa - \Delta + 3, \dots, \kappa\}$, $\deg_{G}(w_{1}) = \Delta$ and $\deg_{G}(u_{2}) \geq \Delta - 1$. Notice that $|\Upsilon(uu_{2}) \cap \{2, 3, \dots, \kappa - \Delta - 1\}| \leq 1$, thus there exists a color $\zeta$ which is in $\{2, 3, \dots, \kappa - \Delta - 1\} \setminus \Upsilon(uu_{2})$ (note that this set is nonempty). But assigning $\kappa - \Delta + 2$ to $uw$ and $\zeta$ to $uw_{1}$ results in an acyclic edge coloring of $G$, a contradiction.

\begin{subcase}\label{1Common2}%
The edge $uw_{1}$ is colored with $\lambda$. By symmetry, assume that $\phi(uw_{1}) = 2$, $\phi(uu_{2}) = \kappa - \Delta$, $\phi(uu_{3}) = \kappa - \Delta + 1$, $\phi(uu_{4}) = \kappa - \Delta + 2$.
\end{subcase}
By \autoref{WV-Path5}, we have that $\{\kappa - \Delta + 3, \dots, \kappa\} \subseteq \Upsilon(uw_{1})\cap \Upsilon(ww_{2})$ and $\deg_{G}(w_{1}) = \Delta$ and $\deg_{G}(w_{2}) \geq \Delta - 1$. Modify $\phi$ by reassigning $1$ to $wu$ and reassigning a color in $\{\kappa - \Delta, \kappa - \Delta + 1, \kappa - \Delta + 2\} \setminus \mathcal{U}(w_{2})$ to $ww_{1}$, we obtain an acyclic edge coloring of $G$, a contradiction.

\begin{subcase}%
Neither $w_{1}w$ nor $w_{1}u$ is colored with $\lambda$. By symmetry, assume that $\phi(uw_{1}) = \kappa - \Delta$, $\phi(uu_{2}) = 2$, $\phi(uu_{3}) = \kappa - \Delta + 1$, $\phi(uu_{4}) = \kappa - \Delta + 2$.
\end{subcase}
By \autoref{WV-Path5}, we have that $C(wu) \subseteq \Upsilon(uu_{2})\cap \Upsilon(ww_{2})$ and $\deg_{G}(w_{2}) \geq \Delta - 1$ and $\deg_{G}(u_{2}) \geq \Delta - 1$. Notice that $\{1, \kappa - \Delta\} \nsubseteq \mathcal{U}(w_{2})$.

If $\deg_{G}(w) \leq \kappa - \Delta - 1$, then $\mathcal{U}(w) = \{1, 2, \dots, \kappa - \Delta - 2\}$ and $\deg(w_{2}) = \deg(u_{2}) = \Delta$, but reassigning $1$ to $uu_{2}$ will take us back to Subcase~\ref{1Common1}.

So we may assume that $\deg(w) = \kappa - \Delta$, $C(wu) = \{\kappa - \Delta + 3, \dots, \kappa\}$ and $\deg(w_{1}) \leq 6$.

Suppose that $C(wu) \subseteq \mathcal{U}(w_{1})$. Thus $\mathcal{U}(w_{1}) = \{1, \kappa - \Delta\} \cup C(wu)$. If $1 \notin \mathcal{U}(w_{2})$, then reassigning $1, \kappa - \Delta$ and $3$ to $wu, ww_{1}$ and $w_{1}u$ respectively results in an acyclic edge coloring of $G$, a contradiction. So we may assume that $1 \in \mathcal{U}(w_{2})$ and $\Upsilon(ww_{2}) = \{1\} \cup \{\kappa - \Delta + 3, \dots, \kappa\}$. But reassigning $\kappa - \Delta$ to $wu$ and $3$ to $w_{1}u$ results in an acyclic edge coloring of $G$, a contradiction. Hence, we have that $C(wu) \nsubseteq \mathcal{U}(w_{1})$. 

We further suppose that $1 \in \mathcal{U}(u_{2})$ and $\Upsilon(uu_{2}) = \{1\} \cup \{\kappa - \Delta + 3, \dots, \kappa\}$. If there is a $(2, 1, u, w)$-critical path, then $\deg_{G}(w_{2}) = \Delta(G)$ and $\Upsilon(ww_{2}) = \{1\} \cup \{\kappa  - \Delta + 3, \dots, \kappa\}$, but reassigning $\kappa - \Delta$ to $ww_{2}$ will take us back to Subcase~\ref{1Common2}. So we may assume that there is no $(2, 1, u, w)$-critical path. There exists a $(\tau^{*}, \alpha_{1}, w, w_{1})$-critical path with some $\tau^{*} \in \mathcal{U}(w) \setminus \{1, 2\}$, otherwise reassigning $\alpha_{1}$ to $ww_{1}$ and $1$ to $uw$ will result in an acyclic edge coloring of $G$. By symmetry, assume that $\tau^{*} = 3$ and there exists a $(3, \alpha_{1}, w, w_{1})$-critical path. But reassigning $3$ to $uu_{2}$ and $\alpha_{1}$ to $wu$ results in an acyclic edge coloring of $G$.

So we may assume that $1 \notin \mathcal{U}(u_{2})$. There exists a $(\kappa  - \Delta + 1, 1, u, u_{2})$- or $(\kappa - \Delta + 2, 1, u, u_{2})$-critical path; otherwise, reassigning $1$ to $uu_{2}$ will take us back to Subcase~\ref{1Common1}. By symmetry, assume that there exists a $(\kappa  - \Delta + 2, 1, u, u_{2})$-critical path and $1 \in \Upsilon(uu_{4})$, thus $\deg_{G}(u_{2}) = \Delta(G)$ and $\Upsilon(uu_{2}) = \{\kappa - \Delta + 2, \kappa - \Delta + 3, \dots, \kappa\}$.

There exists a $(\kappa  - \Delta + 1, \alpha_{1}, u, w_{1})$- or $(\kappa - \Delta + 2, \alpha_{1}, u, w_{1})$-critical path; otherwise, reassigning $\alpha_{1}$ to $uw_{1}$ and $\kappa - \Delta$ to $uw$ will result in an acyclic edge coloring of $G$. Hence, $\{\kappa - \Delta + 1, \kappa - \Delta + 2\} \cap \mathcal{U}(w_{1}) \neq \emptyset$. Similarly, there exists a $(\tau, \alpha_{1}, w, w_{1})$-critical path with some $\tau \in \mathcal{U}(w) \setminus \{1, 2\}$. By symmetry, assume that $\tau = 3$ and there exists a $(3, \alpha_{1}, w, w_{1})$-critical path. Hence, $|\mathcal{U}(w_{1}) \cap (\mathcal{U}(w) \cup \mathcal{U}(u))| \geq 4$ and $|\mathcal{U}(w_{1}) \cap C(wu)| \leq 2$, and then $|C(wu) \setminus \mathcal{U}(w_{1})| \geq \Delta - 4$.

Suppose that $\mathcal{A}(u_{4}) \cap \mathcal{A}(w_{1}) \neq \emptyset$, say $\zeta \in \mathcal{A}(u_{4}) \cap \mathcal{A}(w_{1})$. Thus there exists a $(\kappa - \Delta + 1, \zeta, u, w_{1})$-critical path; otherwise, reassigning $\zeta$ to $uw_{1}$ and $\kappa - \Delta$ to $uw$ will result in an acyclic edge coloring of $G$. There exists a $(2, \kappa - \Delta + 2, u, w)$-critical path, otherwise reassigning $\zeta$ to $uu_{4}$ and $\kappa - \Delta + 2$ to $uw$ will result in an acyclic edge coloring of $G$. Hence, we have that $\Upsilon(ww_{2}) = \Upsilon(uu_{2}) = \{\kappa - \Delta + 2, \kappa - \Delta + 3, \dots, \kappa\}$. But reassigning $\kappa - \Delta$ to $ww_{2}$ will take us back to Subcase~\ref{1Common2}. So we have that $\mathcal{A}(u_{4}) \cap \mathcal{A}(w_{1}) = \emptyset$.

There exists a $(\kappa  - \Delta + 2, 3, u, u_{2})$-critical path, for otherwise reassigning $3$ to $uu_{2}$ and $\alpha_{1}$ to $uw$ will result in an acyclic edge coloring of $G$. It follows that $\{1, 3\} \cup \mathcal{A}(w_{1}) \subseteq \Upsilon(uu_{4})$. If $2 \notin \mathcal{U}(u_{4})$ and there exists no $(\kappa - \Delta + 1, 2, u, u_{4})$-critical path, then reassigning $2, 1$ and $\alpha_{1}$ to $uu_{4}, uu_{2}$ and $uw$, respectively, will result in an acyclic edge coloring of $G$. Hence, $\mathcal{U}(u_{4}) = \{1, 2, 3, \kappa - \Delta + 2\} \cup \mathcal{A}(w_{1})$ or $\mathcal{U}(u_{4}) = \{1, 3, \kappa - \Delta + 1, \kappa - \Delta + 2\} \cup \mathcal{A}(w_{1})$. Consequently, we have that $|\mathcal{A}(w_{1})| = \Delta - 4$, and then $\mathcal{U}(w_{1}) \cap (\mathcal{U}(w) \cup \mathcal{U}(u)) = \{1, 3, \kappa - \Delta, \kappa - \Delta + 1\}$ or $\{1, 3, \kappa - \Delta, \kappa - \Delta + 2\}$.

There exists a $(\kappa - \Delta + 1, 2, u, w_{1})$- or $(\kappa - \Delta + 2, 2, u, w_{1})$-critical path, for otherwise we reassign $\alpha_{1}, 2$ and $\kappa - \Delta$ to $uw, uw_{1}$ and $uu_{2}$. Thus, $\mathcal{U}(u_{4}) = \{1, 2, 3, \kappa - \Delta + 2\} \cup \mathcal{A}(w_{1})$. If $\mathcal{U}(w_{1}) \cap (\mathcal{U}(w) \cup \mathcal{U}(u)) = \{1, 3, \kappa - \Delta, \kappa - \Delta + 2\}$, then we reassign $\alpha_{1}, 2, 3$ and $\kappa - \Delta$ to $uw, uw_{1}, uu_{2}$ and $uu_{4}$. Therefore, $\mathcal{U}(w_{1}) \cap (\mathcal{U}(w) \cup \mathcal{U}(u)) = \{1, 3, \kappa - \Delta, \kappa - \Delta + 1\}$, but reassigning $\kappa - \Delta + 2, 1$ and $4$ to $uw, uu_{2}$ and $uu_{4}$ results in an acyclic edge coloring of $G$.

\begin{case}%
$\mathcal{U}(w) \cap \mathcal{U}(u) = \{\lambda_{1}, \dots, \lambda_{m}\}$ and $m \geq 2$.
\end{case}
We can relabel the vertices in $\{u_{1}, u_{2}, u_{3}, u_{4}\}$ as $\{v_{1}, v_{2}, v_{3}, v_{4}\}$. By symmetry, we may assume that $\phi(uv_{i}) = \lambda_{i}$ for $i \in \{1, \dots, m\}$.

\begin{claim}\label{CapEmpty}
The sets $\mathcal{A}(v_{1}), \mathcal{A}(v_{2}), \dots, \mathcal{A}(v_{m})$ are pairwise disjoint.
\end{claim}
\begin{proof}
Suppose, to the contrary, that $\alpha \in \mathcal{A}(v_{1}) \cap \mathcal{A}(v_{2})$. By Claim \ref{WV-Path5} and the symmetry, we may assume that there exists a $(\lambda_{3}, \alpha, u, w)$-critical path and $m \geq 3$, which implies that there exists no $(\lambda_{3}, \alpha, u, v_{2})$-critical path. Consequently, there exists a $(\phi(uv_{4}), \alpha, u, v_{2})$-critical path; otherwise, reassigning $\alpha$ to $uv_{2}$ to obtain a new acyclic edge coloring of $G - wu$, which contradicts the minimality of $m$. Now, reassigning $\alpha$ to $uv_{1}$ to obtain an acyclic edge coloring $\pi$ of $G - wu$, but $|\mathcal{U}_{\pi}(u) \cap \mathcal{U}_{\pi}(w)| < |\mathcal{U}(u) \cap \mathcal{U}(w)|$, which is a contradiction.
\end{proof}
\begin{claim}
Every color in $C(wu)$ appears at least twice in $\mathbb{S}$.
\end{claim}
\begin{proof}%
Suppose that there exists a color $\alpha$ in $C(wu)$ such that $\mul_{\mathbb{S}}(\alpha) = 1$. By \autoref{WV-Path5} and symmetry, we may assume that there exists a $(\lambda_{1}, \alpha, u, w)$-critical path and $\alpha \in \mathcal{U}(v_{1})$. But reassigning $\alpha$ to $uv_{2}$ results in a new acyclic edge coloring of $G - wu$, which contradicts the assumption ($\ast$).
\end{proof}
Let $X = \{\,\theta \mid \theta \in C(wu) \mbox{ and }\mul_{\mathbb{S}}(\theta) \geq 3\,\}$.
\begin{align*}%
&\sum_{x \in N_{G}(u)} \deg_{G}(x)\\
&= \deg_{G}(u) + \deg_{G}(w) -1 + \sum_{\theta \in [\kappa]} \mul_{\mathbb{S}}(\theta)\\
&= \deg_{G}(u) + \deg_{G}(w) -1 + \sum_{\theta \in C(wu)} \mul_{\mathbb{S}}(\theta) + \sum_{\theta \in \mathcal{U}(w) \cup \mathcal{U}(u)} \mul_{\mathbb{S}}(\theta) \\
&\geq \deg_{G}(u) + \deg_{G}(w) - 1 + 2|C(wu)| + |X| + \sum_{\theta \in \mathcal{U}(w) \cup \mathcal{U}(u)} \mul_{\mathbb{S}}(\theta)\\
&= \deg_{G}(u) + \deg_{G}(w) - 1 + 2(\kappa - (\deg_{G}(w) + \deg_{G}(u) - 2 - m)) + |X| + \sum_{\theta \in \mathcal{U}(w) \cup \mathcal{U}(u)} \mul_{\mathbb{S}}(\theta)\\
&= 2\kappa - \deg_{G}(u) - \deg_{G}(w) + 2m + 3 + |X| + \sum_{\theta \in \mathcal{U}(w) \cup \mathcal{U}(u)} \mul_{\mathbb{S}}(\theta)
\end{align*}

It is sufficient to prove that
\begin{equation}\label{EQ*}%
\sum_{\theta \in \mathcal{U}(w) \cup \mathcal{U}(u)} \mul_{\mathbb{S}}(\theta) + |X| \geq \deg_{G}(w) - 2m + 9.
\end{equation}

\begin{subcase}\label{2Common}%
$\mathcal{U}(w) \cap \mathcal{U}(u) = \{\lambda_{1}, \lambda_{2}\}$ and $w_{1} = u_{1}$. Note that $\mathcal{A}(w_{1}) \neq \emptyset$.
\end{subcase}
\begin{subsubcase}\label{2Common1}%
The two colors on the edges $w_{1}w$ and $w_{1}u$ are all common colors.
\end{subsubcase}
Without loss of generality, assume that $\phi(uw_{1}) = 2$, $\phi(uu_{2}) = 1$, $\phi(uu_{3}) = \kappa - \Delta$ and $\phi(uu_{4}) = \kappa - \Delta + 1$. Consequently, we conclude that $\{\kappa - \Delta +2, \dots, \kappa\} \subseteq \mathcal{U}(w_{1})$ and $\deg_{G}(w_{1}) \geq \Delta + 1$, a contradiction.

\begin{subsubcase}\label{2Common2}%
The color on $w_{1}w$ is a common color and the color on $w_{1}u$ is not a common color.
\end{subsubcase}
Without loss of generality, assume that $\phi(uw_{1}) = \kappa - \Delta$, $\phi(uu_{2}) = 1$, $\phi(uu_{3}) = 2$ and $\phi(uu_{4}) = \kappa - \Delta + 1$.

For every color $\alpha_{i} \in \mathcal{A}(w_{1})$, there exists a $(\theta_{i}, \alpha_{i}, w, w_{1})$-critical path with some $\theta_{i} \in \mathcal{U}(w) \setminus \{1, 2\}$; otherwise, reassigning $\alpha_{i}$ to $ww_{1}$ will take us back to Case 1. By symmetry, we may assume that there exists a $(3, \alpha^{*}, w, w_{1})$-critical path with some $\alpha^{*}$, and then $3 \in \mathcal{U}(w_{1})$. If $\Upsilon(ww_{2}) \subseteq C(wu)$, then reassigning $\kappa - \Delta$ to $ww_{2}$ will take us back to Subcase~\ref{2Common1}. So we have that $\Upsilon(ww_{2}) \nsubseteq C(wu)$ and $\mathcal{A}(w_{2}) \neq \emptyset$. Consequently, for every color $\zeta_{i}^{*} \in \mathcal{A}(w_{2})$, there exists a $(\mu_{i}, \zeta_{i}^{*}, w, w_{2})$-critical path with some $\mu_{i} \in \mathcal{U}(w) \setminus \{1, 2\}$; otherwise, reassigning $\zeta_{i}^{*}$ to $ww_{2}$ will take us back to Case 1. Hence, $\{1, 3, \kappa - \Delta\} \subseteq \mathcal{U}(w_{1})$ and $\{2, \mu_{1}\} \subseteq \mathcal{U}(w_{2})$, and then $|\mathcal{A}(w_{1})| \geq 2$ and $|\mathcal{A}(w_{2})| \geq 1$.

If $\Upsilon(uu_{2}) \subseteq C(wu)$, then reassigning $\mu_{1}$ to $uu_{2}$ and $\zeta_{1}^{*}$ to $wu$ results in an acyclic edge coloring of $G$, a contradiction. Thus, we have that $\Upsilon(uu_{2}) \nsubseteq C(wu)$ and $\mathcal{A}(u_{2}) \neq \emptyset$. For every color $\beta_{i} \in \mathcal{A}(u_{2})$, there exists an $(\varepsilon_{i}, \beta_{i}, u, u_{2})$-critical path with some $\varepsilon_{i} \in \{\kappa - \Delta, \kappa - \Delta + 1\}$; otherwise, reassigning $\beta_{i}$ to $uu_{2}$ will take us back to Case~\ref{1Common}.

If $\Upsilon(uu_{3}) \subseteq C(wu)$, then reassigning $3$ to $uu_{3}$ and $\alpha^{*}$ to $wu$ results in an acyclic edge coloring of $G$, a contradiction. Thus, we have that $\Upsilon(uu_{3}) \nsubseteq C(wu)$ and $\mathcal{A}(u_{3}) \neq \emptyset$. Consequently, for every color $\xi_{i} \in \mathcal{A}(u_{3})$, there exists a $(m_{i}, \xi_{i}, u, u_{3})$-critical path with some $m_{i} \in \{\kappa - \Delta, \kappa - \Delta + 1\}$; otherwise, reassigning $\xi_{i}$ to $uu_{3}$ will take us back to Case~\ref{1Common}.

\begin{claim}\label{Claim*}%
$\mathcal{A}(u_{2}) \cap \mathcal{A}(u_{4}) = \emptyset$.
\end{claim}
\begin{proof}[Proof of \autoref{Claim*}]%
Suppose that $\beta_{1} \in \mathcal{A}(u_{2}) \cap \mathcal{A}(u_{4})$. It follows that there exists a $(\kappa - \Delta, \beta_{1}, u, u_{2})$-critical path, $\kappa - \Delta \in\Upsilon(uu_{2})$ and $\beta_{1} \in \Upsilon(uw_{1})$. Also, there exists a $(1, \kappa - \Delta + 1, w, u)$- or $(2, \kappa - \Delta + 1, w, u)$-critical path; otherwise, reassigning $\beta_{1}$ to $uu_{4}$ and $\kappa - \Delta + 1$ to $wu$ results in an acyclic edge coloring of $G$. Suppose that there exists a $(1, \kappa - \Delta + 1, w, u)$-critical path and $\kappa - \Delta + 1 \in \Upsilon(ww_{1}) \cap \Upsilon(uu_{2})$. It follows that $\{1, \kappa - \Delta, \kappa - \Delta + 1\} \subseteq \mathcal{U}(u_{2})$ and $|\mathcal{A}(u_{2})| \geq 2$. Furthermore, we can conclude that $\{1, 3, \kappa - \Delta, \kappa - \Delta + 1, \beta_{1}\} \cup \mathcal{A}(w_{2}) \subseteq \mathcal{U}(w_{1})$. Note that $\mathcal{A}(w_{2}) \cap \mathcal{A}(u_{2}) = \emptyset$, thus $\mathcal{U}(w_{1}) = \{1, 3, \kappa - \Delta, \kappa - \Delta + 1, \beta_{1}\} \cup \mathcal{A}(w_{2})$ and $\mathcal{U}(w_{2}) \cap (\mathcal{U}(w) \cup \mathcal{U}(u)) = \{2, \mu_{1}\}$. Recall that $\beta_{2} \notin \mathcal{A}(w_{2})$, thus $\varepsilon_{2} = \kappa - \Delta + 1$ and there exists a $(\kappa - \Delta + 1, \beta_{2}, u, u_{2})$-critical path. Now, reassigning $\beta_{2}$ to $uw_{1}$ and $\kappa - \Delta$ to $wu$ results in an acyclic edge coloring of $G$.

So, we may assume that there exists a $(2, \kappa - \Delta + 1, w, u)$-critical path. Hence, $\{1, \kappa - \Delta, 3, \beta_{1}\} \cup \mathcal{A}(w_{2}) \subseteq \mathcal{U}(w_{1})$ and $\{2, \mu_{1}, \kappa - \Delta + 1\} \subseteq \Upsilon(ww_{2})$. It follows that $\mathcal{U}(w_{1}) = \{1, \kappa - \Delta, 3, \beta_{1}\} \cup \mathcal{A}(w_{2})$ and $\mathcal{U}(w_{2}) \cap (\mathcal{U}(w) \cup \mathcal{U}(u)) = \{2, \mu_{1}, \kappa - \Delta + 1\}$. Now, reassigning $\alpha_{1}$ to $uw_{1}$ and $\kappa - \Delta$ to $wu$ results in an acyclic edge coloring of $G$. This completes the proof of \autoref{Claim*}.
\end{proof}

\begin{claim}\label{Claim**}%
$\mathcal{A}(u_{2}) \cap \mathcal{A}(w_{1}) = \emptyset$.
\end{claim}
\begin{proof}[Proof of \autoref{Claim**}]
By contradiction, assume that $\alpha_{1} = \beta_{1}$. It follows that there exists a $(\kappa - \Delta + 1, \beta_{1}, u, u_{2})$-critical path and $\kappa - \Delta + 1 \in \mathcal{U}(u_{2})$. There exists a $(2, \kappa - \Delta, w, u)$-critical path; otherwise, reassigning $\beta_{1}$ to $uw_{1}$ and $\kappa - \Delta$ to $uw$ results in an acyclic edge coloring of $G$, a contradiction. So we have that $\kappa - \Delta \in \Upsilon(ww_{2}) \cap \Upsilon(uu_{3})$.

Note that $\{1, \kappa - \Delta, 3\} \subseteq \mathcal{U}(w_{1})$ and $\{2, \mu_{1}, \kappa - \Delta\} \subseteq \mathcal{U}(w_{2})$. If $\deg_{G} (w)\leq \kappa - \Delta$ and $\deg_{G}(w_{1}) \leq 6$, then $\mathcal{A}(w_{2}) \subseteq \mathcal{U}(w_{1})$ with $|\mathcal{A}(w_{2})| \geq 2$,  and then $|\mathcal{U}(w_{1}) \cap \mathcal{U}(w)| \leq 3$; similarly, if $\deg_{G}(w) \leq \kappa - \Delta - 1$ and $\deg_{G}(w_{1}) \leq 7$, then $\mathcal{A}(w_{2}) \subseteq \mathcal{U}(w_{1})$ with $|\mathcal{A}(w_{2})| \geq 3$, and then $|\mathcal{U}(w_{1}) \cap \mathcal{U}(w)| \leq 3$. If $\mathcal{U}(w_{1}) \cap \mathcal{U}(w)  = \{1, 3\}$, then $\{2, \kappa - \Delta\} \subsetneq \mathcal{U}(u_{3}) \cap (\mathcal{U}(w) \cup \mathcal{U}(u))$; otherwise, reassigning $\alpha_{1}, \alpha_{2}$ and $3$ to $uw_{1}, uw$ and $uu_{3}$ respectively results in an acyclic edge coloring of $G$. Suppose that $\mathcal{U}(w_{1}) \cap \mathcal{U}(w) = \{1, 3, s\}$. Since $|\mathcal{A}(w_{1})| \geq 3$, thus there exists a $\tau \in \{3, s\}$ and $\alpha_{i}, \alpha_{j}$ such that both $(\tau, \alpha_{i}, w, w_{1})$- and $(\tau, \alpha_{j}, w, w_{1})$-critical path exist, and thus $\{2, \kappa - \Delta\} \subsetneq \mathcal{U}(u_{3}) \cap (\mathcal{U}(w) \cup \mathcal{U}(u))$; otherwise, reassigning $\alpha_{i}, \alpha_{j}$ and $\tau$ to $uw_{1}, uw$ and $uu_{3}$ respectively. Anyway, we have that $|\mathcal{U}(u_{3}) \cap (\mathcal{U}(w) \cup \mathcal{U}(u))| \geq 3$.

If $\kappa - \Delta$ only appears only once (at $u_{3}$) in $\mathbb{S}$, then reassigning $\kappa - \Delta$ to $uu_{2}$ and $\beta_{1}$ to $uw_{1}$ will take us back to Case 1. So we conclude that the color $\kappa - \Delta$ appears at least twice in $\mathbb{S}$.

If $1 \notin \mathbb{S} \setminus \mathcal{U}(w_{1})$, then reassigning $1, \beta_{1}$  and $\xi_{1}$ to $uu_{4}, uu_{2}$ and $wu$ respectively, results in an acyclic edge coloring of $G$, a contradiction. Therefore, the color $1$ appears at least twice in $\mathbb{S}$.

Suppose that $4 \notin \mathbb{S}$. Thus there exists a $(4, \xi_{1}, w, u_{2})$-alternating path; otherwise, reassigning $4$ to $uu_{2}$ and $\xi_{1}$ to $wu$ results in an acyclic edge coloring of $G$, a contradiction. Now, reassigning $4, \beta_{1}$ and $\xi_{1}$ to $uu_{4}, uu_{2}$ and $wu$ respectively, results in an acyclic edge coloring $G$, a contradiction. So we conclude that $4 \in \mathbb{S}$. By symmetry, we can also obtain that every color in $\mathcal{U}(w) \setminus \{1, 2, 3\}$ appears in $\mathbb{S}$.

Suppose that every color in $\mathcal{U}(w) \setminus \{1, 2\}$ appears exactly once in $\mathbb{S}$. Suppose that $\mathcal{U}(w_{1}) \cap (\mathcal{U}(w) \setminus \{1, 2\}) = \{3, s\}$. Thus, $\mathcal{U}(w_{1}) = \{1, \kappa - \Delta, 3, s\} \cup \mathcal{A}(w_{2})$ and $\kappa - \Delta + 1 \notin \mathcal{U}(w_{1})$. Since $|\mathcal{A}(w_{1})| \geq 3$, thus there exists a $\tau \in \{3, s\}$ and $\alpha_{i}, \alpha_{j}$ such that both $(\tau, \alpha_{i}, w, w_{1})$- and $(\tau, \alpha_{j}, w, w_{1})$-critical path exist. Reassigning $\tau, \alpha_{i}$ and $\alpha_{j}$ to $uu_{3}, uw_{1}$ and $wu$ respectively, results in an acyclic edge coloring of $G$, a contradiction. So we may assume that $|\mathcal{U}(w_{1}) \cap (\mathcal{U}(w) \setminus \{1, 2\})| = 1$, that is $\mathcal{U}(w_{1}) \cap (\mathcal{U}(w) \setminus \{1, 2\}) = \{3\}$. Reassigning $3, \alpha_{1}$ and $\alpha_{2}$ to $uu_{3}, uw_{1}$ and $wu$ respectively, results in an acyclic edge coloring of $G$, a contradiction. Hence, we may assume that the color $3$ appears at least twice in $\mathbb{S}$.

Suppose that $\xi_{1} \in \mathcal{A}(u_{4})$. Thus, there exists a $(\kappa - \Delta, \xi_{1}, u, u_{3})$-critical path; otherwise, reassigning $\xi_{1}$ to $uu_{3}$ will take us back to Case 1. Furthermore, $\kappa - \Delta + 1 \in \Upsilon(ww_{1}) \cup \Upsilon(uu_{3})$; otherwise, reassigning $\xi_{1}$ to $uu_{4}$ and $\kappa - \Delta + 1$ to $wu$ results in an acyclic edge coloring of $G$. If $2 \notin \mathbb{S}$, then reassigning $\alpha_{1}, 2$ and $\xi_{1}$ to $wu, uw_{1}$ $uu_{3}$ respectively, results in an acyclic edge coloring of $G$. So we have that $2 \in \mathbb{S}$. Hence,
\[
\sum_{\theta \in \mathcal{U}(w) \cup \mathcal{U}(u)} \mul_{\mathbb{S}}(\theta) + |X| \geq
|\{4, \dots, \deg_{G}(w) - 1\}| + 2|\{1, 3, \kappa - \Delta, \kappa - \Delta + 1\}| + |\{2\}| = \deg_{G}(w) + 5.
\]

So we may assume that $\mathcal{A}(u_{3}) \cap \mathcal{A}(u_{4}) = \emptyset$. It is obvious that $\mathcal{A}(u_{3}) \subseteq X$. Hence,

\[
\sum_{\theta \in \mathcal{U}(w) \cup \mathcal{U}(u)} \mul_{\mathbb{S}}(\theta) + |X| \geq
|\{4, \dots, \deg_{G}(w) - 1\}| + 2|\{1, 3, \kappa - \Delta\}| + |\{\kappa - \Delta + 1\}| + |\mathcal{A}(u_{3})| \geq \deg_{G}(w) + 4.
\]
The equality holds only if $\kappa - \Delta + 1$ appears only once in $\mathbb{S}$ and $2$ does not appear in $\mathbb{S}$; but reassigning $\alpha_{1}, 2$ and $\xi_{1}$ to $wu, uw_{1}$ and $uu_{3}$ respectively, results in an acyclic edge coloring. Therefore, inequality \eqref{EQ*} holds, we are done.  This completes the proof \autoref{Claim**}.
\end{proof}

By \autoref{Claim**}, the three sets $\mathcal{A}(w_{1}), \mathcal{A}(u_{2})$ and $\mathcal{A}(u_{3})$ are pairwise disjoint.

(1) Suppose that there exists no $(2, \kappa - \Delta, w, u)$-critical path. This implies that there exists a $(\kappa - \Delta + 1, \alpha_{i}, u, w_{1})$-critical path; otherwise, reassigning $\alpha_{i}$ to $uw_{1}$ and $\kappa - \Delta$ to $wu$ results in an acyclic edge coloring of $G$. Thus, $\{1, 3, \kappa - \Delta, \kappa - \Delta + 1\} \subseteq \mathcal{U}(w_{1})$ and $\mathcal{A}(w_{1}) \subseteq \mathcal{U}(u_{4})$. Note that $\mathcal{A}(u_{2}) \cup \mathcal{A}(u_{3}) \subseteq \mathcal{U}(w_{1})$, thus $|\mathcal{A}(u_{2})| = |\mathcal{A}(u_{3})| = 1$ and $\mathcal{U}(w_{1}) \cap (\mathcal{U}(w) \cup \mathcal{U}(u)) = \{1, 3, \kappa - \Delta, \kappa - \Delta + 1\}$. Similarly, we know that $\mathcal{A}(u_{2}) \cup \mathcal{A}(w_{2}) \subseteq \mathcal{U}(w_{1})$, which implies that $|\mathcal{A}(w_{2})| = |\mathcal{A}(u_{3})| = 1$ and $\mathcal{A}(w_{2}) = \mathcal{A}(u_{3})$. Hence, $\mathcal{U}(w_{2}) \cap (\mathcal{U}(w) \cup \mathcal{U}(u)) = \{2, \mu_{1}\}$. By \autoref{Claim*}, we conclude that $\mathcal{U}(u_{4}) \supseteq \mathcal{A}(w_{1}) \cup \mathcal{A}(u_{2}) \cup \{\kappa - \Delta + 1\}$. If $\Upsilon(uu_{4}) \subseteq C(uw)$, then reassigning $3, \alpha^{*}$ and $\kappa - \Delta$ to $uu_{4}, uw_{1}$ and $wu$ respectively results in an acyclic edge coloring of $G$. Note that $|\mathcal{A}(w_{1})| + |\mathcal{A}(u_{2})| = \Delta - 2$, so we may assume that $|\Upsilon(uu_{4}) \cap (\mathcal{U}(w) \cup \mathcal{U}(u))| = 1$. In addition, $\Upsilon(uu_{4}) \cap C(uw) = \mathcal{A}(w_{1}) \cup \mathcal{A}(u_{2})$ and $\mathcal{U}(u_{2}) \cap (\mathcal{U}(w) \cup \mathcal{U}(u)) = \{1, \varepsilon_{1}\}$. Recall that $\mathcal{A}(w_{1}), \mathcal{A}(u_{2})$ and $\mathcal{A}(u_{3})$ are pairwise disjoint, thus $\mathcal{A}(u_{3}) \cap \mathcal{U}(u_{4}) = \emptyset$, and then there exists a $(\kappa - \Delta, \xi_{1}, u, u_{3})$-critical path and $\kappa - \Delta \in \Upsilon(uu_{3})$. There exists a $(1, \kappa - \Delta + 1, u, w)$-critical path; otherwise, reassigning $\xi_{1}$ and $\kappa - \Delta + 1$ to $uu_{4}$ and $uw$ results in an acyclic edge coloring of $G$. Hence, $\mathcal{U}(u_{2}) \cap (\mathcal{U}(w) \cup \mathcal{U}(u)) = \{1, \varepsilon_{1}\}= \{1, \kappa - \Delta + 1\}$. There exists a $(\kappa - \Delta + 1, \mu_{1}, u, u_{2})$-critical path, otherwise, reassigning $\mu_{1}$ to $uu_{2}$ and $\zeta_{1}^{*}$ to $uw$ results in an acyclic edge coloring of $G$. Hence, $\Upsilon(uu_{4}) \cap (\mathcal{U}(w) \cup \mathcal{U}(u)) = \{\mu_{1}, \kappa - \Delta + 1\}$. Now, reassigning $\zeta_{1}^{*}, \alpha_{1}, \mu_{1}$ and $\kappa - \Delta$ to $uw, uw_{1}, uu_{2}$ and $uu_{4}$ respectively, yields an acyclic edge coloring of $G$.

(2) Now, we may assume that there exists a $(2, \kappa - \Delta, w, u)$-critical path and $\kappa - \Delta \in \Upsilon(uu_{3}) \cap \Upsilon(ww_{2})$. Clearly, $\mathcal{U}(w_{1}) \supseteq \mathcal{A}(u_{2}) \cup \mathcal{A}(w_{2}) \cup \{1, 3, \kappa - \Delta\}$. If $\deg_{G}(w) \leq \kappa - \Delta - 1$, then $\deg_{G}(w_{1}) \geq 2 + 3 + 3 = 8$, a contradiction. Thus, $\deg_{G}(w) = \kappa - \Delta$, which implies that $\mathcal{U}(w_{1}) = \mathcal{A}(u_{2}) \cup \mathcal{A}(w_{2}) \cup \{1, 3, \kappa - \Delta\}$, $|\mathcal{A}(u_{2})| = 1$ and $|\mathcal{A}(w_{2})| = 2$. It is easy to see that $\mathcal{U}(w_{2}) \cap (\mathcal{U}(w) \cup \mathcal{U}(u)) = \{2, \kappa - \Delta, \mu_{1}\}$ and $\mathcal{U}(u_{2}) \cap (\mathcal{U}(w) \cup \mathcal{U}(u)) =  \{1, \varepsilon_{1}\}$. If $3 \notin \mathcal{U}(u_{3})$, then there exists a $(\kappa - \Delta + 1, 3, u, u_{3})$-critical path; otherwise, reassigning $\alpha_{1}, \alpha_{2}$ and $3$ to $uw_{1}, uw$ and $uu_{3}$ respectively results in an acyclic edge coloring of $G$. Hence, we have that $\{3, \kappa - \Delta + 1\} \cap \mathcal{U}(u_{3}) \neq \emptyset$ and $|\mathcal{A}(u_{3})| \geq 2$. Recall that $\mathcal{A}(w_{1}), \mathcal{A}(u_{2})$ and $\mathcal{A}(u_{3})$ are disjoint, thus $\mathcal{U}(w_{1}) \supseteq \mathcal{A}(u_{2}) \cup \mathcal{A}(u_{3}) \cup \{1, 3, \kappa - \Delta\}$. Moreover, $\mathcal{U}(w_{1}) = \mathcal{A}(u_{2}) \cup \mathcal{A}(u_{3}) \cup \{1, 3, \kappa - \Delta\}$, $\mathcal{A}(u_{3}) = \mathcal{A}(w_{2})$. If there exists a $\xi_{i} \notin \mathcal{U}(u_{4})$, then there exists a $(\kappa - \Delta, \xi_{i}, u, u_{3})$-critical path, and then reassigning $\xi_{i}$ to $uu_{4}$ and $\kappa - \Delta + 1$ to $uw$ results in an acyclic edge coloring of $G$. So we have that $\mathcal{A}(u_{2}) \cup \mathcal{A}(u_{3}) \subseteq \mathcal{U}(u_{4})$.

There exists a $(\kappa - \Delta, \mu_{1}, u, u_{2})$- or $(\kappa - \Delta + 1, \mu_{1}, u, u_{2})$-critical path; otherwise, reassigning $\mu_{1}$ to $uu_{2}$ and $\zeta_{1}^{*}$ to $uw$ results in an acyclic edge coloring of $G$. If there exists a $(\kappa - \Delta, \mu_{1}, u, u_{2})$-critical path, then $\mu_{1} = 3$ and $\varepsilon_{1} = \kappa - \Delta$; but reassigning $\mu_{1}, \alpha^{*}$ and $\zeta_{1}^{*}$ to $uu_{2}, uw_{1}$ and $uw$ results in an acyclic edge coloring. So there exists a $(\kappa - \Delta + 1, \mu_{1}, u, u_{2})$-critical path, thus $\varepsilon_{1} = \kappa - \Delta + 1$ and $\mathcal{U}(u_{4}) \cap (\mathcal{U}(w) \cup \mathcal{U}(u)) = \{\mu_{1}, \kappa - \Delta + 1\}$. Now, reassigning $\kappa - \Delta, \mu_{1}, \alpha^{*}$ and $\zeta_{1}^{*}$ to $uu_{4}, uu_{2}, uw_{1}$ and $uw$ respectively, yields an acyclic edge coloring of $G$.

\begin{subsubcase}\label{2Common3}%
The color on $w_{1}w$ is not a common color and the color on $w_{1}u$ is a common color.
\end{subsubcase}
Without loss of generality, assume that $\phi(uw_{1}) = 3$, $\phi(uu_{2}) = 2$, $\phi(uu_{3}) = \kappa - \Delta$ and $\phi(uu_{4}) = \kappa - \Delta + 1$.

For every color $\alpha_{i} \in \mathcal{A}(w_{1})$, there exists a $(\theta_{i}, \alpha_{i}, u, w_{1})$-critical path with some $\theta_{i} \in \{\kappa - \Delta, \kappa - \Delta + 1\}$; otherwise, reassigning $\alpha_{i}$ to $uw_{1}$ will take us back to Case~\ref{1Common}. If $\Upsilon(uu_{2}) \subseteq C(wu)$, then reassigning $1$ to $uu_{2}$ will take us back to Subcase~\ref{2Common1}. So we have that $\Upsilon(uu_{2}) \nsubseteq C(wu)$ and $\mathcal{A}(u_{2}) \neq \emptyset$. Consequently, for every color $\beta_{i} \in \mathcal{A}(u_{2})$, there exists a $(\varepsilon_{i}, \beta_{i}, u, u_{2})$-critical path with some $\varepsilon_{i} \in \{\kappa - \Delta, \kappa - \Delta + 1\}$; otherwise, reassigning $\beta_{i}$ to $uu_{2}$ will take us back to Case~\ref{1Common}. Hence, we have $\{\kappa - \Delta, \kappa - \Delta + 1\} \cap \Upsilon(uu_{2}) \neq \emptyset$.

\begin{subsubsubcase}%
Suppose that $\{\kappa - \Delta, \kappa - \Delta + 1\} \subseteq \mathcal{U}(w_{1})$.
\end{subsubsubcase}
If $\{\kappa - \Delta, \kappa - \Delta + 1\} \subseteq \mathcal{U}(u_{2})$, then $\mathcal{U}(w_{1}) = \{1, 3, \kappa - \Delta, \kappa - \Delta + 1\} \cup \mathcal{A}(u_{2})$ and $\mathcal{U}(u_{2}) \cap (\mathcal{U}(w) \cup \mathcal{U}(u)) = \{\kappa - \Delta, \kappa - \Delta + 1, 2\}$; but reassigning $\alpha_{1}$ to $ww_{1}$ and $1$ to $uw$ results in an acyclic edge coloring of $G$. This implies that $|\{\kappa - \Delta, \kappa - \Delta + 1\} \cap \mathcal{U}(u_{2})| = 1$, say $\kappa - \Delta \in \mathcal{U}(u_{2})$. Hence, we have $\varepsilon_{i} = \kappa - \Delta$ and $\mathcal{A}(u_{2}) \subseteq \mathcal{U}(u_{3})$.

   Suppose that there exists no $(2, 1, u, w)$-critical path. Thus, there exists a $(\mu_{i}, \alpha_{i}, w, w_{1})$-critical path with $\mu_{i} \in \mathcal{U}(w) \setminus \{1, 2, 3\}$; otherwise, reassigning $\alpha_{i}$ to $ww_{1}$ and $1$ to $wu$ will result in an acyclic edge coloring of $G$. Note that $\mathcal{U}(w_{1}) \supseteq \{1, 3, \kappa - \Delta, \kappa - \Delta + 1, \mu_{1}\} \cup \mathcal{A}(u_{2})$, it follows that $\mathcal{U}(u_{2}) \cap (\mathcal{U}(u) \cup \mathcal{U}(w)) = \{2, \kappa - \Delta\}$ and $|\mathcal{U}(u_{2}) \cap C(wu)| = \Delta - 2$. Moreover, $\mathcal{U}(w_{1}) = \{1, 3, \kappa - \Delta, \kappa - \Delta + 1, \mu_{1}\} \cup \mathcal{A}(u_{2})$ and $|\mathcal{A}(u_{2})| = 1$, say $\mu_{1} = 4$. Thus, there exists a $(\kappa - \Delta, 1, u, u_{2})$-critical path; otherwise, reassigning $1$ to $uu_{2}$ will take us back to Subcase~\ref{2Common1}. So, we have $1 \in \mathcal{U}(u_{3})$. Furthermore, there exists a $(\kappa - \Delta, 4, u, u_{2})$-critical path; otherwise, reassigning $4$ to $uu_{2}$ and $\alpha_{1}$ to $wu$ results in an acyclic edge coloring of $G$. Hence, $\{1, 4, \kappa - \Delta\} \subseteq \mathcal{U}(u_{3})$. Recall that $|\mathcal{A}(u_{3})| \geq 2$ and $|\mathcal{A}(u_{2})| = 1$, it follows that $\mathcal{A}(w_{1}) \cap \mathcal{A}(u_{3}) \neq \emptyset$, say $\alpha_{1} \notin \mathcal{U}(u_{3})$. If $1 \notin \mathcal{U}(u_{4})$, then reassigning $1$ to $uu_{4}$ and $\alpha_{1}$ to $uw_{1}$ will take us back to Subcase~\ref{2Common2}. Thus, we have $1 \in \mathcal{U}(u_{4})$. If $2 \notin \mathbb{S}$, then reassigning $2, \beta_{1}$ and $\alpha_{1}$ to $uu_{3}, uu_{2}$ and $uw$ respectively results in an acyclic edge coloring of $G$. Thus $2 \in \mathbb{S}$. If $3 \notin \mathbb{S}$, then reassigning $3, \alpha_{1}$ and $\beta_{1}$ to $uu_{4}, uw_{1}$ and $uw$ respectively results in an acyclic edge coloring of $G$. Thus $3 \in \mathbb{S}$. If $5 \notin \mathbb{S}$, then there exists a $(5, \alpha_{1}, w, u_{2})$-alternating path, otherwise, reassigning $5$ to $uu_{2}$ and $\alpha_{1}$ to $uw$ results in an acyclic edge coloring of $G$; but reassigning $5, \beta_{1}$ and $\alpha_{1}$ to $uu_{3}, uu_{2}$ and $uw$ results in an acyclic edge coloring of $G$. Thus $5 \in \mathbb{S}$. Similarly, $\{5, 6, \dots, \deg_{G}(w) - 1\} \subseteq \mathbb{S}$. Therefore, we have
\[
\sum_{\theta \in \mathcal{U}(w) \cup \mathcal{U}(u)} \mul_{\mathbb{S}}(\theta) + |X| \geq 3|\{1\}| + 2|\{4, \kappa - \Delta\}| + |\{2, 3, \kappa - \Delta + 1, 5, 6, \dots, \deg_{G}(w) - 1\}| = \deg_{G}(w) + 5.
\]

   Suppose that there exists a $(2, 1, w, u)$-critical path. It follows that $\{1, 2, \kappa - \Delta\} \subseteq \mathcal{U}(u_{2}) \cap (\mathcal{U}(w) \cup \mathcal{U}(u))$. It is obvious that $\mathcal{A}(u_{2}) \subseteq \mathcal{U}(w_{1})$, thus $\mathcal{U}(w_{1}) = \{1, 3, \kappa - \Delta, \kappa - \Delta + 1\} \cup \mathcal{A}(u_{2})$, $\mathcal{U}(u_{2}) \cap (\mathcal{U}(w) \cup \mathcal{U}(u)) = \{1, 2, \kappa - \Delta\}$ and $|\mathcal{U}(u_{2}) \cap C(wu)| = \Delta - 3$. If $\mathcal{A}(w_{1}) \subseteq \mathcal{U}(u_{3})$, then $\mathcal{U}(u_{3}) = \mathcal{A}(w_{1}) \cup \mathcal{A}(u_{2}) \cup \{\kappa - \Delta\} = C(wu) \cup \{\kappa - \Delta\}$; but reassigning $2, \beta_{1}$ and $\alpha_{1}$ to $uu_{3}, uu_{2}$ and $uw$ respectively results in an acyclic edge coloring of $G$. So we may assume that $\mathcal{A}(w_{1}) \nsubseteq \mathcal{U}(u_{3})$ and $\alpha_{1} \notin \mathcal{U}(u_{3})$. If $3 \notin \mathbb{S}$, then reassigning $3, \alpha_{1}$ and $\beta_{1}$ to $uu_{4}, uw_{1}$ and $uw$ respectively results in an acyclic edge coloring of $G$. Thus, we have $3 \in \mathbb{S}$. For every color $\theta$ in $\mathcal{U}(w) \setminus \{3\}$, we have that $\theta \in \mathbb{S}$; otherwise, reassigning $\theta, \beta_{1}$ and $\alpha_{1}$ to $uu_{3}, uu_{2}$ and $uw$ respectively results in an acyclic edge coloring of $G$. If $1 \notin \mathcal{U}(u_{3}) \cup \mathcal{U}(u_{4})$, then reassigning $1$ to $uu_{4}$ and $\alpha_{1}$ to $uw_{1}$ will take us back to Subcase~\ref{2Common1}. Hence, the color $1$ appears exactly three times in $\mathbb{S}$.  If $\kappa - \Delta + 1$ appears at least twice in $\mathbb{S}$ or $|X| \geq 1$, then
\[
\sum_{\theta \in \mathcal{U}(w) \cup \mathcal{U}(u)} \mul_{\mathbb{S}}(\theta) + |X| \geq \deg_{G}(w) + 5.
\]
So we may assume that $\kappa - \Delta + 1$ appears precisely once (at $w_{1}$) and $X = \emptyset$. Note that $\beta_{1} \notin \mathcal{U}(u_{4})$. But reassigning $\beta_{1}$ to $uu_{4}$ and $\kappa - \Delta + 1$ to $uu_{2}$ will take us back to Case \ref{1Common}.

\begin{subsubsubcase}%
Now, we may assume that $\{\kappa - \Delta, \kappa - \Delta + 1\} \nsubseteq \mathcal{U}(w_{1})$ and $\kappa - \Delta  + 1 \notin \mathcal{U}(w_{1})$.
\end{subsubsubcase}

Thus, there exists a $(\kappa - \Delta, \alpha_{i}, u, w_{1})$-critical path for every $\alpha_{i}$; otherwise, reassigning $\alpha_{i}$ to $uw_{1}$ will take us back to Case~\ref{1Common}. It follows that $\kappa - \Delta \in \mathcal{U}(w_{1})$ and $\mathcal{A}(w_{1}) \subseteq \mathcal{U}(u_{3}) \cap \mathcal{U}(u_{2})$. If $\Upsilon(uu_{3}) \subseteq C(uw)$, then reassigning $\alpha_{1}$ to $uw_{1}$ and $1$ to $uu_{3}$ will take us back to Subcase~\ref{2Common2}. So we may assume that $\Upsilon(uu_{3}) \nsubseteq C(wu)$ and $C(wu) \nsubseteq \Upsilon(uu_{3})$.

(1) Suppose that $\mathcal{A}(u_{2}) \cap \mathcal{A}(u_{3}) = \emptyset$. It follows that $\mathcal{A}(w_{1}), \mathcal{A}(u_{2})$ and $\mathcal{A}(u_{3})$ are pairwise disjoint. Suppose that there exists no $(2, 1, u, w)$-critical path. Thus, there exists a $(\tau, \alpha_{1}, w, w_{1})$-critical path, where $\tau \in \mathcal{U}(w) \setminus \{1, 2, 3\}$; otherwise, reassigning $1$ to $uw$ and $\alpha_{1}$ to $ww_{1}$ results in an acyclic edge coloring of $G$. Since $\mathcal{U}(u_{3}) \supseteq \mathcal{A}(w_{1}) \cup \mathcal{A}(u_{2})$ and $C(wu) \nsubseteq \mathcal{U}(u_{3})$, it follows that $|\mathcal{A}(w_{1})| = \Delta - 3$, $|\mathcal{A}(u_{2})| = 1$ and $|\mathcal{U}(u_{3}) \cap (\mathcal{U}(w) \cup \mathcal{U}(u))| = 2$. If $1 \notin \mathcal{U}(u_{3})$, then there exists a $(\kappa - \Delta + 1, 1, u, u_{3})$-critical path; otherwise, reassigning $1$ to $uu_{3}$ and $\alpha_{1}$ to $uw_{1}$ will take us back to Subcase~\ref{2Common2}. Thus, $\Upsilon(uu_{3}) \cap (\mathcal{U}(w) \cup \mathcal{U}(u)) = \{1\}$ or $\{\kappa - \Delta + 1\}$. If there exists no $(\kappa - \Delta + 1, 3, u, u_{3})$-critical path, then reassigning $3, \alpha_{1}$ and $\beta_{1}$ to $uu_{3}, uw_{1}$ and $uw$ results in an acyclic edge coloring of $G$. Hence, there exists a $(\kappa - \Delta + 1, 3, u, u_{3})$-critical path and $\Upsilon(uu_{3}) \cap (\mathcal{U}(w) \cup \mathcal{U}(u)) = \{\kappa - \Delta + 1\}$. But reassigning $1$ to $uu_{2}$ will take us back to Subcase~\ref{2Common1}.

   Now, we consider the other subcase: suppose that there exists a $(2, 1, u, w)$-critical path and $1 \in \mathcal{U}(u_{2})$. Since $\mathcal{U}(u_{3}) \supseteq \mathcal{A}(w_{1}) \cup \mathcal{A}(u_{2})$ and $C(wu) \nsubseteq \mathcal{U}(u_{3})$, so we have that $|\mathcal{A}(w_{1})| = \Delta - 4$, $|\mathcal{A}(u_{2})| = 2$ and $\mathcal{U}(w_{1}) \cap (\mathcal{U}(w) \cup \mathcal{U}(u)) = \{1, 3, \kappa - \Delta\}$, $\mathcal{U}(u_{2}) \cap (\mathcal{U}(w) \cup \mathcal{U}(u)) = \{1, 2, \varepsilon_{1}\}$ and $|\mathcal{U}(u_{3}) \cap (\mathcal{U}(w) \cup \mathcal{U}(u))| = 2$. If $1 \in \mathcal{U}(u_{3})$, then $\mathcal{U}(u_{3}) \cap (\mathcal{U}(w) \cup \mathcal{U}(u)) = \{1, \kappa - \Delta\}$, and then reassigning $\beta_{1}, \alpha_{1}$ and $3$ to $uw, uw_{1}$ and $uu_{3}$ results in an cyclic edge coloring. Thus, $1 \notin \mathcal{U}(u_{3})$. There exists a $(\kappa - \Delta + 1, 1, u, u_{3})$-critical path; otherwise, reassigning $1$ to $uu_{3}$ and $\alpha_{1}$ to $uw_{1}$ will take us back to Subcase~\ref{2Common2}. This implies that $\mathcal{U}(u_{3}) \cap (\mathcal{U}(w) \cup \mathcal{U}(u)) = \{\kappa - \Delta, \kappa - \Delta + 1\}$ and $1$ appears three times in $\mathbb{S}$. There exists a $(\kappa - \Delta + 1, 3, u, u_{3})$-critical path, otherwise, reassigning $\beta_{1}, \alpha_{1}$ and $3$ to $uw, uw_{1}$ and $uu_{3}$ results in an acyclic edge coloring of $G$. Now, we have $\{1, 3\} \subseteq \Upsilon(uu_{4})$. If $\beta \in \mathcal{A}(u_{2}) \cap \mathcal{A}(u_{4})$, then $\varepsilon_{1} = \kappa - \Delta$ and there exists a $(\kappa - \Delta, \beta, u, u_{2})$-critical path; but reassigning $\beta$ to $uu_{4}$ and $\kappa - \Delta + 1$ to $uw$ results in an acyclic edge coloring of $G$. This implies that $\mathcal{A}(u_{2}) \subseteq \mathcal{U}(u_{4})$ and $\mathcal{A}(u_{2}) \subseteq X$. Suppose that $\{4, 5, \dots, \deg_{G}(w) - 1\} \nsubseteq \mathbb{S}$. So, by symmetry, we may assume that $4 \notin \mathbb{S}$. There exists a $(4, \beta_{1}, w, w_{1})$-alternating path; otherwise, reassigning $4$ to $uw_{1}$ and $\beta_{1}$ to $uw$ results in an acyclic edge coloring of $G$. But reassigning $4, \alpha_{1}$ and $\beta_{1}$ to $uu_{3}, uw_{1}$ and $uw$ results in an acyclic edge coloring of $G$. Hence, $\{3, 4, \dots, \deg_{G}(w) - 1\} \subseteq \mathbb{S}$.

\[
\sum_{\theta \in \mathcal{U}(w) \cup \mathcal{U}(u)} \mul_{\mathbb{S}}(\theta) + |X| \geq |\{3, 4, \dots, \deg_{G}(w) - 1\}| + 3|\{1\}| + |\{\varepsilon_{1}\}| + |\{\kappa - \Delta\}| + |\{\kappa - \Delta + 1\}| + |\mathcal{A}(u_{2})| \geq \deg_{G}(w) + 5.
\]

(2) So we may assume that $\mathcal{A}(u_{2}) \cap \mathcal{A}(u_{3}) \neq \emptyset$, say $\beta_{1} \in \mathcal{A}(u_{2}) \cap \mathcal{A}(u_{3})$. Thus, there exists a $(\kappa - \Delta + 1, \beta_{1}, u, u_{2})$-critical path; otherwise, reassigning $\beta_{1}$ to $uu_{2}$ will take us back to Case \ref{1Common}. So, we have $\kappa - \Delta + 1 \in \mathcal{U}(u_{2})$.

   Suppose that the color $1$ only appears once (at $w_{1}$) in $\mathbb{S}$. If there exists no $(3, 1, u, u_{2})$-critical path, then reassigning $1$ to $uu_{2}$ will take us back to Subcase~\ref{2Common1}. But if there exists a $(3, 1, u, u_{2})$-critical path, then reassigning $1$ to $uu_{4}$ and $\beta_{1}$ to $uu_{2}$ will take us back to Subcase~\ref{2Common1} again. Hence, the color $1$ appears at least twice in $\mathbb{S}$.

   If $2 \notin \mathbb{S}$, then reassigning $2, \beta_{1}$ and $\alpha_{1}$ to $uu_{4}, uu_{2}$ and $uw$ respectively results in an acyclic edge coloring of $G$. If $3 \notin \mathbb{S}$, then reassigning $3, \alpha_{1}$ and $\beta_{1}$ to $uu_{3}, uw_{1}$ and $uw$ respectively, results in an acyclic edge coloring of $G$. Suppose that $4 \notin \mathbb{S}$. There exists a $(4, \beta_{1}, w, w_{1})$-alternating path; otherwise, reassigning $4$ to $uw_{1}$ and $\beta_{1}$ to $uw$ results in an acyclic edge coloring of $G$. Now, reassigning $4, \alpha_{1}$ and $\beta_{1}$ to $uu_{3}, uw_{1}$ and $uw$ respectively results in an acyclic edge coloring of $G$. Thus, $\{2, 3, 4\} \subseteq \mathbb{S}$. By symmetry, we have that $\mathcal{U}(w) \setminus \{1, 2, 3\} \subseteq \mathbb{S}$.

   Suppose that $\kappa - \Delta$ appears only once (at $w_{1}$) in $\mathbb{S}$. Thus, there exists a $(3, \kappa - \Delta, u, w)$-critical path; otherwise, reassigning $\kappa - \Delta$ to $uw$ and $\beta_{1}$ to $uu_{3}$ results in an acyclic edge coloring of $G$. But reassigning $\beta_{1}$ to $uu_{3}$ and $\kappa - \Delta$ to $uu_{2}$ will take us back to Case 1. Hence, the color $\kappa - \Delta$ appears at least twice in $\mathbb{S}$.

   Note that $|\mathcal{A}(w_{1})| \geq 2$. If $\mathcal{A}(w_{1}) \subseteq \mathcal{U}(u_{4})$, then $ \mathcal{A}(w_{1}) \subseteq X$, and then
\[
\sum_{\theta \in \mathcal{U}(w) \cup \mathcal{U}(u)} \mul_{\mathbb{S}}(\theta) + |X| \geq |\{\kappa - \Delta + 1, 2, 3, \dots, \deg_{G}(w) - 1\}| + 2|\{1, \kappa - \Delta\}| + |\mathcal{A}(w_{1})| \geq \deg_{G}(w) + 5.
\]

   So we may assume that $\mathcal{A}(w_{1}) \nsubseteq \mathcal{U}(u_{4})$, say $\alpha_{1} \notin \mathcal{U}(u_{4})$. There exists a $(2, \kappa - \Delta + 1, w, u)$-critical path; otherwise, reassigning $\alpha_{1}$ to $uu_{4}$ and $\kappa - \Delta + 1$ to $uw$ results in an acyclic edge coloring of $G$. Consequently, there exists a $(\kappa - \Delta, \kappa - \Delta + 1, u, w_{1})$-critical path and $\kappa - \Delta + 1 \in \mathcal{U}(u_{3})$; otherwise, reassigning $\alpha_{1}$ to $uu_{4}$ and $\kappa - \Delta + 1$ to $uw_{1}$ will take us back to Case 1. Hence, the color $\kappa - \Delta + 1$ appears exactly twice in $\mathbb{S}$.

   Suppose that there exists no $(2, 1, u, w)$-critical path. Thus, there exists a $(\tau, \alpha_{1}, w, w_{1})$-critical path with $\tau \in \mathcal{U}(w) \setminus \{1, 2, 3\}$; otherwise, reassigning $1$ to $uw$ and $\alpha_{1}$ to $ww_{1}$ results in an acyclic edge coloring of $G$. If $\tau$ only appears once (at $w_{1}$) in $\mathbb{S}$, then reassigning $\tau, \alpha_{1}$ and $\beta_{1}$ to $uu_{3}, uw_{1}$ and $uw$ respectively results in an acyclic edge coloring of $G$. Hence, the color $\tau$ appears at least twice in $\mathbb{S}$. Hence,
\[
\sum_{\theta \in \mathcal{U}(w) \cup \mathcal{U}(u)} \mul_{\mathbb{S}}(\theta) + |X| \geq |\{2, 3, \dots, \deg_{G}(w) - 1\}| + 2|\{1, \kappa - \Delta, \kappa - \Delta + 1\}| + |\{\tau\}|= \deg_{G}(w) + 5.
\]

Suppose there exists a $(2, 1, u, w)$-critical path and $1 \in \mathcal{U}(u_{2})$. If $1 \notin \mathcal{U}(u_{3}) \cup \mathcal{U}(u_{4})$, then reassigning $1$ to $uu_{3}$ and $\alpha_{1}$ to $uw_{1}$ will take us back to Subcase~\ref{2Common1}. Hence, the color $1$ appears at least three times in $\mathbb{S}$,
\[
\sum_{\theta \in \mathcal{U}(w) \cup \mathcal{U}(u)} \mul_{\mathbb{S}}(\theta) + |X| \geq |\{2, 3, \dots, \deg_{G}(w) - 1\}| + 3|\{1\}| + 2|\{\kappa - \Delta, \kappa - \Delta + 1\}| = \deg_{G}(w) + 5.
\]

\begin{subsubcase}%
Neither the color on $w_{1}w$ nor the color on $w_{1}u$ is a common color.
\end{subsubcase}
By symmetry, assume that $\phi(uw_{1}) = \kappa - \Delta$, $\phi(uu_{2}) = 2$, $\phi(uu_{3}) = 3$ and $\phi(uu_{4}) = \kappa - \Delta + 1$.

If $\Upsilon(uu_{2}) \subseteq C(wu)$, then reassigning $1$ to $uu_{2}$ will take us back to Subcase~\ref{2Common2}. This implies that $\Upsilon(uu_{2}) \nsubseteq C(wu)$ and $\mathcal{A}(u_{2}) \neq \emptyset$. Thus, there exists a $(\varepsilon_{i}, \beta_{i}, u, u_{2})$-critical path with $\varepsilon_{i} \in \{\kappa - \Delta, \kappa - \Delta + 1\}$; otherwise, reassigning $\beta_{i}$ to $uu_{2}$ will take us back to Case~\ref{1Common}. Similarly, we have that $\Upsilon(uu_{3}) \nsubseteq C(wu)$ and $\mathcal{A}(u_{3}) \neq \emptyset$, and thus there exists a $(m_{i}, \xi_{i}, u, u_{3})$-critical path with $m_{i} \in \{\kappa - \Delta, \kappa - \Delta + 1\}$. If $1 \notin \mathcal{U}(u_{2}) \cup \mathcal{U}(u_{3})$, then reassigning $1$ to $uu_{2}$ will create a $(1, \kappa - \Delta + 1)$-dichromatic cycle containing $uu_{2}$, otherwise, it will take us back to Subcase~\ref{2Common2}; but reassigning $1$ to $uu_{3}$ will take us back to Subcase~\ref{2Common2} again. It follows that $1 \in \mathcal{U}(u_{2}) \cup \mathcal{U}(u_{3})$ and $1$ appears at least twice in $\mathbb{S}$.

\begin{subsubsubcase}%
Suppose that $\mathcal{A}(u_{2}) \cup \mathcal{A}(u_{3}) \nsubseteq \mathcal{U}(w_{1})$ and $\beta_{1} = \alpha_{1} \notin \mathcal{U}(w_{1})$.
\end{subsubsubcase}
Hence, there exists a $(3, \beta_{1}, u, w)$-critical path and $(\kappa - \Delta + 1, \beta_{1}, u, u_{2})$-critical path, thus $\kappa - \Delta + 1 \in \mathcal{U}(u_{2})$.

There exists a $(2, \kappa - \Delta, u, w)$- or $(3, \kappa - \Delta, u, w)$-critical path; otherwise, reassigning $\beta_{1}$ to $uw_{1}$ and $\kappa - \Delta$ to $uw$ results in an acyclic edge coloring of $G$. It follows that $\kappa - \Delta \in \mathcal{U}(u_{2}) \cup \mathcal{U}(u_{3})$. Moreover, $\kappa - \Delta$ appears at least twice in $\mathbb{S}$; otherwise, assume that $\kappa - \Delta$ only appears at $u_{2}$, thus reassigning $\kappa - \Delta$ to $uu_{3}$ and $\beta_{1}$ to $uw_{1}$ will take us back to Case \ref{1Common}.

If $2 \notin \mathbb{S}$, then reassigning $\beta_{1}, 2$ and $\xi_{1}$ to $uu_{2}, uu_{4}$ and $uw$ respectively results in an acyclic edge coloring of $G$. Thus $2 \in \mathbb{S}$.

Suppose that $4 \notin \mathbb{S}$. There exists a $(4, \xi_{1}, w, u_{2})$-alternating path for every $\xi_{i} \in \mathcal{A}(u_{3})$; otherwise, reassigning $4$ to $uu_{2}$ and $\xi_{1}$ to $uw$ results in an acyclic edge coloring of $G$. Now, reassigning $\beta_{1}, 4$ and $\xi_{1}$ to $uu_{2}, uu_{4}$ and $uw$ respectively results in an acyclic edge coloring of $G$ again. Hence, the color $4$ appears in $\mathbb{S}$. Similarly, we can prove that $\mathcal{U}(w) \setminus \{1, 2, 3\} \subseteq \mathbb{S}$.

Suppose that $3 \notin \mathbb{S}$. If there exists no $(\kappa - \Delta + 1, \xi_{i}, u, u_{3})$-critical path, then reassigning $3$ to $uw_{1}$ and $\xi_{i}$ to $uu_{3}$ will take us back to Subcase~\ref{2Common3}. Hence, there exists a $(\kappa - \Delta + 1, \xi_{i}, u, u_{3})$-critical path for every $\xi_{i} \in \mathcal{A}(u_{3})$, and then $\kappa - \Delta + 1 \in \mathcal{U}(u_{3})$ and $\mathcal{A}(u_{3}) \subseteq \mathcal{U}(u_{4})$. If there exists no $(\kappa - \Delta, \xi_{i}, u, u_{3})$-critical path, then reassigning $3, \xi_{i}$ and $\beta_{1}$ to $uu_{4}, uu_{3}$ and $uw$ respectively, results in an acyclic edge coloring of $G$. Hence, both $(\kappa - \Delta, \xi_{i}, u, u_{3})$- and $(\kappa - \Delta + 1, \xi_{i}, u, u_{3})$-critical path exist for every $\xi_{i} \in \mathcal{A}(u_{3})$, and then $\{\kappa - \Delta, \kappa - \Delta + 1\} \subseteq \mathcal{U}(u_{3})$ and $\mathcal{A}(u_{3}) \subseteq \mathcal{U}(w_{1}) \cap \mathcal{U}(u_{4})$. Clearly, every color in $\mathcal{A}(u_{3})$ appears precisely three times in $\mathbb{S}$. Therefore,

\[
\sum_{\theta \in \mathcal{U}(w) \cup \mathcal{U}(u)} \mul_{\mathbb{S}}(\theta) + |X| \geq |\{2\} \cup \{4, \dots, \deg_{G}(w) - 1\}| + 2|\{1, \kappa - \Delta, \kappa - \Delta + 1\}| + |\mathcal{A}(u_{3})| \geq \deg_{G}(w) + 5.
\]
So, in the following, we may assume that $3 \in \mathbb{S}$.

If there exists a $(2, 1, u, w)$-critical path (or $(3, 1, u, w)$-critical path) and $1$ appears only twice in $\mathbb{S}$, then reassigning $1$ to $uu_{3}$ (to $uu_{2}$) will take us back to Subcase~\ref{2Common2}. In other words, if there exists a $(2, 1, u, w)$-critical path or $(3, 1, u, w)$-critical path, then the color $1$ appears at least three times in $\mathbb{S}$.

Suppose that neither $(2, 1, u, w)$-critical path nor $(3, 1, u, w)$-critical path exists. If there exists no $(\tau, \beta_{1}, w, w_{1})$-critical path with some $\tau \in \mathcal{U}(w) \setminus \{1, 3\}$, then reassigning $1$ to $uw$ and $\beta_{1}$ to $ww_{1}$ results in an acyclic edge coloring of $G$. Hence, there exists a $(\tau, \beta_{1}, w, w_{1})$-critical path with some $\tau \in \mathcal{U}(w) \setminus \{1, 3\}$. Suppose that there exists a $(2, \beta_{1}, w, w_{1})$-critical path and $2$ appears only once in $\mathbb{S}$. This implies that there exists a $(\kappa - \Delta, 2, u, u_{4})$-critical path; otherwise reassigning $2, \beta_{1}$ and $\xi_{1}$ to $uu_{4}, uu_{2}$ and $uw$ respectively results in an acyclic edge coloring of $G$. But reassigning $2, \beta_{1}, \xi_{1}$ and $\zeta^{*}$ ($\zeta^{*} = \beta_{2}$ if $|\mathcal{A}(u_{2})| \geq 2$, otherwise, $\zeta^{*} = \kappa - \Delta$) to $uu_{4}, uw_{1}, uw$ and $uu_{2}$ respectively, and we obtain an acyclic edge coloring of $G$. Thus, if there exists a $(2, \beta_{1}, w, w_{1})$-critical path, then the color $2$ appears at least twice in $\mathbb{S}$. Suppose that there exists a $(4, \beta_{1}, w, w_{1})$-critical path and $4$ only appears once in $\mathbb{S}$. Hence, there is a $(\kappa - \Delta, 4, u, u_{3})$-critical path, otherwise, reassigning $4$ to $uu_{3}$ and $\beta_{1}$ to $uw$ results in an acyclic edge coloring of $G$. Now, reassigning $4, \beta_{1}$ and $\xi_{1}$ to $uu_{4}, uu_{2}$ and $uw$ will create a $(4, \xi_{1})$-dichromatic cycle containing $uw$; otherwise, the resulting coloring is an acyclic edge coloring of $G$. But reassigning $4$ to $uu_{2}$ and $\xi_{1}$ to $uw$ results in an acyclic edge coloring of $G$. Thus, if there exists a $(4, \beta_{1}, w, w_{1})$-critical path, then the color $4$ appears at least twice in $\mathbb{S}$. Similarly, if there exists a $(\tau, \beta_{1}, w, w_{1})$-critical path with $\tau \geq 4$, then the color $\tau$ appears at least twice in $\mathbb{S}$. Therefore, the color $\tau$ appears at least twice in $\mathbb{S}$.

By the above arguments, regardless of the existence of $(2, 1, u, w)$-critical path or $(3, 1, u, w)$-critical path, if $\kappa - \Delta + 1$ appears at least twice or $|X| \geq 1$, then

\[
\sum_{\theta \in \mathcal{U}(w) \cup \mathcal{U}(u)} \mul_{\mathbb{S}}(\theta) + |X| \geq \deg_{G}(w) + 5.
\]

So we may assume that the color $\kappa - \Delta + 1$ appears only once (at $u_{2}$) in $\mathbb{S}$ and $X = \emptyset$. If $\mathcal{A}(u_{3}) \nsubseteq \mathcal{U}(w_{1})$, say $\xi_{1} \notin \mathcal{U}(w_{1})$, then there exists a $(2, \xi_{1}, u, w)$-critical path and $(\kappa - \Delta + 1, \xi_{1}, u, u_{3})$-critical path, and then $\kappa - \Delta + 1 \in \mathcal{U}(u_{3})$, a contradiction. So we may assume that $\mathcal{A}(u_{3}) \subseteq \mathcal{U}(w_{1})$ and $\mathcal{A}(u_{3}) \cap \mathcal{U}(u_{4}) = \emptyset$.

Clearly, there exists a $(2, \xi_{1}, u, w)$-critical path. Thus, there exists a $(\kappa - \Delta, \xi_{1}, u, u_{3})$-critical path; otherwise, reassigning $\xi_{1}$ to $uu_{3}$ will take us back to Case \ref{1Common}. Hence, there exists a $(2, \kappa - \Delta + 1, u, w)$-critical path; otherwise, reassigning $\xi_{1}$ to $uu_{4}$ and $\kappa - \Delta + 1$ to $uw$ will result in an acyclic edge coloring of $G$. Now, reassigning $\xi_{1}$ to $uu_{4}$ and $\kappa - \Delta + 1$ to $uu_{3}$ will take us back to Case \ref{1Common}.

\begin{subsubsubcase}%
$\mathcal{A}(u_{2}) \cup \mathcal{A}(u_{3}) \subseteq \mathcal{U}(w_{1})$.
\end{subsubsubcase}
Firstly, suppose that $\mathcal{A}(u_{2}) \cup \mathcal{A}(u_{3}) \nsubseteq \mathcal{U}(u_{4})$ and $\beta_{1} = \zeta_{1} \notin \mathcal{U}(u_{4})$. Hence, there exists a $(3, \beta_{1}, u, w)$-critical path and a $(\kappa - \Delta, \beta_{1}, u, u_{2})$-critical path, and then $\kappa - \Delta \in \mathcal{U}(u_{2})$.

If $\{2, 3, \kappa - \Delta + 1\} \cap \mathcal{U}(w_{1}) = \emptyset$, then reassigning $\beta_{1}$ to $uu_{2}$ and $2$ to $uw_{1}$ will take us back to Subcase \ref{2Common3}. Hence, $\{2, 3, \kappa - \Delta + 1\} \cap \mathcal{U}(w_{1}) \neq \emptyset$. Recall that $1 \in \mathcal{U}(u_{2}) \cup \mathcal{U}(u_{3})$. Since $\mathcal{A}(u_{2}) \cup \mathcal{A}(u_{3}) \subseteq \mathcal{U}(w_{1})$, it follows that $\deg_{G}(w_{1}) = 6$, $\deg_{G}(w) = \kappa - \Delta$, and $|\mathcal{A}(u_{2})| + |\mathcal{A}(u_{3})| = 3$. Furthermore, we have that $\{\kappa - \Delta, \kappa - \Delta + 1\} \cap \mathcal{U}(u_{2}) = \{\kappa - \Delta\}$ and $|\{\kappa - \Delta, \kappa - \Delta + 1\} \cap \mathcal{U}(u_{3})| = 1$. Thus, there exists a $(3, \kappa - \Delta + 1, w, u)$-critical path; otherwise reassigning $\beta_{1}$ to $uu_{4}$ and $\kappa- \Delta + 1$ to $uw$ will result in an acyclic edge coloring of $G$. Hence, $\{\kappa - \Delta, \kappa - \Delta + 1\} \cap \mathcal{U}(u_{3}) = \{\kappa - \Delta + 1\}$.

If $1 \notin \mathcal{U}(u_{2})$, then $\mathcal{U}(u_{2}) \cap (\mathcal{U}(w) \cup \mathcal{U}(u)) = \{2, \kappa - \Delta\}$, but reassigning $1$ to $uu_{2}$ will take us back to Subcase \ref{2Common2}. This implies that $\mathcal{U}(u_{2}) \cap (\mathcal{U}(w) \cup \mathcal{U}(u)) = \{1, 2, \kappa - \Delta\}$ and $\mathcal{U}(u_{3}) \cap (\mathcal{U}(w) \cup \mathcal{U}(u)) = \{3, \kappa - \Delta + 1\}$. Now, there is a $(\kappa - \Delta + 1, 1, u, u_{3})$-critical path; otherwise, reassigning $1$ to $uu_{3}$ will take us back to Subcase \ref{2Common2}. Thus, there exists a $(\kappa- \Delta, \kappa- \Delta + 1, u, u_{2})$-critical path; otherwise, reassigning $\beta_{1}$ to $uu_{4}$ and $\kappa- \Delta + 1$ to $uu_{2}$ will take us back to Case~\ref{1Common}. Hence, $\mathcal{U}(w_{1}) \cap (\mathcal{U}(w) \cup \mathcal{U}(u)) = \{1, \kappa - \Delta, \kappa- \Delta + 1\}$. Moreover, there exists a $(\kappa - \Delta + 1, 2, u, w_{1})$-critical path; otherwise, reassigning  $\beta_{1}, 2$ and $\kappa - \Delta$ to $uu_{2}$, $uw_{1}$ and $uw$ respectively, results in an acyclic edge coloring of $G$. It is obvious that $2\in \mathcal{U}(u_{4})$. If $\kappa - \Delta \notin \mathcal{U}(u_{4})$, then reassigning $\kappa - \Delta, \beta_{1}, 2$ and $\xi_{1}$ to $uu_{4}, uu_{2}, uw_{1}$ and $uw$ respectively, results in an acyclic edge coloring of $G$. Thus, $\kappa - \Delta \in \mathcal{U}(u_{4})$. Recall that $\{1, 2\} \subseteq \mathcal{U}(u_{4})$. If there exists a color $\tau$ in $\mathcal{U}(w)\setminus \mathcal{U}(u_{4})$, then reassigning $\tau$, $\xi_{1}$ and $\beta_{1}$ to $uu_{4}, uu_{3}$ and $uw$ respectively, will result in an acyclic edge coloring of $G$. Hence, $\mathcal{U}(w) \subseteq \mathcal{U}(u_{4})$. Then
\[
 \sum_{\theta \in \mathcal{U}(w) \cup \mathcal{U}(u)} \mul_{\mathbb{S}}(\theta) + |X| \geq |\mathcal{U}(w)| + 2|\{1, \kappa - \Delta, \kappa - \Delta + 1\}| = \deg_{G}(w) + 5.
\]

Secondly, suppose that $\mathcal{A}(u_{2}) \cup \mathcal{A}(u_{3}) \subseteq \mathcal{U}(u_{4})$. Thus, every color in $\mathcal{A}(u_{2}) \cup \mathcal{A}(u_{3})$ appears three times in $\mathbb{S}$, and then $\mathcal{A}(u_{2}) \cup \mathcal{A}(u_{3}) \subseteq X$. Recall that $1 \in \mathcal{U}(u_{2}) \cup \mathcal{U}(u_{3})$. Since $\mathcal{A}(u_{2}) \cup \mathcal{A}(u_{3}) \subseteq \mathcal{U}(w_{1})$, it follows that $|\{\kappa - \Delta, \kappa - \Delta + 1\} \cap \mathcal{U}(u_{2})| = 1$ or $|\{\kappa - \Delta, \kappa - \Delta + 1\} \cap \mathcal{U}(u_{3})| = 1$.

Suppose that $1 \in \Upsilon(uu_{2}) \cap \Upsilon(uu_{3})$. It follows that $\deg_{G}(w) = 6$ and $\mathcal{U}(w_{1}) = \{1, \kappa - \Delta\} \cup \mathcal{A}(u_{2}) \cup \mathcal{A}(u_{3})$. Moreover, we have that $|\{\kappa - \Delta, \kappa - \Delta + 1\} \cap \Upsilon(uu_{2})| = 1$ and $|\{\kappa - \Delta, \kappa - \Delta + 1\} \cap \Upsilon(uu_{3})| = 1$. If $\kappa - \Delta + 1 \notin \Upsilon(uu_{2})$, then reassigning $\beta_{1}, 2$ and $\xi_{1}$ to $uu_{2}, uw_{1}$ and $wu$ respectively results in an acyclic edge coloring of $G$. So we have that $\kappa - \Delta + 1 \in \Upsilon(uu_{2})$. Similarly, we have that $\kappa - \Delta + 1\in \Upsilon(uu_{3})$. But reassigning $\alpha_{1}$ to $uw_{1}$ and $\kappa - \Delta$ to $wu$ results in an acyclic edge coloring of $G$.

So we may assume that $1 \notin \Upsilon(uu_{2}) \cap \Upsilon(uu_{3})$ and $1 \in \Upsilon(uu_{2})$. If there is a $(2, 1, u, u_{3})$-critical path, then $\mathcal{U}(w_{1}) = \{1, \kappa - \Delta\} \cup \mathcal{A}(u_{2}) \cup \mathcal{A}(u_{3})$ and $2 \in \mathcal{U}(u_{3})$, but reassigning $\alpha_{1}$ to $ww_{1}$ and $1$ to $wu$ results in an acyclic edge coloring of $G$. Hence, there exists no $(2, 1, u, u_{3})$-critical path. Thus, there exists a $(\kappa - \Delta + 1, 1, u, u_{3})$-critical path, otherwise, reassigning $1$ to $uu_{3}$ will take us back to Subcase~\ref{2Common2}. This implies that the color $1$ appears at least three times in $\mathbb{S}$.

Suppose that $3 \notin \mathbb{S}$. Thus, there exists a $(2, \kappa - \Delta + 1, w, u)$-critical path; otherwise, reassigning $3, 1$ and $\kappa - \Delta + 1$ to $uu_{4}, uu_{3}$ and $uw$ respectively, results in an acyclic edge coloring of $G$. Hence, $\kappa - \Delta + 1 \in \mathcal{U}(u_{2}) \cap \mathcal{U}(u_{3})$. If $\kappa - \Delta \notin \mathcal{U}(u_{3})$, then reassigning $3, \xi_{1}$ and $\beta_{1}$ to $uu_{4}, uu_{3}$ and $uw$ respectively, results in an acyclic edge coloring of $G$. So we may assume that $\kappa - \Delta \in \mathcal{U}(u_{3})$. Hence, $|\mathcal{A}(u_{2})| = |\mathcal{A}(u_{3})| = 2$ and $\mathcal{U}(w_{1}) = \{1, \kappa - \Delta\} \cup \mathcal{A}(u_{2}) \cup \mathcal{U}(u_{3})$. Now, reassigning $3, 1, \beta_{1}$ and $\alpha_{1}$ to $uu_{4}, uu_{3}, uw$ and $ww_{1}$, results in an acyclic edge coloring of $G$. Therefore, we can conclude that $3 \in \mathbb{S}$.

Suppose that $4 \notin \mathbb{S}$. Thus, there is a $(4, \beta_{1}, w, u_{3})$-alternating path; otherwise, reassigning $4$ to $uu_{3}$ and $\beta_{1}$ to $uw$ will result in an acyclic edge coloring of $G$. Similarly, there exists a $(4, \xi_{1}, w, u_{2})$-alternating path. Moreover, there exists a $(\kappa - \Delta, \xi_{1}, u, u_{3})$-critical path; otherwise, reassigning $4$, $\xi_{1}$ and $\beta_{1}$ to $uu_{4}$, $uu_{3}$ and $uw$ respectively, results in an acyclic edge coloring of $G$. Thus, $\mathcal{U}(u_{3}) \cap (\mathcal{U}(w) \cup \mathcal{U}(u)) = \{3, \kappa - \Delta, \kappa - \Delta + 1\}$, $|\mathcal{A}(u_{3})| = 2$ and $|\mathcal{A}(u_{2})| = 2$. Hence, $|\mathcal{U}(u_{2}) \cap \{\kappa - \Delta, \kappa - \Delta + 1\}| = 1$. If $\kappa - \Delta \notin \mathcal{U}(u_{2})$, then reassigning $4, \beta_{1}$ and $\xi_{1}$ to $uu_{4}, uu_{2}$ and $uw$ respectively, results in an acyclic edge coloring of $G$. Hence, we have that $\mathcal{U}(u_{2}) \cap (\mathcal{U}(w) \cup \mathcal(u)) = \{1, 2, \kappa - \Delta\}$. But reassigning $\xi_{1}, 4$ and $\beta_{1}$ to $uw, uw_{1}$ and $uu_{2}$ respectively, results in an acyclic edge coloring of $G$. So, $4 \in \mathbb{S}$. Similarly, we have that $\mathcal{U}(w) \setminus \{1, 2, 3\} \subseteq \mathbb{S}$.

Recall that $|\mathcal{A}(u_{2}) \cup \mathcal{A}(u_{3})| \geq 3$, $|\{\kappa - \Delta, \kappa - \Delta + 1\} \cap \mathcal{U}(u_{2})| \geq 1$ and $|\{\kappa - \Delta, \kappa - \Delta + 1\} \cap \mathcal{U}(u_{3})| \geq 1$. Hence,

\[
\sum_{\theta \in \mathcal{U}(w) \cup \mathcal{U}(u)} \mul_{\mathbb{S}}(\theta) + |X| \geq |\{3, 4, \dots, \deg_{G}(w) - 1\}| + 3|\{1\}| + 1 + 1 + |\mathcal{A}(u_{2}) \cup \mathcal{A}(u_{3})|\geq \deg_{G}(w) + 5.
\]

\begin{subcase}\label{3Common}%
$\mathcal{U}(w) \cap \mathcal{U}(u) = \{\lambda_{1}, \lambda_{2}, \lambda_{3}\}$ and $w_{1} = u_{1}$. Note that $|C(wu)| \geq \Delta$.
\end{subcase}

\begin{subsubcase}\label{3Common1}%
The color on $uw_{1}$ is a common color.
\end{subsubcase}

By symmetry, assume that $\phi(uw_{1}) = \lambda_{1}$, $\phi(uu_{2}) = \lambda_{2}$, $\phi(uu_{3}) = \lambda_{3}$ and $\phi(uu_{4}) = \kappa - \Delta$.

If $\Upsilon(uu_{2}) \subseteq C(wu)$, then reassigning $\beta_{1}$ to $uu_{2}$ will take us back to Subcase~\ref{2Common}. So we have that $\Upsilon(uu_{2}) \nsubseteq C(wu)$ and $|\mathcal{U}(u_{2}) \cap C(wu)| \leq \Delta - 2$; similarly, we also have that $|\mathcal{U}(u_{3}) \cap C(wu)| \leq \Delta - 2$. If $\mathcal{U}(w_{1}) \cap (\mathcal{U}(w) \cup \mathcal{U}(u)) = \{1, \lambda_{1}\}$, then reassigning $\alpha_{1}$ to $uw_{1}$ will take us back to Subcase~\ref{2Common} again. Hence, $|\mathcal{U}(w_{1}) \cap (\mathcal{U}(w) \cup \mathcal{U}(u))| \geq 3$. By \autoref{CapEmpty}, we have $\mathcal{A}(u_{2}) \cup \mathcal{A}(u_{3}) \subseteq \mathcal{U}(w_{1})$ and $\mathcal{A}(u_{2}) \cap \mathcal{A}(u_{3}) = \emptyset$. Further, we have that $|\mathcal{U}(w_{1})| \geq 3 + |\mathcal{A}(u_{2})| + |\mathcal{A}(u_{3})| > \deg_{G}(w_{1})$, which is a contradiction.

\begin{subsubcase}\label{3Common2}%
The color on $uw_{1}$ is not a common color, but the color on $ww_{1}$ is a common color.
\end{subsubcase}

By symmetry, assume that $\phi(uw_{1}) = \kappa - \Delta$, $\phi(uu_{2}) = 2$, $\phi(uu_{3}) = 3$ and $\phi(uu_{4}) = 1$.

If $\Upsilon(uu_{2}) \subseteq C(wu)$, then reassigning $\beta_{1}$ to $uu_{2}$ will take us back to Subcase~\ref{2Common}. So we have that $\Upsilon(uu_{2}) \nsubseteq C(wu)$ and $|\mathcal{U}(u_{2}) \cap C(wu)| \leq \Delta - 2$; similarly, we also have that $|\mathcal{U}(u_{3}) \cap C(wu)| \leq \Delta - 2$ and $|\mathcal{U}(u_{4}) \cap C(wu)| \leq \Delta - 2$.

If $\mathcal{U}(w_{1}) \cap (\mathcal{U}(w) \cup \mathcal{U}(u)) = \{1, \kappa - \Delta\}$, then reassigning $\alpha_{1}$ to $ww_{1}$ will take us back to Subcase~\ref{2Common} again. Hence, $|\mathcal{U}(w_{1}) \cap (\mathcal{U}(w) \cup \mathcal{U}(u))| \geq 3$.

Furthermore, we have that $\mathcal{A}(u_{2}) \cup \mathcal{A}(u_{3}) \nsubseteq \mathcal{U}(w_{1})$. Otherwise, if $\deg_{G}(w) \leq \kappa - \Delta$, then $|\mathcal{U}(w_{1})| \geq 3 + |\mathcal{A}(u_{2})| + |\mathcal{A}(u_{3})| \geq 3 + 2 + 2 > 6$; and if $\deg_{G}(w) \leq \kappa - \Delta - 1$, then $|\mathcal{U}(w_{1})| \geq 3 + |\mathcal{A}(u_{2})| + |\mathcal{A}(u_{3})| \geq 3 + 3 + 3 > 7$. Without loss of generality, assume that $\beta_{1} \notin \mathcal{U}(w_{1})$. Since $\beta_{1} \notin \mathcal{U}(w_{1}) \cup \mathcal{U}(u_{2})$, it follows that there exists a $(3, \beta_{1}, u, w)$-critical path. There exists a $(1, \beta_{1}, u, u_{2})$-critical path; otherwise, reassigning $\beta_{1}$ to $uu_{2}$ will take us back to Subcase~\ref{2Common}. Hence, $1 \in \mathcal{U}(u_{2})$ and $1$ appears at least twice in $\mathbb{S}$.

There exists a $(2, \kappa - \Delta, u, w)$- or $(3, \kappa - \Delta, u, w)$-critical path; otherwise, reassigning $\kappa - \Delta$ to $uw$ and $\beta_{1}$ to $uw_{1}$ will result in an acyclic edge coloring of $G$. If $\kappa - \Delta$ appears only once in $\mathbb{S}$, then reassigning $\beta_{1}$ to $uw_{1}$ and $\kappa - \Delta$ to $uu_{4}$ will take us back to Subcase~\ref{2Common}. Hence, the color $\kappa - \Delta$ appears at least twice in $\mathbb{S}$.

Let $t \in \mathcal{U}(w) \setminus \{1, 3\}$. If $t \notin \mathbb{S}$, then reassigning $\beta_{1}$ to $uu_{2}$ and $t$ to $uu_{4}$ will take us back to Subcase~\ref{2Common}. Hence, we have that $\mathcal{U}(w) \setminus \{1, 3\} \subseteq \mathbb{S}$.

If $\mathcal{A}(u_{3}) \subseteq \mathcal{U}(w_{1})$, then every color in $\mathcal{A}(u_{3})$ appears precisely three times in $\mathbb{S}$, and then
\[
\sum_{\theta \in \mathcal{U}(w) \cup \mathcal{U}(u)} \mul_{\mathbb{S}}(\theta) + |X| \geq |\{2, 4, 5, \dots, \deg_{G}(w) - 1\}| + 2|\{1, \kappa - \Delta\}| + |\mathcal{A}(u_{3})| \geq \deg_{G}(w) + 3.
\]

So we may assume that $\mathcal{A}(u_{3}) \nsubseteq \mathcal{U}(w_{1})$ and $\xi_{1} \notin \mathcal{U}(w_{1}) \cup \mathcal{U}(u_{3})$. Similar to above, we can prove that there exists a $(2, \xi_{1}, w, u)$- and $(1, \xi_{1}, u, u_{3})$-critical path, and then $1$ appears precisely three times in $\mathbb{S}$. If $3 \notin \mathbb{S}$, then reassigning $3$ to $uu_{4}$ and $\xi_{1}$ to $uu_{3}$ will take us back to Subcase~\ref{2Common}. Thus, the color $3$ appears at least once in $\mathbb{S}$. Therefore, we have

\[
\sum_{\theta \in \mathcal{U}(w) \cup \mathcal{U}(u)} \mul_{\mathbb{S}}(\theta) + |X| \geq |\{2, 3, \dots, \deg_{G}(w) - 1\}| + 2 |\{\kappa - \Delta \}| + 3|\{1\}| = \deg_{G}(w) + 3.
\]

\begin{subsubcase}\label{3Common3}%
Neither the color on $w_{1}w$ nor the color on $w_{1}u$ is a common color.
\end{subsubcase}
By symmetry, assume that $\phi(uw_{1}) = \kappa - \Delta$, $\phi(uu_{2}) = 2$, $\phi(uu_{3}) = 3$ and $\phi(uu_{4}) = 4$.

If $\Upsilon(uu_{2}) \subseteq C(wu)$, then reassigning $\beta_{1}$ to $uu_{2}$ will take us back to Subcase~\ref{2Common}. So we have that $\Upsilon(uu_{2}) \nsubseteq C(wu)$ and $|\mathcal{U}(u_{2}) \cap C(wu)| \leq \Delta - 2$; similarly, we also have that $|\mathcal{U}(u_{3}) \cap C(wu)| \leq \Delta - 2$ and $|\mathcal{U}(u_{4}) \cap C(wu)| \leq \Delta - 2$.

Suppose that the color $1$ appears at most twice in $\mathbb{S}$; by symmetry, assume that $1 \notin \mathcal{U}(u_{3}) \cup \mathcal{U}(u_{4})$. Thus there exists a $(2, 1, u, u_{4})$-critical path; otherwise, reassigning $1$ to $uu_{4}$ will take us back to Subcase~\ref{3Common2}. But reassigning $1$ to $uu_{3}$ will take us back to Subcase~\ref{3Common2} again. Hence, the color $1$ appears at least three times in $\mathbb{S}$.

Furthermore, $\mathcal{A}(u_{2}) \cup \mathcal{A}(u_{3}) \cup \mathcal{A}(u_{4}) \nsubseteq \mathcal{U}(w_{1})$; otherwise, we have $|\mathcal{U}(w_{1})| \geq 2 + |\mathcal{A}(u_{2})| + |\mathcal{A}(u_{3})| + |\mathcal{A}(u_{4})| > \deg_{G}(w_{1})$, which is a contradiction. Without loss of generality, assume that $\beta_{1} \notin \mathcal{U}(w_{1})$. Clearly, there exists a $(3, \beta_{1}, u, w)$- or $(4, \beta_{1}, u, w)$-critical path. By symmetry, assume that there exists a $(3, \beta_{1}, u, w)$-critical path. There exists a $(4, \beta_{1}, u, u_{2})$-critical path; otherwise, reassigning $\beta_{1}$ to $uu_{2}$ will take us back to Subcase~\ref{2Common}. It follows that $4 \in \mathcal{U}(u_{2})$.

If $2 \notin \mathbb{S}$, then reassigning $2$ to $uu_{4}$ and $\beta_{1}$ to $uu_{2}$ will take us back to Subcase~\ref{2Common}. So we have $2 \in \mathbb{S}$; similarly, we can obtain that $\mathcal{U}(w) \setminus \{1, 3, 4\} \subseteq \mathbb{S}$.

If $3 \notin \mathbb{S}$, then $4 \in \mathcal{U}(w_{1}) \cup \mathcal{U}(u_{3})$; otherwise, reassigning $3, 4$ and $\beta_{1}$ to $uu_{4}, uu_{3}$ and $uu_{2}$ respectively, and then we go back to Subcase~\ref{2Common}. Anyway, we have that $\mul_{\mathbb{S}}(3) + \mul_{\mathbb{S}}(4) \geq 2$.

There exists a $(2, \kappa - \Delta, u, w)$- or $(3, \kappa - \Delta, u, w)$- or $(4, \kappa - \Delta, u, w)$-critical path; otherwise, reassigning $\beta_{1}$ to $uw_{1}$ and $\kappa - \Delta$ to $uw$ results in an acyclic edge coloring of $G$. If $\kappa - \Delta \notin \mathcal{U}(u_{3}) \cup \mathcal{U}(u_{4})$, then reassigning $\kappa - \Delta$ to $uu_{3}$ and $\beta_{1}$ to $uw_{1}$ will take us back to Case \ref{2Common}. This implies that $\kappa - \Delta \in \mathcal{U}(u_{3}) \cup \mathcal{U}(u_{4})$; similarly, we can prove that $\kappa - \Delta \in \mathcal{U}(u_{2}) \cup \mathcal{U}(u_{4})$ and $\kappa - \Delta \in \mathcal{U}(u_{2}) \cup \mathcal{U}(u_{3})$. Hence, the color $\kappa - \Delta$ appears at least twice in $\mathbb{S}$. Therefore, we have

\[
\sum_{\theta \in \mathcal{U}(w) \cup \mathcal{U}(u)} \mul_{\mathbb{S}}(\theta) + |X| \geq 3|\{1\}| + 2|\{\kappa - \Delta\}| + \sum_{\theta \in \mathcal{U}(w) \setminus \{1\}}\mul_{\mathbb{S}}(\theta) \geq \deg_{G}(w) + 3.
\]

\begin{subcase}%
$|\mathcal{U}(w) \cap \mathcal{U}(u)| = 4$.
\end{subcase}
In other words, $\mathcal{U}(u) \subseteq \mathcal{U}(w)$. It follows that $|C(wu)| = \kappa - \deg_{G}(w) + 1 \geq \Delta + 1$ and $|\mathcal{A}(u_{i})| \geq 2$ for $i = 2, 3, 4$. By \autoref{CapEmpty}, we have that $\mathcal{A}(u_{2}), \mathcal{A}(u_{3})$ and $\mathcal{A}(u_{4})$ are pairwise disjoint and $\mathcal{U}(w_{1}) \supseteq \mathcal{A}(u_{2}) \cup \mathcal{A}(u_{3}) \cup \mathcal{A}(u_{4})$, which implies that $|\mathcal{U}(w_{1})| \geq 2 + |\mathcal{A}(u_{2})| + |\mathcal{A}(u_{3})| + |\mathcal{A}(u_{4})| > \deg_{G}(w_{1})$, a contradiction.
\resetcounter
\end{proof}

\section{The main result}
Now, we are ready to prove the main result, \autoref{MResult}.
\begin{figure}[htbp]%
\centering
\subcaptionbox{$(3, 9^{+}, 9^{+})$-face\label{fig:subfig:a}}{\includegraphics{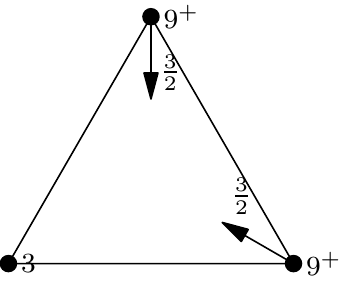}}\hfill~
\subcaptionbox{$(4, 4, 10^{+})$-face\label{fig:subfig:b}}{\includegraphics{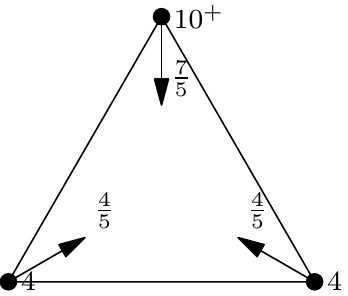}}\hfill~
\subcaptionbox{$(4, 5, 11)$-face\label{fig:subfig:c}}{\includegraphics{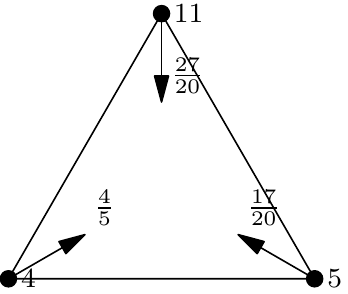}}\hfill~
\subcaptionbox{$(4, 5, \neq 11)$-face\label{fig:subfig:d}}{\includegraphics{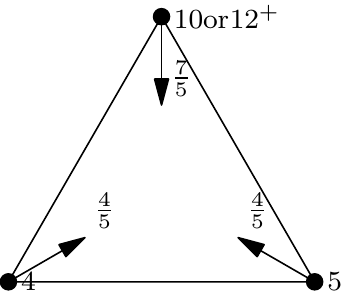}}
\label{fig:contfig:one}
\end{figure}

\begin{figure}[htbp]%
\ContinuedFloat
\centering
\subcaptionbox{$(4, 6, 10^{+})$-face\label{fig:subfig:e}}{\includegraphics{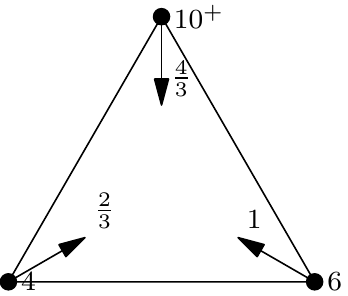}}\hfill~
\subcaptionbox{$(4, 7\mbox{--}8, 7\mbox{--}9)$-face\label{fig:subfig:f}}{\includegraphics{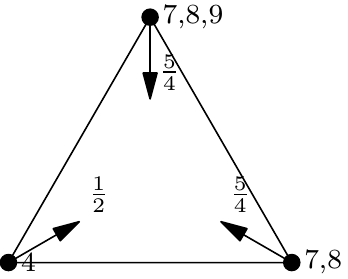}}\hfill~
\subcaptionbox{$(4, 7\mbox{--}8, 10^{+})$-face\label{fig:subfig:g}}{\includegraphics{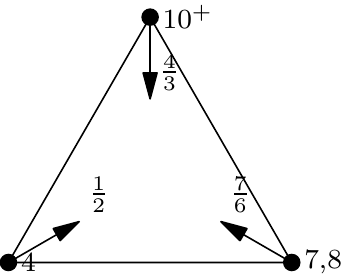}}\hfill~
\subcaptionbox{$(4, 9^{+}, 9^{+})$-face\label{fig:subfig:i}}{\includegraphics{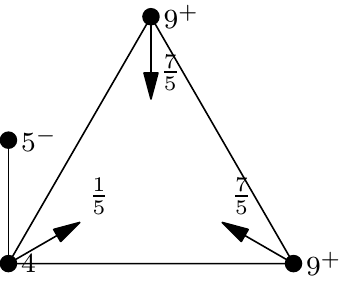}}\hfill~
\subcaptionbox{$(4, 9^{+}, 9^{+})$-face\label{fig:subfig:j}}{\includegraphics{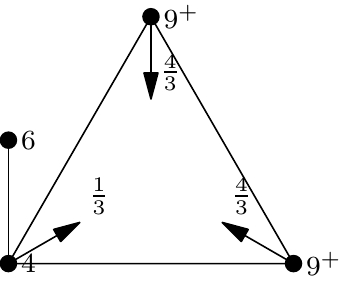}}\hfill~
\subcaptionbox{$(4, 9^{+}, 9^{+})$-face\label{fig:subfig:k}}{\includegraphics{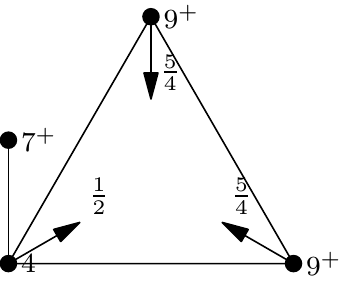}}\hfill~
\subcaptionbox{$(5, 5, 5\mbox{--}7)$-face\label{fig:subfig:l}}{\includegraphics{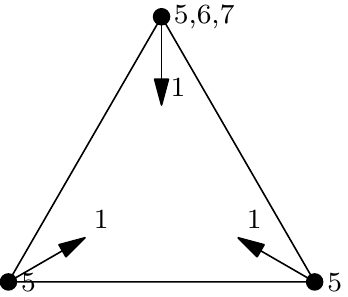}}
\label{fig:contfig:two}
\end{figure}

\begin{figure}[htbp]%
\ContinuedFloat
\centering
\subcaptionbox{$(5, 5, 8^{+})$-face\label{fig:subfig:m}}{\includegraphics{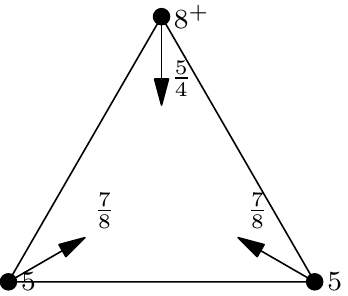}}\hfill~
\subcaptionbox{$(5, 6, 6)$-face\label{fig:subfig:n}}{\includegraphics{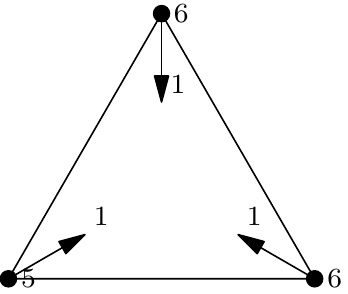}}\hfill~
\subcaptionbox{$(5, 6, 7)$-face\label{fig:subfig:o}}{\includegraphics{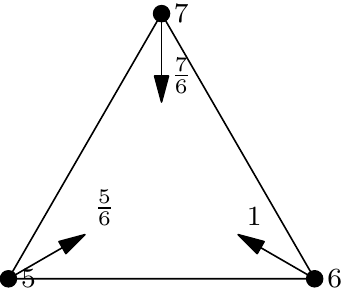}}\hfill~
\subcaptionbox{$(5, 6, 8^{+})$-face\label{fig:subfig:p}}{\includegraphics{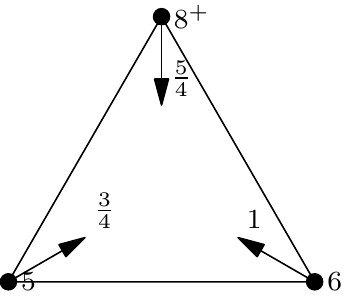}}\hfill~
\subcaptionbox{$(5, 7, 7)$-face\label{fig:subfig:q}}{\includegraphics{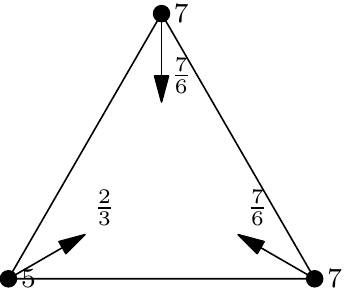}}\hfill~
\subcaptionbox{$(5, 7, 8^{+})$-face\label{fig:subfig:r}}{\includegraphics{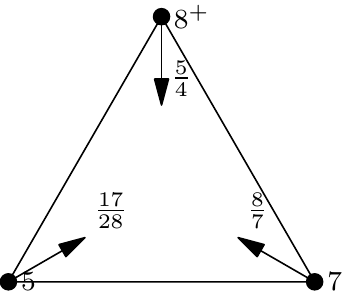}}\hfill~
\subcaptionbox{$(5, 8^{+}, 8^{+})$-face\label{fig:subfig:s}}{\includegraphics{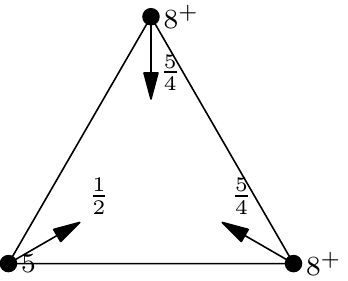}}\hfill~
\subcaptionbox{$(6^{+}, 6^{+}, 6^{+})$-face\label{fig:subfig:t}}{\includegraphics{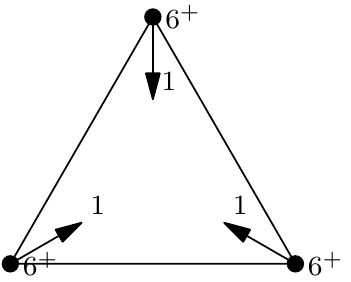}}
\caption{Discharging rules}
\label{fig:contfig:three}
\end{figure}

\begin{proof}[Proof of \autoref{MResult}]%
Suppose that $G$ is a counterexample with $|V| + |E|$ is minimum, and fix $\kappa = \Delta(G) + 6$. Since the hypothesis is minor-closed, it follows that $G$ is a $\kappa$-minimal graph. Let $G^{*}$ be obtained from $G$ by removing all the $2$-vertices. By \autoref{delta2} and \autoref{2+edge}, the minimum degree of $G^{*}$ is at least three. Take a component $H$ of $G^{*}$ and embed it in the plane. In the following, we will do arguments on the graph $H$ to obtain a contradiction.

By \autoref{2+edge}~\ref{A}, we have the following claims.
\begin{claim}\label{8+H}%
If $\deg_{H}(v) < \deg_{G}(v)$, then $\deg_{H}(v) \geq 8 + m$, where $m$ is the number of adjacent $7^{-}$-vertices in $H$.
\end{claim}

\begin{claim}\label{8-H}%
If $\deg_{H}(v) \leq 7$, then $\deg_{G}(v) = \deg_{H}(v)$.
\end{claim}

From the Euler's formula, we have the following equality:
\begin{equation}%
\sum_{v \in V(H)} (2\deg_{H}(v) - 6) + \sum_{f \in F(H)} (\deg_{H}(f) - 6) = - 12
\end{equation}

Assign the initial charge of every vertex $v$ to be $2\deg_{H}(v) - 6$ and the initial charge of every face $f$ to be $\deg_{H}(f) - 6$. Clearly, the sum of the initial charge of vertices and faces is $-12$. We design appropriate discharging rules and redistribute charge among the vertices and faces, such that the final charge of every vertex and every face is nonnegative, which derive a contradiction.

\bigskip
{\bf Discharging Rules:}
\begin{enumerate}[label = (R\arabic*)]%
\item If $w$ is a $4$-vertex adjacent to a $5^{-}$-vertex $u$, then $w$ sends $\frac{4}{5}$ to each face incident with $wu$, and sends $\frac{1}{5}$ to each other face.
\item If $w$ is a $4$-vertex adjacent to a $6$-vertex $u$, then $w$ sends $\frac{2}{3}$ to each face incident with $wu$, and sends $\frac{1}{3}$ to each other face.
\item If $w$ is a $4$-vertex which is not adjacent to $6^{-}$-vertices, then $w$ sends $\frac{1}{2}$ to each incident face.
\item All the rules regarding $3$-faces are in the Fig~\subref{fig:subfig:a}--\subref{fig:subfig:t}.
\item Every $9^{+}$-vertex sends $1$ to each incident $4^{+}$-face.
\item Every vertex with degree $5, 6, 7$ or $8$ sends $\frac{1}{2}$ to each incident $4^{+}$-face.
\end{enumerate}

{\noindent\em Computing the final charge of faces.}
\bigskip

Let $f = w_{1}w_{2}w_{3}$ be a $3$-face with $\deg_{H}(w_{1}) \leq \deg_{H}(w_{2}) \leq \deg_{H}(w_{3})$.

If $w_{1}$ is a $3$-vertex, then \autoref{3++edge} implies that both $w_{2}$ and $w_{3}$ are $9^{+}$-vertices in $G$, and they also are $9^{+}$-vertices in $H$ by \autoref{8+H}, thus $f$ is a $(3, 9^{+}, 9^{+})$-face in $H$ and the final charge is $-3 + 2 \times \frac{3}{2} = 0$.

If $w_{1}w_{2}$ is a $(4, 4)$-edge, then \autoref{4Sum} implies that $w_{3}$ is a $12^{+}$-vertex in $G$, and it is a $10^{+}$-vertex in $H$ by \autoref{8+H}, thus $f$ is a $(4, 4, 10^{+})$-face and the final charge is $- 3 + 2 \times \frac{4}{5} + \frac{7}{5} = 0$.

If $w_{1}w_{2}$ is a $(4, 5)$-edge, then \autoref{4Sum} implies that $w_{3}$ is a $11^{+}$-vertex in $G$, and it is a $10^{+}$-vertex in $H$ by \autoref{8+H}, thus the final charge of $f$ is $-3 + \frac{4}{5} + \frac{17}{20} + \frac{27}{20} = 0$ if $\deg_{H}(w_{3}) = 11$, or $- 3 + 2 \times \frac{4}{5} + \frac{7}{5} = 0$ if $w_{3}$ is a $10$- or $12^{+}$-vertex in $H$.

If $w_{1}w_{2}$ is a $(4, 6)$-edge, then \autoref{4Sum} implies that $w_{3}$ is a $10^{+}$-vertex in $G$, and it is a $10^{+}$-vertex in $H$ by \autoref{8+H}, and then the final charge is $- 3 + \frac{2}{3} + 1 + \frac{4}{3} = 0$.

If $\deg_{H}(w_{1}) = 4$, $\deg_{H}(w_{2}) \in \{7, 8\}$ and $\deg_{H}(w_{3}) \in \{7, 8, 9\}$, then the final charge of $f$ is $-3 + \frac{1}{2} + 2 \times \frac{5}{4} = 0$.

If $\deg_{H}(w_{1}) = 4$, $\deg_{H}(w_{2}) \in \{7, 8\}$ and $\deg_{H}(w_{3}) \geq 10$, then the final charge of $f$ is $-3 + \frac{1}{2} + \frac{7}{6} + \frac{4}{3} = 0$. 

Suppose that $f$ is a $(4, 9^{+}, 9^{+})$-face. If $w_{1}$ is adjacent to a $5^{-}$-vertex $u$, then $w_{1}$ sends $\frac{1}{5}$ to $f$, and then the final charge of $f$ is $-3 + \frac{1}{5} + 2 \times \frac{7}{5} = 0$; if $w_{1}$ is adjacent to a $6$-vertex $u$, then $w_{1}$ sends $\frac{1}{3}$ to $f$, and then the final charge of $f$ is $-3 + \frac{1}{3} + 2 \times \frac{4}{3} = 0$; if $w_{1}$ is not adjacent to $6^{-}$-vertices, then $w_{1}$ sends $\frac{1}{2}$ to $f$, and then the final charge of $f$ is $-3 + \frac{1}{2} + 2 \times \frac{5}{4} = 0$.

If $\deg_{H}(w_{1}) = \deg_{H}(w_{2}) = 5$ and $\deg_{H}(w_{3}) \in \{5, 6, 7\}$, then the final charge of $f$ is $-3 + 3 \times 1 = 0$.

If $f$ is a $(5, 5, 8^{+})$-face, then the final charge is $-3 + 2 \times \frac{7}{8} + \frac{5}{4} = 0$.

If $f$ is a $(5, 6, 6)$-face, then the final charge is $-3 + 3 \times 1 = 0$.

If $f$ is a $(5, 6, 7)$-face, then the final charge is $-3 + \frac{5}{6} + 1 + \frac{7}{6} = 0$.

If $f$ is a $(5, 6, 8^{+})$-face, then the final charge is $-3 + \frac{3}{4} + 1 + \frac{5}{4} = 0$.

If $f$ is a $(5, 7, 7)$-face, then the final charge is $-3 + \frac{2}{3} + 2 \times \frac{7}{6} = 0$.

If $f$ is a $(5, 7, 8^{+})$-face, then the final charge is $-3 + \frac{17}{28} + \frac{8}{7} + \frac{5}{4} = 0$.

If $f$ is a $(5, 8^{+}, 8^{+})$-face, then the final charge is $-3 + \frac{1}{2} + 2 \times \frac{5}{4} = 0$.

If $f$ is a $(6^{+}, 6^{+}, 6^{+})$-face, then the final charge is $-3 + 3 \times 1 = 0$.

Next, we compute the final charge of $4$-faces. Let $w_{1}w_{2}w_{3}w_{4}$ be a $4$-face with $w_{2}$ having the minimum degree on the boundary. If $\deg_{H}(w_{2}) \geq 5$, then the final charge of $f$ is at least $-2 + 4 \times \frac{1}{2} = 0$. If $\deg_{H}(w_{1}), \deg_{H}(w_{3}) \geq 9$, then the final charge is at least $-2 + 2 \times 1 = 0$. So we may assume that $\deg_{H}(w_{2}) \in \{3, 4\}$ and $\deg_{H}(w_{1}) \leq 8$. By \autoref{3++edge} and \autoref{8+H}, we have that $\deg_{H}(w_{2}) = 4$ and $\deg_{G}(w_{1}) = \deg_{H}(w_{1}) \leq 8$. By \autoref{4Sum} and discharging rules, the face $f$ receives at least $\frac{1}{2}$ from each incident vertex, so the final charge of $f$ is at least $-2 + 4 \times \frac{1}{2} = 0$.

Suppose that $f$ is a $5$-face. If $f$ is incident with a $9^{+}$-vertex, then the final charge is at least $-1 + 1 = 0$. So we may assume that $f$ is incident with five $8^{-}$-vertices. It is obvious that $f$ is incident with at least two $5^{+}$-vertices, and then the final charge is at least $-1 + 2 \times \frac{1}{2} = 0$.

If $f$ is a $6^{+}$-face, then the final charge is at least $\deg_{H}(f) - 6 \geq 0$. 

\bigskip
{\noindent\em Computing the final charge of vertices.}

\paragraph{\indent Let $v$ be a $3$-vertex.} Clearly, the final charge is zero.

\paragraph{\indent Let $v$ be a $4$-vertex.} If $v$ is adjacent to a $5^{-}$-vertex, then \autoref{4Sum} and \autoref{8+H} implies that $v$ is adjacent to three $9^{+}$-vertices, and then the final charge is $2 - 2 \times \frac{4}{5} - 2 \times \frac{1}{5} = 0$. If $v$ is adjacent to a $6$-vertex, then \autoref{4Sum} and \autoref{8+H} implies that $v$ is adjacent to three $9^{+}$-vertices, and then the final charge is $2 - 2 \times \frac{2}{3} - 2 \times \frac{1}{3} = 0$. If $v$ is not adjacent to $6^{-}$-vertices, then the final charge is $2 - 4 \times \frac{1}{2} = 0$.

\paragraph{\indent Let $v$ be a $5$-vertex with neighbors $v_{1}, v_{2}, \dots, v_{5}$ in anticlockwise order.}  If $v$ sends at most $\frac{4}{5}$ to each incident face, then the final charge is at least $4 - 5 \times \frac{4}{5} = 0$. So we may assume that $v$ sends more than $\frac{4}{5}$ to some face $f$.

If $f$ is a $(5, 5, 5)$-face, then \autoref{5Sum} and \autoref{8+H} implies that the other three vertices adjacent to $v$ are $9^{+}$-vertices, and then the final charge of $v$ is at least $4 - 1 - 2 \times \frac{7}{8} - 2 \times \frac{1}{2} > 0$.

If $f$ is a $(5, 5, 6)$-face, then \autoref{5Sum} and \autoref{8+H} implies that the other three vertices adjacent to $v$ are $8^{+}$-vertices, and then the final charge of $v$ is at least $4 - 1 - \frac{7}{8} - \frac{3}{4} - 2 \times \frac{1}{2} > 0$.

If $f$ is a $(5, 5, 7)$-face, then \autoref{5Sum} and \autoref{8+H} implies that the other three vertices adjacent to $v$ are $7^{+}$-vertices, and then the final charge of $v$ is at least $4 - 2 \times 1 - 3 \times \frac{2}{3} = 0$.

If $f$ is a $(5, 6, 6)$-face, then \autoref{5Sum} and \autoref{8+H} implies that the other three vertices adjacent to $v$ are $7^{+}$-vertices, and then the final charge of $v$ is at least $4 - 1 - 2 \times \frac{5}{6} - 2 \times \frac{2}{3} = 0$.

If $v$ sends at most $\frac{1}{2}$ to an incident face, then the final charge of $v$ is at least $4 - 4 \times \frac{7}{8} - \frac{1}{2} = 0$. So we may assume that the $5$-vertex $v$ sends more than $\frac{1}{2}$ to each incident face, thus $v$ is incident with five $3$-faces.

Suppose that $f = vv_{1}v_{2}$ is a $3$-face with $\deg_{H}(v_{1}) = 5$ and $\deg_{H}(v_{2}) \geq 8$. By the excluded cases in the above, the vertex $v_{5}$ is an $8^{+}$-vertex. Since $v$ sends more than $\frac{1}{2}$ to the $3$-face $vv_{2}v_{3}$, the vertex $v_{3}$ is a $7^{-}$-vertex. Similarly, the vertex $v_{4}$ is also a $7^{-}$-vertex. Now, the $3$-face $vv_{3}v_{4}$ is a $(5, 7^{-}, 7^{-})$-face. By the excluded cases, we only have to consider the edge $v_{3}v_{4}$ is a $(6, 7)$- or $(7, 6)$- or $(7, 7)$-edge. If $v_{3}v_{4}$ is a $(7, 7)$-edge, then the final charge of $v$ is at least $4 - 2 \times \frac{7}{8} - 2 \times \frac{17}{28} - \frac{2}{3} > 0$. If $v_{3}v_{4}$ is $(6, 7)$- or $(7, 6)$-edge, then the final charge of $v$ is at least $4 - 2 \times \frac{7}{8} - \frac{3}{4} - \frac{5}{6} - \frac{17}{28} > 0$.

Suppose that $f = vv_{1}v_{2}$ is a $(5, 6, 7)$-face with $\deg_{H}(v_{1}) = 6$ and $\deg_{H}(v_{2}) = 7$. By the excluded cases, the vertex $v_{3}$ is a $6^{+}$-vertex and the vertex $v_{5}$ is a $7^{+}$-vertex. By \autoref{3++edge} and \autoref{8+H}, the vertex $v_{4}$ is a $4^{+}$-vertex. If $\deg_{H}(v_{4}) = 4$, then \autoref{4Sum} and \autoref{8+H} implies that both $v_{3}$ and $v_{5}$ are $11^{+}$-vertices, thus the final charge of $v$ is at least $4 - \frac{5}{6} - \frac{3}{4} - \frac{17}{28} - 2 \times \frac{17}{20} > 0$. By the excluded cases, the vertex $v_{4}$ cannot be a $5$-vertex. If $\deg_{H}(v_{4}) = 6$, then $\deg_{H}(v_{3}) \geq 7$, and then the final charge of $v$ is at least $4 - \frac{2}{3} - 4 \times \frac{5}{6} = 0$. If $\deg_{H}(v_{4}) \geq 7$, then the final charge is at least $4 - \frac{2}{3} - 4 \times \frac{5}{6} = 0$.

Suppose that $f = vv_{1}v_{2}$ is a $(5, 4, 11)$-face. By \autoref{4Sum} and \autoref{8+H}, the vertex $v_{5}$ is a $10^{+}$-vertex. If one of $v_{3}$ and $v_{4}$ is a $8^{+}$-vertex, then $v$ sends $\frac{1}{2}$ to an incident $3$-face, a contradiction. So we may assume that $\deg_{H}(v_{3}), \deg_{H}(v_{4}) \leq 7$. By the excluded cases, the edge $v_{3}v_{4}$ is a $(7, 7)$-edge, and then the final charge of $v$ is at least $4 - 2 \times \frac{17}{28} - \frac{2}{3} - 2 \times \frac{17}{20} > 0$.

\paragraph{\indent Let $v$ be a $6$-vertex.} The final charge is at least $6 - 6 \times 1 = 0$.

\paragraph{\indent Let $v$ be a $7$-vertex.} If $v$ sends at most $\frac{1}{2}$ to an incident face, then the final charge is at least $8 - 6 \times \frac{5}{4} - \frac{1}{2} = 0$. So we may assume that $v$ sends more than $\frac{1}{2}$ to each incident face, thus $v$ is incident with seven $3$-faces. By \autoref{4Sum}~(b) and \autoref{8+H}, the vertex $v$ is not incident with $(4, 7, 9^{-})$-faces. Now, the vertex $v$ sends at most $\frac{7}{6}$ to each incident face. If $v$ is incident with a $(5^{-}, 5^{-}, 7)$- or $(6^{+}, 6^{+}, 7)$-face, then the final charge is at least $8 - 6 \times \frac{7}{6} - 1 = 0$. So every face incident with $v$ is a $(5^{-}, 6^{+}, 7)$-face, but the vertex $v$ is a $7$-vertex and the number $7$ is odd, a contradiction.

\paragraph{\indent Let $v$ be an $8$-vertex.} Every $8$-vertex sends at most $\frac{5}{4}$ to each incident face, thus the final charge is at least $10 - 8 \times \frac{5}{4} = 0$.

\paragraph{\indent Let $v$ be a $9$-vertex.}

If $\deg_{G}(v) > 9$, then \autoref{8+H} implies that $v$ is adjacent to at most one $7^{-}$-vertex in $H$, and then the final charge of $v$ is at least $12 - 7 \times 1 - 2 \times \frac{3}{2} > 0$. So we may assume that $\deg_{G}(v) = \deg_{H}(v) = 9$.

Suppose that $(3, 9)$-edge $uv$ is incident with two $3$-faces. By \autoref{L9}, the vertex $v$ is adjacent to eight $8^{+}$-vertices, and then the final charge is at least $12 - 7 \times 1 - 2 \times \frac{3}{2} > 0$. So every $(3, 9)$-edge $uv$ is incident with at most one $3$-face.

Let $\tau$ be the number of incident $4^{+}$-faces. If $\tau \geq 4$, then the final charge is at least $12 - 5 \times \frac{3}{2} - 4 \times 1 > 0$. Since $\deg_{G}(v) = \deg_{H}(v) = 9$, \autoref{4Sum} implies that $v$ is not incident with face~\subref{fig:subfig:i} or \subref{fig:subfig:j}. If $\tau \leq 3$, then the final charge is at least $12 - \tau - 2\tau \times \frac{3}{2} - (9 - 3\tau) \times \frac{5}{4} \geq 0$.

\paragraph{\indent Let $v$ be a $10$-vertex.} If $\deg_{G}(v) > 10$, then \autoref{8+H} implies that $v$ is adjacent to at most two $7^{-}$-vertices, and then the final charge is at least $14 - 4 \times \frac{3}{2} - 6 \times 1 > 0$. So we may assume that $\deg_{G}(v) = \deg_{H}(v) = 10$. Hence, the vertex $v$ is not incident with face~\subref{fig:subfig:b}, \subref{fig:subfig:d} or \subref{fig:subfig:i}, and thus $v$ sends $\frac{3}{2}, \frac{4}{3}, \frac{5}{4}$ or $1$ to each incident face.

If $v$ is incident with at least two $4^{+}$-faces, then the final charge is at least $14 - 8 \times \frac{3}{2} - 2 \times 1 = 0$. Hence, the vertex $v$ is incident with at most one $4^{+}$-face. \autoref{4Sum} implies that $v$ is adjacent to at most five $4^{-}$-vertices. Let $s$ be the number of incident $(10, 3, 9^{+})$-faces, and let $s^{*}$ be the number of incident $(10, 4, 6^{+})$-faces.

If $s \leq 4$, then the final charge is at least $14 - s \times \frac{3}{2} - (10 - s) \times \frac{4}{3} = \frac{2}{3} - \frac{s}{6} \geq 0$. So we may assume that $s \geq 5$, and then the number of adjacent $3$-vertices is at least three.

\begin{enumerate}[label = (\arabic*)]%
\item $s \in \{5, 6\}$.

If $s^{*} = 0$, then the final charge is at least $14 - 6 \times \frac{3}{2} - 4 \times \frac{5}{4}  = 0$. If $v$ is incident with exactly one $4^{+}$-face, then the final charge is at least $14 - 6 \times \frac{3}{2} - 1 - 3 \times \frac{4}{3} = 0$. So we may assume that $s^{*} \geq 1$ and $v$ is not incident with any $4^{+}$-face. Clearly, the vertex $v$ is incident with exactly six $(10, 3, 9^{+})$-faces and $s = 6$. It is obvious that $v$ is adjacent to at least one $4$-vertex. \autoref{3-10vertex} implies that the vertex $v$ is adjacent to exactly three $3$-vertices, one $4$-vertex and six $6^{+}$-vertices. Hence, it is incident with exactly two $(10, 6^{+}, 6^{+})$-faces, and then the final charge is at least $14 - 6 \times \frac{3}{2} - 2 \times \frac{4}{3} - 2 \times 1 > 0$.
\item $s \geq 7$.

 Clearly, the vertex $v$ is adjacent to at least four $3$-vertices. \autoref{3-10vertex} implies that the vertex $v$ is adjacent to exactly four $3$-vertices and six $6^{+}$-vertices. Hence, the vertex $v$ is incident with two $(10, 6^{+}, 6^{+})$-faces, or one $(10, 6^{+}, 6^{+})$-face and one $4^{+}$-face, thus the final charge is at least $14 - 8 \times \frac{3}{2} - 2 \times 1 = 0$.
\end{enumerate}

\paragraph{\indent Let $v$ be an $11$-vertex.}

If $\deg_{G}(v) > 11$, then $v$ is adjacent to at most three $7^{-}$-vertices in $H$, and then the final charge is at least $16 - 6 \times \frac{3}{2} - 5 \times 1 > 0$. So we may assume that $\deg_{G}(v) = \deg_{H}(v) = 11$.

If $v$ sends at most $1$ to an incident face, then the final charge is at least $16 - 10 \times \frac{3}{2} - 1 = 0$. So we may assume that $v$ is not incident with $4^{+}$-faces and is not incident with $(11, 6^{+}, 6^{+})$-faces. Since the degree of $v$ is odd, the vertex $v$ cannot be incident with eleven $(11, 5^{-}, 6^{+})$-faces. So $v$ is incident with a $(11, 5^{-}, 5^{-})$-face $f$. \autoref{3++edge} and \autoref{4Sum} implies that the face $f$ is a $(4, 5, 11)$-face or $(5, 5, 11)$-face. Hence, the vertex $v$ is adjacent to at most four $3$-vertices. If $v$ is adjacent to at most three $3$-vertices, then the final charge is at least $16 - 6 \times \frac{3}{2} - 5 \times \frac{7}{5} = 0$. Hence, the vertex $v$ is adjacent to exactly four $3$-vertices, see \autoref{W11}. If $f$ is a $(5, 5, 11)$-face, then the final charge of $v$ is $16 - 8 \times \frac{3}{2} - 3 \times \frac{5}{4} > 0$. If $f$ is a $(4, 5, 11)$-face, then the final charge is $16 - 8 \times \frac{3}{2} - \frac{5}{4} - \frac{27}{20} - \frac{7}{5} = 0$.

\begin{figure}%
\centering
\includegraphics{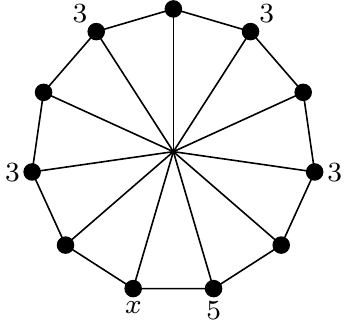}
\caption{The vertex $x$ is a $4$- or $5$-vertex.}
\label{W11}
\end{figure}

\paragraph{\indent Let $v$ be a $12^{+}$-vertex.} The final charge is at least $2 \deg_{H}(v) - 6 - \deg_{H}(v) \times \frac{3}{2} = \frac{1}{2} \deg_{H}(v) - 6 \geq 0$.
\end{proof}

\vskip 0mm \vspace{0.3cm} \noindent{\bf Acknowledgments.} This project was supported by the National Natural Science Foundation of China (11101125) and partially supported by the Fundamental Research Funds for Universities in Henan. The authors would like to thank the referees for their valuable comments.


\begin{thebibliography}{10}

\bibitem{MR1109695}
N.~Alon, C.~McDiarmid and B.~Reed, Acyclic coloring of graphs, Random
  Structures Algorithms 2~(3) (1991) 277--288.

\bibitem{MR1837021}
N.~Alon, B.~Sudakov and A.~Zaks, Acyclic edge colorings of graphs, J. Graph
  Theory 37~(3) (2001) 157--167.

\bibitem{MR2817509}
M.~Basavaraju, L.~S. Chandran, N.~Cohen, F.~Havet and T.~M{\"u}ller, Acyclic
  edge-coloring of planar graphs, SIAM J. Discrete Math. 25~(2) (2011)
  463--478.

\bibitem{MR3037985}
L.~Esperet and A.~Parreau, Acyclic edge-coloring using entropy compression,
  European J. Combin. 34~(6) (2013) 1019--1027.

\bibitem{MR526851}
I.~Fiam{\v{c}}{\'{\i}}k, The acyclic chromatic class of a graph, Math. Slovaca
  28~(2) (1978) 139--145.

\bibitem{MR2458434}
A.~Fiedorowicz, M.~Ha{\l}uszczak and N.~Narayanan, About acyclic edge
  colourings of planar graphs, Inform. Process. Lett. 108~(6) (2008) 412--417.

\bibitem{2014arXiv1407.5374G}
I.~Giotis, L.~Kirousis, K.~I. Psaromiligkos and D.~M. Thilikos, On the
  algorithmic {L}ov\'asz local lemma and acyclic edge coloring, eprint
  arXiv:1407.5374.

\bibitem{MR3031510}
Y.~Guan, J.~Hou and Y.~Yang, An improved bound on acyclic chromatic index of
  planar graphs, Discrete Math. 313~(10) (2013) 1098--1103.

\bibitem{MR2849391}
J.~Hou, N.~Roussel and J.~Wu, Acyclic chromatic index of planar graphs with
  triangles, Inform. Process. Lett. 111~(17) (2011) 836--840.

\bibitem{MR2491777}
J.~Hou, J.~Wu, G.~Liu and B.~Liu, Acyclic edge colorings of planar graphs and
  series-parallel graphs, Sci. China Ser. A 52~(3) (2009) 605--616.

\bibitem{MR1715600}
M.~Molloy and B.~Reed, Further algorithmic aspects of the local lemma, in:
  Proceedings of the Thirtieth Annual ACM Symposium on the Theory of Computing,
  ACM, New York, 1998, pp. 524--529.

\bibitem{MR2864444}
S.~Ndreca, A.~Procacci and B.~Scoppola, Improved bounds on coloring of graphs,
  European J. Combin. 33~(4) (2012) 592--609.

\bibitem{MR3166127}
T.~Wang and Y.~Zhang, Acyclic edge coloring of graphs, Discrete Appl. Math. 167
  (2014) 290--303.

\bibitem{MR2994403}
W.~Wang, Q.~Shu and Y.~Wang, A new upper bound on the acyclic chromatic indices
  of planar graphs, European J. Combin. 34~(2) (2013) 338--354.

\end{thebibliography}
\end{document}